\documentclass{bmcart}

\usepackage{amsthm,amsmath,amssymb}
\usepackage[utf8]{inputenc} 
\usepackage{graphicx}
\usepackage{siunitx}

\usepackage{mathabx}
\usepackage{hyperref}

\usepackage{mathtools}
\usepackage{booktabs,dcolumn}

\startlocaldefs

\newcommand*\tupref[2]{\href{http://math.mit.edu/~primegaps/tuples/admissible_#1_#2.txt}{\num{#2}}}

\DeclareMathOperator*\EH{EH}
\DeclareMathOperator*\GEH{GEH}
\DeclareMathOperator*\MPZ{MPZ}
\DeclareMathOperator*\DHL{DHL}

\DeclareFontFamily{OT1}{rsfs}{}
\DeclareFontShape{OT1}{rsfs}{n}{it}{<-> rsfs10}{}
\DeclareMathAlphabet{\mathscr}{OT1}{rsfs}{n}{it}

\theoremstyle{plain}

\newtheorem{theorem}{Theorem}[section]
\newtheorem{proposition}[theorem]{Proposition}
\newtheorem{lemma}[theorem]{Lemma}
\newtheorem{corollary}[theorem]{Corollary}

\newtheorem{claim}[theorem]{Claim}

\theoremstyle{definition}

\newtheorem{definition}[theorem]{Definition}
\newtheorem{remark}[theorem]{Remark}

\newcommand\E{\mathbb{E}}
\newcommand\R{\mathbb{R}}
\newcommand\Z{\mathbb{Z}}
\newcommand\N{\mathbb{N}}
\newcommand\C{\mathbb{C}}

\newcommand\eps{\varepsilon}

\renewcommand\Re{\operatorname{Re}}

\newcommand\Scal{\mathcal{S}}

\renewcommand\P{\mathbb{P}}

\renewcommand{\sim}{\asymp} 

\renewcommand{\lessapprox}{\llcurly}

\renewcommand{\phi}{\varphi}
\newcommand{\onef}{\mathbf{1}}

\endlocaldefs

\begin{document}

\begin{frontmatter}

\begin{fmbox}
\dochead{Research}


\title{Variants of the Selberg sieve, and bounded intervals containing many primes}


\author[
   addressref={aff1},                   
   corref={aff1},                       
   email={tao@math.ucla.edu}   
]{\inits{DHJP}\fnm{DHJ} \snm{Polymath}}


\address[id=aff1]{
  \street{\url{michaelnielsen.org/polymath1/index.php?title=Bounded\_gaps\_between\_primes}},                     %
}



\end{fmbox}


\begin{abstractbox}

\begin{abstract} 
For any $m \geq 1$, let $H_m$ denote the quantity $\liminf_{n \to \infty} (p_{n+m}-p_n)$, where $p_n$ is the $n^{\operatorname{th}}$ prime. A celebrated recent result of Zhang showed the finiteness of $H_1$, with the explicit bound $H_1 \leq 70000000$. This was then improved by us (the Polymath8 project) to $H_1 \leq 4680$, and then by Maynard to $H_1 \leq 600$, who also established for the first time a finiteness result for $H_m$ for $m \geq 2$, and specifically that $H_m \ll m^3 e^{4m}$. If one also assumes the Elliott-Halberstam conjecture, Maynard obtained the bound $H_1 \leq 12$, improving upon the previous bound $H_1 \leq 16$ of Goldston, Pintz, and Y{\i}ld{\i}r{\i}m, as well as the bound $H_m \ll m^3 e^{2m}$.

  In this paper, we extend the methods of Maynard by generalizing the Selberg sieve further, and by performing more extensive numerical calculations. As a consequence, we can obtain the bound $H_1 \leq 246$ unconditionally, and $H_1 \leq 6$ under the assumption of the generalized Elliott-Halberstam conjecture. Indeed, under the latter conjecture we show the stronger statement that for any admissible triple $(h_1,h_2,h_3)$, there are infinitely many $n$ for which at least two of $n+h_1,n+h_2,n+h_3$ are prime, and also obtain a related disjunction asserting that either the twin prime conjecture holds, or the even Goldbach conjecture is asymptotically true if one allows an additive error of at most $2$, or both. We also modify the ``parity problem'' argument of Selberg to show that the $H_1 \leq 6$ bound is the best possible that one can obtain from purely sieve-theoretic considerations. For larger $m$, we use the distributional results obtained previously by our project to obtain the unconditional asymptotic bound $H_m \ll m e^{(4-\frac{28}{157})m}$, or $H_m \ll m e^{2m}$ under the assumption of the Elliott-Halberstam conjecture. We also obtain explicit upper bounds for $H_m$ when $m=2,3,4,5$.
	\end{abstract}



\end{abstractbox}
%

\end{frontmatter}





\section{Introduction}

For any natural number $m$, let $H_m$ denote the quantity
$$ H_m \coloneqq \liminf_{n \to \infty} (p_{n+m} - p_n),$$
where $p_n$ denotes the $n^{\operatorname{th}}$ prime.  The twin prime conjecture asserts that $H_1=2$; more generally, the Hardy-Littlewood prime tuples conjecture \cite{hardy} implies that $H_m = H(m+1)$ for all $m \geq 1$, where $H(k)$ is the diameter of the narrowest admissible $k$-tuple (see Section \ref{subclaim-sec} for a definition of this term).  Asymptotically, one has the bounds
$$ (\frac{1}{2}+o(1)) k \log k \leq H(k) \leq (1+o(1)) k \log k$$
as $k \to \infty$ (see Theorem \ref{hk-bound} below); thus the prime tuples conjecture implies that $H_m$ is comparable to $m \log m$ as $m \to \infty$.

Until very recently, it was not known if any of the $H_m$ were finite, even in the easiest case $m=1$.  In the breakthrough work of Goldston, Pintz, and Y{\i}ld{\i}r{\i}m \cite{gpy}, several results in this direction were established, including the following conditional result assuming the Elliott-Halberstam conjecture $\EH[\vartheta]$ (see Claim \ref{eh-def} below) concerning the distribution of the prime numbers in arithmetic progressions:

\begin{theorem}[GPY theorem]\label{gpy-thm}  Assume the Elliott-Halberstam conjecture $\EH[\vartheta]$ for all $0 < \vartheta < 1$.  Then $H_1 \leq 16$.
\end{theorem}

Furthermore, it was shown in \cite{gpy} that any result of the form $\EH[\frac{1}{2} + 2\varpi]$ for some fixed $0 < \varpi < 1/4$ would imply an explicit finite upper bound on $H_1$ (with this bound equal to $16$ for $\varpi > 0.229855$).  Unfortunately, the only results of the type $\EH[\vartheta]$ that are known come from the Bombieri-Vinogradov theorem (Theorem \ref{bv-thm}), which only establishes $\EH[\vartheta]$ for $0 < \vartheta < 1/2$.

The first unconditional bound on $H_1$ was established  in a breakthrough work of Zhang \cite{zhang}:

\begin{theorem}[Zhang's theorem]\label{zhang-thm}  $H_1 \leq \num{70000000}$.
\end{theorem}

Zhang's argument followed the general strategy from \cite{gpy} on finding small gaps between primes, with the major new ingredient being a proof of a weaker version of $\EH[\frac{1}{2}+2\varpi]$, which we call $\MPZ[\varpi,\delta]$; see Claim \ref{mpz-claim} below.  It was quickly realized that Zhang's numerical bound on $H_1$ could be improved.  By optimizing many of the components in Zhang's argument, we were able \cite{polymath8a, polymath8a-unabridged} to improve Zhang's bound to
$$ H_1 \leq \num{4680}.$$

Very shortly afterwards, a further breakthrough was obtained by Maynard \cite{maynard-new} (with related work obtained independently in unpublished work of Tao), who developed a more flexible ``multidimensional'' version of the Selberg sieve to obtain stronger bounds on $H_m$. This argument worked without using any equidistribution results on primes beyond the Bombieri-Vinogradov theorem, and amongst other things was able to establish finiteness of $H_m$ for all $m$, not just for $m=1$.  More precisely, Maynard established the following results.

\begin{theorem}[Maynard's theorem]  Unconditionally, we have the following bounds:
\begin{itemize}
\item[(i)] $H_1 \leq 600$.
\item[(ii)] $H_m \leq C m^3 e^{4m}$ for all $m \geq 1$ and an absolute (and effective) constant $C$.
\end{itemize}
Assuming the Elliott-Halberstam conjecture $\EH[\vartheta]$ for all $0 < \vartheta < 1$, we have the following improvements:
\begin{itemize}
\item[(iii)] $H_1 \leq 12$.
\item[(iv)] $H_2 \leq 600$.
\item[(v)] $H_m \leq C m^3 e^{2m}$ for all $m \geq 1$ and an absolute (and effective) constant $C$.
\end{itemize}
\end{theorem}

For a survey of these recent developments, see \cite{granville}.  

In this paper, we refine Maynard's methods to obtain the following further improvements.

\begin{theorem}\label{main} Unconditionally, we have the following bounds:
\begin{itemize}
\item[(i)] $H_1 \leq 246$.   
\item[(ii)] $H_2 \leq \num{398130}$. 
\item[(iii)] $H_3 \leq \num{24797814}$. 
\item[(iv)] $H_4 \leq \num{1431556072}$.  
\item[(v)] $H_5 \leq \num{80550202480}$. 
\item[(vi)] $H_m \leq C m \exp( (4 - \frac{28}{157}) m )$ for all $m \geq 1$ and an absolute (and effective) constant $C$.
\end{itemize}
Assume the Elliott-Halberstam conjecture $\EH[\vartheta]$ for all $0 < \vartheta < 1$.  Then we have the following improvements:
\begin{itemize}
\item[(vii)] $H_2 \leq 270$.
\item[(viii)] $H_3 \leq \num{52116}$.
\item[(ix)] $H_4 \leq \num{474266}$.
\item[(x)] $H_5 \leq \num{4137854}$.
\item[(xi)] $H_m \leq Cme^{2m}$ for all $m \geq 1$ and an absolute (and effective) constant $C$.
\end{itemize}
Finally, assume the generalized Elliott-Halberstam conjecture $\GEH[\vartheta]$ (see Claim \ref{geh-def} below) for all $0 < \vartheta < 1$.  Then
\begin{itemize}
\item[(xii)] $H_1 \leq 6$.
\item[(xiii)] $H_2 \leq 252$.
\end{itemize}
\end{theorem}

In Section \ref{subclaim-sec} we will describe the key propositions that will be combined together to prove the various components of Theorem \ref{main}.
As with Theorem \ref{gpy-thm}, the results in (vii)-(xiii) do not require $\EH[\vartheta]$ or $\GEH[\vartheta]$ for all $0 < \vartheta < 1$, but only for a single explicitly computable $\vartheta$ that is sufficiently close to $1$.  

Of these results, the bound in (xii) is perhaps the most interesting, as the parity problem \cite{selberg} prohibits one from achieving any better bound on $H_1$ than $6$ from purely sieve-theoretic methods; we review this obstruction in Section \ref{parity-sec}.  If one only assumes the Elliott-Halberstam conjecture $\EH[\vartheta]$ instead of its generalization $\GEH[\vartheta]$, we were unable to improve upon Maynard's bound $H_1 \leq 12$; however the parity obstruction does not exclude the possibility that one could achieve (xii) just assuming $\EH[\vartheta]$ rather than $\GEH[\vartheta]$, by some further refinement of the sieve-theoretic arguments (e.g. by finding a way to establish Theorem \ref{nonprime-asym}(ii) below using only $\EH[\vartheta]$ instead of $\GEH[\vartheta]$).

The bounds (ii)-(vi) rely on the equidistribution results on primes established in our previous paper \cite{polymath8a}.  However, the bound (i) uses only the Bombieri-Vinogradov theorem, and the remaining bounds (vii)-(xiii) of course use either the Elliott-Halberstam conjecture or a generalization thereof.

A variant of the proof of Theorem \ref{main}(xii), which we give in Section \ref{remarks-sec}, also gives the following conditional ``near miss'' to (a disjunction of) the twin prime conjecture and the even Goldbach conjecture:

\begin{theorem}[Disjunction]\label{disj} Assume the generalized Elliott-Halberstam conjecture $\GEH[\vartheta]$ for all $0 < \vartheta < 1$.  Then at least one of the following statements is true:
\begin{itemize}
\item[(a)]  (Twin prime conjecture) $H_1=2$.
\item[(b)]  (near-miss to even Goldbach conjecture)  If $n$ is a sufficiently large multiple of six, then at least one of $n$ and $n-2$ is expressible as the sum of two primes.  Similarly with $n-2$ replaced by $n+2$.  (In particular, every sufficiently large even number lies within $2$ of the sum of two primes.)
\end{itemize}
\end{theorem}

We remark that a disjunction in a similar spirit was obtained in \cite{bgc}, which established (prior to the appearance of Theorem \ref{zhang-thm}) that either $H_1$ was finite, or that every interval $[x,x+x^\eps]$ contained the sum of two primes if $x$ was sufficiently large depending on $\eps>0$.  

There are two main technical innovations in this paper.  The first is a further generalization of the multidimensional Selberg sieve introduced by Maynard and Tao, in which the support of a certain cutoff function $F$ is permitted to extend into a larger domain than was previously permitted (particularly under the assumption of the generalized Elliott-Halberstam conjecture).  As in \cite{maynard-new}, this largely reduces the task of bounding $H_m$ to that of efficiently solving a certain multidimensional variational problem involving the cutoff function $F$.  Our second main technical innovation is to obtain efficient numerical methods for solving this variational problem for small values of the dimension $k$, as well as sharpened asymptotics in the case of large values of $k$.

The methods of Maynard and Tao have been used in a number of subsequent applications \cite{freiberg}, \cite{consecutive}, \cite{thorner}, \cite{benatar}, \cite{li-pan}, \cite{castillo}, \cite{pollack}, \cite{banks}, \cite{clusters}, \cite{lola}, \cite{pintz-ratio}, \cite{chua}, \cite{pintz-new}.  The techniques in this paper should be able to be used to obtain slight numerical improvements to such results, although we did not pursue these matters here.

\subsection{Organization of the paper}

The paper is organized as follows.  After some notational preliminaries, we recall in Section \ref{dist-sec} the known (or conjectured) distributional estimates on primes in arithmetic progressions that we will need to prove Theorem \ref{main}.  Then, in Section \ref{subclaim-sec}, we give the key propositions that will be combined together to establish this theorem.  One of these propositions, Lemma \ref{crit}, is an easy application of the pigeonhole principle.  Two further propositions, Theorem \ref{prime-asym} and Theorem \ref{nonprime-asym}, use the prime distribution results from Section \ref{dist-sec} to give asymptotics for certain sums involving sieve weights and the von Mangoldt function; they are established in Section \ref{sieving-sec}.  Theorems \ref{maynard-thm}, \ref{maynard-trunc}, \ref{epsilon-trick}, \ref{epsilon-beyond} use the asymptotics established in Theorems \ref{prime-asym}, \ref{nonprime-asym}, in combination with Lemma \ref{crit}, to give various criteria for bounding $H_m$, which all involve finding sufficiently strong candidates for a variety of multidimensional variational problems; these theorems are proven in Section \ref{variational-sec}.  These variational problems are analysed in the asymptotic regime of large $k$ in Section \ref{asymptotics-sec}, and for small and medium $k$ in Section \ref{h1-sec}, with the results collected in Theorems \ref{mlower}, \ref{mlower-var}, \ref{mke-lower}, \ref{piece}.  Combining these results with the previous propositions gives Theorem \ref{main-dhl}, which, when combined with the bounds on narrow admissible tuples in Theorem \ref{hk-bound} that are established in Section \ref{tuples-sec}, will give Theorem \ref{main}.  (See also Table \ref{ingredients} for some more details of the logical dependencies between the key propositions.)

Finally, in Section \ref{parity-sec} we modify an argument of Selberg to show that the bound $H_1 \leq 6$ may not be improved using purely sieve-theoretic methods, and in Section \ref{remarks-sec} we establish Theorem \ref{disj} and make some miscellaneous remarks.

\subsection{Notation}

The notation used here closely follows the notation in our previous paper \cite{polymath8a}.

We use $|E|$ to denote the cardinality of a finite set $E$, and $\onef_E$ to denote the indicator function of a set $E$, thus $\onef_E(n)=1$ when $n \in E$ and $\onef_E(n)=0$ otherwise.  In a similar spirit, if $E$ is a statement, we write $\onef_E=1$ when $E$ is true and $\onef_E=0$ otherwise. 

All sums and products will be over the natural numbers $\N \coloneqq \{1,2,3,\ldots\}$ unless otherwise specified, with the exceptions of sums and products over the variable $p$, which will be understood to be over primes.

The following important asymptotic notation will be in use throughout the paper:

\begin{definition}[Asymptotic notation]\label{asym}  We use $x$ to denote a large real parameter, which one should think of as going off to infinity; in particular, we will implicitly assume that it is larger than any specified fixed constant. Some mathematical objects will be independent of $x$ and referred to as \emph{fixed}; but unless otherwise specified we allow all mathematical objects under consideration to depend on $x$ (or to vary within a range that depends on $x$, e.g. the summation parameter $n$ in the sum $\sum_{x \leq n \leq 2x} f(n)$).  If $X$ and $Y$ are two quantities depending on $x$, we say that $X = O(Y)$ or $X \ll Y$ if one has $|X| \leq CY$ for some fixed $C$ (which we refer to as the \emph{implied constant}), and $X = o(Y)$ if one has $|X| \leq c(x) Y$ for some function $c(x)$ of $x$ (and of any fixed parameters present) that goes to zero as $x \to \infty$ (for each choice of fixed parameters).  We use $X \lessapprox Y$ to denote the estimate $|X| \leq x^{o(1)} Y$, $X \sim Y$ to denote the estimate $Y \ll X \ll Y$, and $X \approx Y$ to denote the estimate $Y \lessapprox X \lessapprox Y$.  Finally, we say that a quantity $n$ is of \emph{polynomial size} if one has $n = O(x^{O(1)})$.

If asymptotic notation such as $O()$ or $\lessapprox$ appears on the left-hand side of a statement, this means that the assertion holds true for any specific interpretation of that notation.  For instance, the assertion $\sum_{n=O(N)} |\alpha(n)| \lessapprox N$ means that for each fixed constant $C>0$, one has $\sum_{|n| \leq CN} |\alpha(n)|\lessapprox N$.
\end{definition}

If $q$ and $a$ are integers, we write $a|q$ if $a$ divides $q$.
If $q$ is a natural number and $a \in \Z$, we use $a\ (q)$ to denote
the residue class
$$ a\ (q) \coloneqq \{ a+nq: n \in \Z \}$$
and let $\Z/q\Z$ denote the ring of all such residue classes $a\
(q)$. The notation $b=a\ (q)$ is synonymous to $b \in \, a \ (q)$. We
use $(a,q)$ to denote the greatest common divisor of $a$ and $q$, and
$[a,q]$ to denote the least common
multiple.\footnote{When $a,b$ are real numbers, we will also
  need to use $(a,b)$ and $[a,b]$ to denote the open and closed
  intervals respectively with endpoints $a,b$.  Unfortunately, this
  notation conflicts with the notation given above, but it should be
  clear from the context which notation is in use.}
We also let
$$ (\Z/q\Z)^\times \coloneqq \{ a\ (q): (a,q)=1 \}$$
denote the primitive residue classes of $\Z/q\Z$.  

We use the following standard arithmetic functions:
\begin{itemize}
\item[(i)] $\phi(q) \coloneqq |(\Z/q\Z)^\times|$ denotes the Euler totient function of $q$.
\item[(ii)] $\tau(q) \coloneqq \sum_{d|q} 1$ denotes the divisor function of $q$.
\item[(iii)] $\Lambda(q)$ denotes the von Mangoldt function of $q$, thus $\Lambda(q)=\log p$ if $q$ is a power of a prime $p$, and $\Lambda(q)=0$ otherwise.
\item[(iv)] $\theta(q)$ is defined to equal $\log q$ when $q$ is a prime, and $\theta(q)=0$ otherwise.
\item[(v)] $\mu(q)$ denotes the M\"obius function of $q$, thus $\mu(q) = (-1)^k$ if $q$ is the product of $k$ distinct primes for some $k \geq 0$, and $\mu(q)=0$ otherwise.
\item[(vi)] $\Omega(q)$ denotes the number of prime factors of $q$ (counting multiplicity).
\end{itemize}

We recall the elementary \emph{divisor bound}
\begin{equation}\label{divisor-bound}
\tau(n) \lessapprox 1
\end{equation}
whenever $n \ll x^{O(1)}$, as well as the related estimate
\begin{equation}\label{divisor-2}
\sum_{n \ll x} \frac{\tau(n)^C}{n} \ll \log^{O(1)} x
\end{equation}
for any fixed $C>0$; this follows for instance from \cite[Lemma 1.3]{polymath8a}.

The \emph{Dirichlet convolution} $\alpha \star \beta \colon \N \to \C$
of two arithmetic functions $\alpha,\beta \colon \N \to \C$ is defined
in the usual fashion as
$$ \alpha\star\beta(n) \coloneqq \sum_{d|n} \alpha(d)
\beta\left(\frac{n}{d}\right) =\sum_{ab=n}{\alpha(a)\beta(b)}.$$

\section{Distribution estimates on arithmetic functions}\label{dist-sec}

As mentioned in the introduction, a key ingredient in the Goldston-Pintz-Y{\i}ld{\i}r{\i}m approach to small gaps between primes comes from distributional estimates on the primes, or more precisely on the von Mangoldt function $\Lambda$, which serves as a proxy for the primes.  In this work, we will also need to consider distributional estimates on more general arithmetic functions, although we will not prove any new such estimates in this paper, relying instead on estimates that are already in the literature.  

More precisely, we will need averaged information on the following quantity:

\begin{definition}[Discrepancy]
For any function $\alpha \colon \N \to \C$ with finite support (that is, $\alpha$ is non-zero only on a finite set) and any
primitive residue class $a\ (q)$, we define the (signed)
\emph{discrepancy} $\Delta(\alpha; a\ (q))$ to be the quantity
\begin{equation}\label{disc-def}
  \Delta(\alpha; a\ (q)) \coloneqq \sum_{n = a\ (q)} \alpha(n) - 
  \frac{1}{\phi(q)} \sum_{(n,q)=1} \alpha(n).
\end{equation}
\end{definition}

For any fixed $0 < \vartheta < 1$, let $\EH[\vartheta]$ denote the following claim:

\begin{claim}[Elliott-Halberstam conjecture, {$\EH[\vartheta]$}]\label{eh-def}  If $Q \lessapprox x^\vartheta$ and $A \geq 1$ is fixed, then
\begin{equation}\label{qq}
 \sum_{q \leq Q} \sup_{a \in (\Z/q\Z)^\times} |\Delta(\Lambda \onef_{[x,2x]}; a\ (q))| \ll x \log^{-A} x.
\end{equation} 
\end{claim}

In \cite{elliott} it was conjectured that $\EH[\vartheta]$ held for all $0 < \vartheta < 1$.  (The conjecture fails at the endpoint case $\vartheta=1$; see \cite{fg}, \cite{fghm} for a more precise statement.)  The following classical result of Bombieri \cite{bombieri} and Vinogradov \cite{vinogradov} remains the best partial result of the form $\EH[\vartheta]$:

\begin{theorem}[Bombieri-Vinogradov theorem]\label{bv-thm}\cite{bombieri,vinogradov}  $\EH[\vartheta]$ holds for every fixed $0 < \vartheta < 1/2$.
\end{theorem}

In \cite{gpy} it was shown that any estimate of the form $\EH[\vartheta]$ with some fixed $\vartheta > 1/2$ would imply the finiteness of $H_1$.  While such an estimate remains unproven, it was observed by Motohashi-Pintz \cite{mp} and by Zhang \cite{zhang} that a certain weakened version of $\EH[\vartheta]$ would still suffice for this purpose.  More precisely (and following the notation of our previous paper \cite{polymath8a}), let $\varpi,\delta > 0$ be fixed, and let $\MPZ[\varpi,\delta]$ be the following claim:

\begin{claim}[Motohashi-Pintz-Zhang estimate,
  {$\MPZ[\varpi,\delta]$}]\label{mpz-claim} Let $I \subset
  [1,x^\delta]$ and $Q \lessapprox x^{1/2+2\varpi}$.  Let $P_I$ denote the product of all the primes in $I$, and let $\Scal_I$ denote the square-free natural numbers whose prime factors lie in $I$.  If the residue class $a\ (P_I)$ is
  primitive (and is allowed to depend on $x$), and $A \geq 1$
  is fixed, then
\begin{equation}\label{qq-mpz}
  \sum_{\substack{q \leq Q\\q \in \Scal_I}} |\Delta(\Lambda \onef_{[x,2x]}; a\ (q))| \ll x \log^{-A} x,
\end{equation}
where the implied constant depends only on the fixed quantities $(A,\varpi,\delta)$, but
not on $a$.
\end{claim}

It is clear that $\EH[\frac{1}{2}+2\varpi]$ implies $\MPZ[\varpi,\delta]$ whenever $\varpi,\delta \geq 0$.  The first non-trivial estimate of the form $\MPZ[\varpi,\delta]$ was established by Zhang \cite{zhang}, who (essentially) obtained $\MPZ[\varpi,\delta]$ whenever $0 \leq \varpi, \delta < \frac{1}{1168}$.
In \cite[Theorem 2.17]{polymath8a}, we improved this result to the following.

\begin{theorem}\label{mpz-poly}  $\MPZ[\varpi,\delta]$ holds for every fixed $\varpi,\delta \geq 0$ with $600 \varpi + 180 \delta < 7$.
\end{theorem}

In fact, a stronger result was established in \cite{polymath8a}, in which the moduli $q$ were assumed to be \emph{densely divisible} rather than smooth, but we will not exploit such improvements here.  For our application, the most important thing is to get $\varpi$ as large as possible; in particular, Theorem \ref{mpz-poly} allows one to get $\varpi$ arbitrarily close to $\frac{7}{600} \approx 0.01167$. 

%
%

In this paper, we will also study the following generalization of the Elliott-Halberstam conjecture for a fixed choice of $0 < \vartheta < 1$:

\begin{claim}[Generalized Elliott-Halberstam conjecture, {$\GEH[\vartheta]$}]\label{geh-def}  Let $\eps > 0$ and $A \geq 1$ be fixed.  Let $N,M$ be quantities such that
$x^\eps \lessapprox N \lessapprox x^{1-\eps}$ and $x^\eps \lessapprox M \lessapprox x^{1-\eps}$ with $NM \sim x$, and let $\alpha, \beta: \N \to \R$ be sequences supported on $[N,2N]$ and $[M,2M]$ respectively, such that one has the pointwise bounds
\begin{equation}\label{ab-div}
 |\alpha(n)| \ll \tau(n)^{O(1)} \log^{O(1)} x; \quad |\beta(m)| \ll \tau(m)^{O(1)} \log^{O(1)} x
\end{equation}
for all natural numbers $n,m$.  Suppose also that $\beta$ obeys the Siegel-Walfisz type bound
\begin{equation}\label{sig}
| \Delta(\beta \onef_{(\cdot,r)=1}; a\ (q)) | \ll \tau(qr)^{O(1)} M \log^{-A} x 
\end{equation}
for any $q,r \geq 1$, any fixed $A$, and any primitive residue class $a\ (q)$.  Then for any $Q \lessapprox x^\vartheta$, we have
\begin{equation}\label{qq-gen}
 \sum_{q \leq Q} \sup_{a \in (\Z/q\Z)^\times} |\Delta(\alpha \star \beta; a\ (q))| \ll x \log^{-A} x.
\end{equation} 
\end{claim}

In \cite[Conjecture 1]{bfi} it was essentially conjectured\footnote{Actually, there are some differences between \cite[Conjecture 1]{bfi} and the claim here.  Firstly, we need an estimate that is uniform for all $a$, whereas in \cite{bfi} only the case of a fixed modulus $a$ was asserted.  On the other hand, $\alpha,\beta$ were assumed to be controlled in $\ell^2$ instead of via the pointwise bounds \eqref{ab-div}, and $Q$ was allowed to be as large as $x \log^{-C} x$ for some fixed $C$ (although, in view of the negative results in \cite{fg}, \cite{fghm}, this latter strengthening may be too ambitious).} that $\GEH[\vartheta]$ was true for all $0 < \vartheta < 1$.   This is stronger than the Elliott-Halberstam conjecture:

\begin{proposition}\label{geh-eh}  For any fixed $0 < \vartheta < 1$, $\GEH[\vartheta]$ implies $\EH[\vartheta]$.
\end{proposition}

\begin{proof} (Sketch)  As this argument is standard, we give only a brief sketch.
Let $A > 0$ be fixed.  For $n \in [x,2x]$, we have Vaughan's identity\footnote{One could also use the Heath-Brown identity \cite{hb-ident} here if desired.} \cite{vaughan}
$$ \Lambda(n) = \mu_< \star  L(n) - \mu_< \star  \Lambda_< \star  1(n) + \mu_\geq \star  \Lambda_\geq \star  1(n),$$
where $L(n) \coloneqq \log(n)$, $1(n) \coloneqq 1$, and
\begin{gather}
\Lambda_\geq(n) \coloneqq \Lambda(n) \onef_{n \geq x^{1/3}},\quad\quad 
\Lambda_<(n) \coloneqq \Lambda(n) \onef_{n < x^{1/3}}\\
\mu_\geq(n) \coloneqq \mu(n) \onef_{n \geq x^{1/3}},\quad\quad \mu_<(n) \coloneqq \mu(n)
\onef_{n < x^{1/3}}.
\end{gather}
By decomposing each of the functions $\mu_<$, $\mu_\geq$, $1$, $\Lambda_<$, $\Lambda_{\geq}$ into $O(\log^{A+1} x)$ functions supported on intervals of the form $[N, (1+\log^{-A} x) N]$, and discarding those contributions which meet the boundary of $[x,2x]$ (cf. \cite{fouvry}, \cite{fi-2}, \cite{bfi}, \cite{zhang}), and using $\GEH[\vartheta]$ (with $A$ replaced by a much larger fixed constant $A'$) to control all remaining contributions, we obtain the claim (using the Siegel-Walfisz theorem, see e.g. \cite[Satz 4]{siebert}
  or~\cite[Th. 5.29]{ik}).
\end{proof}

By modifying the proof of the Bombieri-Vinogradov theorem Motohashi \cite{motohashi} established the following generalization of that theorem (see also \cite{gallagher} for some related ideas):

\begin{theorem}[Generalized Bombieri-Vinogradov theorem]\label{gbv-thm}\cite{motohashi}  $\GEH[\vartheta]$ holds for every fixed $0 < \vartheta < 1/2$.
\end{theorem}

One could similarly describe a generalization of the Motohashi-Pintz-Zhang estimate $\MPZ[\varpi,\delta]$, but unfortunately the arguments in \cite{zhang} or Theorem \ref{mpz-poly} do not extend to this setting unless one is in the ``Type I/Type II'' case in which $N,M$ are constrained to be somewhat close to $x^{1/2}$, or if one has ``Type III'' structure to the convolution $\alpha \star \beta$, in the sense that it can refactored as a convolution involving several ``smooth'' sequences.  In any event, our analysis would not be able to make much use of such incremental improvements to $\GEH[\vartheta]$, as we only use this hypothesis effectively in the case when $\vartheta$ is very close to $1$.  In particular, we will not directly use Theorem \ref{gbv-thm} in this paper.


\section{Outline of the key ingredients}\label{subclaim-sec}

In this section we describe the key subtheorems used in the proof of Theorem \ref{main}, with the proofs of these subtheorems mostly being deferred to later sections.

We begin with a weak version of the Dickson-Hardy-Littlewood prime tuples conjecture \cite{hardy}, which (following Pintz \cite{pintz-polignac}) we refer to as $\DHL[k,j]$.  Recall that for any $k \in \mathbb{N}$, an \emph{admissible $k$-tuple} is a tuple ${\mathcal H} = (h_1,\ldots,h_{k})$ of $k$ increasing integers $h_1 < \ldots < h_{k}$ which avoids at least one residue class $a_p\ (p) := \{ a_p + np: n \in \Z \}$ for every $p$.  For instance, $(0,2,6)$ is an admissible $3$-tuple, but $(0,2,4)$ is not.

For any $k \geq j \geq 2$, we let $\DHL[k,j]$ denote the following claim:

\begin{claim}[Weak Dickson-Hardy-Littlewood conjecture,
  {$\DHL[k,j]$}] For any admissible $k$-tuple ${\mathcal
    H}=(h_1,\ldots,h_{k})$ there
  exist infinitely many translates $n + {\mathcal H} =
  (n+h_1,\ldots,n+h_{k})$ of ${\mathcal H}$ which contain at least $j$
  primes.
\end{claim}

The full Dickson-Hardy-Littlewood conjecture is then the assertion that $\DHL[k,k]$ holds for all $k \ge 2$.  In our analysis we will focus on the case when $j$ is much smaller than $k$; in fact $j$ will be of the order of $\log k$.

For any $k$, let $H(k)$ denote the minimal diameter $h_k-h_1$ of an admissible $k$-tuple; thus for instance $H(3)=6$.  It is clear that for any natural numbers $m \geq 1$ and $k \geq m+1$, the claim $\DHL[k,m+1]$ implies that $H_m \leq H(k)$ (and the claim $\DHL[k,k]$ would imply that $H_{k-1} = H(k)$).  We will therefore deduce Theorem \ref{main} from a number of claims of the form $\DHL[k,j]$.  More precisely, we have:

\begin{theorem}\label{main-dhl} Unconditionally, we have the following claims:
\begin{itemize}
\item[(i)] $\DHL[50,2]$.
\item[(ii)] $\DHL[\num{35410},3]$.
\item[(iii)] $\DHL[\num{1649821},4]$.
\item[(iv)] $\DHL[\num{75845707},5]$.
\item[(v)] $\DHL[\num{3473955908},6]$.
\item[(vi)] $\DHL[k,m+1]$ whenever $m \geq 1$ and $k \geq C\exp( (4 - \frac{28}{157}) m )$ for some sufficiently large absolute (and effective) constant $C$.
\end{itemize}
Assume the Elliott-Halberstam conjecture $\EH[\vartheta]$ for all $0 < \vartheta < 1$.  Then we have the following improvements:
\begin{itemize}
\item[(vii)] $\DHL[54,3]$.
\item[(viii)] $\DHL[\num{5511},4]$.
\item[(ix)] $\DHL[\num{41588},5]$.
\item[(x)] $\DHL[\num{309661},6]$.
\item[(xi)] $\DHL[k,m+1]$ whenever $m \geq 1$ and $k \geq C\exp( 2 m )$ for some sufficiently large absolute (and effective) constant $C$.
\end{itemize}
Assume the generalized Elliott-Halberstam conjecture $\GEH[\vartheta]$ for all $0 < \vartheta < 1$.  Then
\begin{itemize}
\item[(xii)] $\DHL[3,2]$.
\item[(xiii)] $\DHL[51,3]$.
\end{itemize}
\end{theorem}

Theorem \ref{main} then follows from Theorem \ref{main-dhl} and the following bounds on $H(k)$ (ordered by increasing value of $k$):

\begin{theorem}[Bounds on $H(k)$]\label{hk-bound}\
\begin{itemize}
\item[(xii)] $H(3)=6$.
\item[(i)] $H(50) = 246$.
\item[(xiii)] $H(51) = 252$.
\item[(vii)] $H(54) = 270$.
\item[(viii)] $H(\num{5511}) \leq \num{52116}$.
\item[(ii)] $H(\num{35410}) \leq \num{398130}$.
\item[(ix)] $H(\num{41588}) \leq \num{474266}$.
\item[(x)] $H(\num{309661}) \leq \num{4137854}$.
\item[(iii)] $H(\num{1649821}) \leq \num{24797814}$.
\item[(iv)] $H(\num{75845707}) \leq \num{1431556072}$.
\item[(v)] $H(\num{3473955908}) \leq \num{80550202480}$.
\item[(vi), (xi)] In the asymptotic limit $k \to \infty$, one has $H(k) \leq k \log k + k \log\log k - k + o(k)$, with the bounds on the decay rate $o(k)$ being effective.
\end{itemize}
\end{theorem}

We prove Theorem \ref{hk-bound} in Section \ref{tuples-sec}.  In the opposite direction, an application of the Brun-Titchmarsh theorem gives $H(k) \geq (\frac{1}{2} + o(1)) k \log k$ as $k \to \infty$; see \cite[\S 3.9]{polymath8a-unabridged} for this bound, as well as with some slight refinements.

\begin{table}
\centering
\setlength{\extrarowheight}{2pt}
\caption{Results used to prove various components of Theorem \ref{main-dhl}. Note that Theorems \ref{maynard-thm}, \ref{maynard-trunc}, \ref{epsilon-trick}, \ref{epsilon-beyond} are in turn proven using Theorems \ref{prime-asym}, \ref{nonprime-asym}, and Lemma \ref{crit}.}\label{ingredients}
\begin{tabular}{lr}
\toprule
Theorem \ref{main-dhl}                 & Results used \\
\midrule
(i) & Theorems \ref{bv-thm}, \ref{epsilon-trick}, \ref{mke-lower} \\
(ii)-(vi) & Theorems \ref{mpz-poly}, \ref{maynard-trunc}, \ref{mlower-var} \\
(vii)-(xi) & Theorems \ref{maynard-thm}, \ref{mlower} \\
(xii) & Theorems \ref{epsilon-beyond}, \ref{piece} \\
(xiii) & Theorems \ref{epsilon-trick}, \ref{mke-lower} \\
\bottomrule
\end{tabular}
\end{table}

The proof of Theorem \ref{main-dhl} follows the Goldston-Pintz-Y{\i}ld{\i}r{\i}m strategy that was also used in all previous progress on this problem (e.g. \cite{gpy}, \cite{mp}, \cite{zhang}, \cite{polymath8a}, \cite{maynard-new}), namely that of constructing a sieve function adapted to an admissible $k$-tuple with good properties.  More precisely, we set
$$ w := \log \log \log x$$
and 
$$ W := \prod_{p \leq w} p,$$
and observe the crude bound
\begin{equation}\label{W-bound}
 W \ll \log \log^{O(1)} x.
\end{equation}
We have the following simple ``pigeonhole principle'' criterion for $\DHL[k,m+1]$ (cf. \cite[Lemma 4.1]{polymath8a}, though the normalization here is slightly different):

\begin{lemma}[Criterion for $\DHL$]\label{crit} Let $k \geq 2$ and $m \geq 1$ be fixed integers, and define the normalization constant
\begin{equation}\label{bnorm}
 B := \frac{\phi(W)}{W} \log x.
\end{equation}
Suppose that for each fixed admissible $k$-tuple $(h_1,\dots,h_k)$
  and each residue class $b\ (W)$ such that $b+h_i$ is coprime to $W$ for all $i=1,\dots,k$, one can find a non-negative weight
  function $\nu \colon \N \to \R^+$ and fixed quantities $\alpha > 0$ and $\beta_1,\dots,\beta_k \geq
  0$, such that one has the
  asymptotic upper bound
\begin{equation}\label{s1}
 \sum_{\substack{x \leq n \leq 2x\\ n = b\ (W)}} \nu(n) \leq (\alpha+o(1)) B^{-k} \frac{x}{W},
 \end{equation}
the asymptotic lower bound
\begin{equation}\label{s2}
  \sum_{\substack{x \leq n \leq 2x\\ n = b\ (W)}} \nu(n) \theta(n+h_i) \geq (\beta_i-o(1)) B^{1-k} \frac{x}{\phi(W)}
\end{equation}
for all $i=1,\dots,k$, and the key inequality
\begin{equation}\label{key}
\frac{\beta_1 + \dots + \beta_k}{\alpha} > m.
\end{equation}
Then $\DHL[k,m+1]$ holds.
\end{lemma}

\begin{proof}
Let $(h_1,\ldots,h_{k})$ be a fixed
  admissible $k$-tuple.  Since it is admissible, there is at least
  one residue class $b\ (W)$ such that $(b+h_i,W)=1$ for all $h_i
  \in {\mathcal H}$.  For an arithmetic function $\nu$ as in the
  lemma, we consider the quantity
$$
N:=\sum_{\substack{x \leq n \leq 2x\\ n = b\ (W)}}
 \nu(n) \left(\sum_{i=1}^{k} \theta(n+h_i) - m\log 3x\right).
$$
Combining \eqref{s1} and \eqref{s2}, we obtain the lower bound
$$
N\geq (\beta_1+\dots+\beta_k-o(1)) B^{1-k} \frac{x}{\phi(W)} - (m\alpha+o(1)) B^{-k} \frac{x}{W} \log 3x.
$$
From \eqref{bnorm} and the crucial condition \eqref{key}, it follows that
$N>0$ if $x$ is sufficiently large.
\par
On the other hand, the sum
$$
\sum_{i=1}^{k} \theta(n+h_i) - m\log 3x
$$
can  be positive only if $n+h_i$ is prime for \emph{at least} $m+1$
indices $i=1, \ldots, k$.  We conclude that, for all sufficiently
large $x$, there exists some integer $n \in [x,2x]$ such that $n+h_i$
is prime for at least $m+1$ values of $i=1,\ldots,k$.
\par
Since $(h_1,\dots,h_k)$ is an arbitrary admissible $k$-tuple,
$\DHL[k,m+1]$ follows.
\end{proof}

The objective is then to construct non-negative weights $\nu$ whose associated ratio $\frac{\beta_1+\dots+\beta_k}{\alpha}$ has provable lower bounds that are as large as possible.   Our sieve majorants will be a variant of the multidimensional Selberg sieves used in \cite{maynard-new}.  As with all Selberg sieves, the $\nu$ are constructed as the square of certain (signed) divisor sums.  The divisor sums we will use will be finite linear combinations of products of ``one-dimensional'' divisor sums.  More precisely, for any fixed smooth compactly supported function $F: [0,+\infty) \to \R$, define the divisor sum $\lambda_F: \Z \to \R$ by the formula
\begin{equation}\label{lambdaf-def}
\lambda_F(n) := \sum_{d|n} \mu(d) F( \log_x d )
\end{equation}
where $\log_x$ denotes the base $x$ logarithm
\begin{equation}\label{logx-def}
\log_x n:= \frac{\log n}{\log x}.
\end{equation}
One should think of $\lambda_F$ as a smoothed out version of the indicator function to numbers $n$ which are ``almost prime'' in the sense that they have no prime factors less than $x^\eps$ for some small fixed $\eps>0$; see Proposition \ref{almostprime} for a more rigorous version of this heuristic.

The functions $\nu$ we will use will take the form
\begin{equation}\label{nuform}
 \nu(n) = \left( \sum_{j=1}^J c_j \lambda_{F_{j,1}}(n+h_1) \dots \lambda_{F_{j,k}}(n+h_k) \right)^2 
\end{equation}
for some fixed natural number $J$, fixed coefficients $c_1,\dots,c_J \in \R$ and fixed smooth compactly supported functions $F_{j,i}: [0,+\infty) \to \R$ with $j=1,\dots,J$ and $i=1,\dots,k$.  (One can of course absorb the constant $c_j$ into one of the $F_{j,i}$ if one wishes.)   Informally, $\nu$ is a smooth restriction to those $n$ for which $n+h_1,\dots,n+h_k$ are all almost prime.

Clearly, $\nu$ is a (positive-definite) fixed linear combination of functions of the form
$$ n \mapsto \prod_{i=1}^k \lambda_{F_i}(n+h_i) \lambda_{G_i}(n+h_i)$$
for various fixed smooth functions $F_1,\dots,F_k,G_1,\dots,G_k: [0,+\infty) \to \R$.  The sum appearing in \eqref{s1} can thus be decomposed into fixed linear combinations of sums of the form
\begin{equation}\label{sfg-1}
 \sum_{\substack{x \leq n \leq 2x\\ n = b\ (W)}} \prod_{i=1}^k \lambda_{F_i}(n+h_i) \lambda_{G_i}(n+h_i).
\end{equation}
Also, if $F$ is supported on $[0,1]$, then from \eqref{lambdaf-def} we clearly have
\begin{equation}\label{lambdan-prime}
\lambda_F(n) = F(0)
\end{equation}
when $n \geq x$ is prime, and so the sum appearing in \eqref{s2} can be similarly decomposed in this case into fixed linear combinations of sums of the form
\begin{equation}\label{sfg-2}
 \sum_{\substack{x \leq n \leq 2x\\ n = b\ (W)}} \theta(n+h_i) \prod_{1 \leq i' \leq k; i' \neq i} \lambda_{F_{i'}}(n+h_{i'}) \lambda_{G_{i'}}(n+h_{i'}).
\end{equation}

To estimate the sums \eqref{sfg-2}, we use the following asymptotic, proven in Section \ref{sieving-sec}.  For each compactly supported $F: [0,+\infty) \to \R$, let
\begin{equation}\label{S-def}
S(F) \coloneqq \sup \{ x \geq 0: F(x) \neq 0\}
\end{equation}
 denote the upper range of the support of $F$ (with the convention that $S(0)=0$).

\begin{theorem}[Asymptotic for prime sums]\label{prime-asym}  Let $k \geq 2$ be fixed, let $(h_1,\dots,h_k)$ be a fixed admissible $k$-tuple, and let $b\ (W)$ be such that $b+h_i$ is coprime to $W$ for each $i=1,\dots,k$.  Let $1 \leq i_0 \leq k$ be fixed, and for each $1 \leq i \leq k$ distinct from $i_0$, let $F_{i}, G_{i}: [0,+\infty) \to \R$ be fixed smooth compactly supported functions.  Assume one of the following hypotheses:
\begin{itemize}
\item[(i)]  (Elliott-Halberstam) There exists a fixed $0 < \vartheta < 1$ such that $\EH[\vartheta]$ holds, and such that
\begin{equation}\label{fg-upper}
\sum_{1 \leq i \leq k; i \neq i_0} ( S(F_{i}) + S(G_{i}) ) < \vartheta.
\end{equation}
\item[(ii)]  (Motohashi-Pintz-Zhang) There exists fixed $0 \leq \varpi < 1/4$ and $\delta > 0$ such that $\MPZ[\varpi,\delta]$ holds, and such that
\begin{equation}\label{fg-upper-alt}
\sum_{1 \leq i \leq k; i \neq i_0} ( S(F_{i}) + S(G_{i}) ) < \frac{1}{2} + 2 \varpi
\end{equation}
and
\begin{equation}
\max_{1 \leq i \leq k; i \neq i_0} \Bigl\{S(F_{i}), S(G_{i}) \Bigr\}< \delta.
\end{equation}
\end{itemize}
Then we have
\begin{equation}\label{theta-oo}
 \sum_{\substack{x \leq n \leq 2x\\ n = b\ (W)}} \theta(n+h_{i_0}) \prod_{1 \leq i \leq k; i \neq i_0} \lambda_{F_{i}}(n+h_{i}) \lambda_{G_{i}}(n+h_{i}) =
(c+o(1)) B^{1-k} \frac{x}{\phi(W)}
\end{equation}
where $B$ is given by \eqref{bnorm} and
$$ c := \prod_{1 \leq i \leq k; i \neq i_0} \left(\int_0^1 F'_i(t_{i}) G'_{i}(t_{i})\ dt_{i}\right).$$
Here of course $F'$ denotes the derivative of $F$.
\end{theorem}

To estimate the sums \eqref{sfg-1}, we use the following asymptotic, also proven in Section \ref{sieving-sec}.

\begin{theorem}[Asymptotic for non-prime sums]\label{nonprime-asym}  Let $k \geq 1$ be fixed, let $(h_1,\dots,h_k)$ be a fixed admissible $k$-tuple, and let $b\ (W)$ be such that $b+h_i$ is coprime to $W$ for each $i=1,\dots,k$.  For each fixed $1 \leq i \leq k$, let $F_{i}, G_{i}: [0,+\infty) \to \R$ be fixed smooth compactly supported functions.  Assume one of the following hypotheses:
\begin{itemize}
\item[(i)]  (Trivial case)  One has
\begin{equation}\label{easy-upper}
\sum_{i=1}^k ( S(F_i) + S(G_i) ) < 1.
\end{equation}
\item[(ii)]  (Generalized Elliott-Halberstam) There exists a fixed $0 < \vartheta < 1$ and $i_0 \in \{1,\dots,k\}$ such that $\GEH[\vartheta]$ holds, and
\begin{equation}\label{eh-upper}
\sum_{1 \leq i \leq k; i \neq i_0} ( S(F_{i}) + S(G_{i}) ) < \vartheta.
\end{equation}
\end{itemize}
Then we have
\begin{equation}\label{lflg}
 \sum_{\substack{x \leq n \leq 2x\\ n = b\ (W)}} \prod_{i=1}^k \lambda_{F_{i}}(n+h_{i}) \lambda_{G_{i}}(n+h_{i}) =
(c+o(1)) B^{-k} \frac{x}{W},
\end{equation}
where $B$ is given by \eqref{bnorm} and
\begin{equation}\label{c-def}
 c := \prod_{i=1}^k \left(\int_0^1 F'_i(t_i) G'_i(t_i)\ dt_i\right).
\end{equation}
\end{theorem}

A key point in (ii) is that no upper bound on $S(F_{i_0})$ or $S(G_{i_0})$ is required (although, as we will see in Section \ref{geh-case}, the result is a little easier to prove when one has $S(F_{i_0})+S(G_{i_0}) < 1$).  This flexibility in the $F_{i_0}, G_{i_0}$ functions will be particularly crucial to obtain part (xii) of Theorem \ref{main-dhl} and Theorem \ref{main}.

\begin{remark}  Theorems \ref{prime-asym}, \ref{nonprime-asym} can be viewed as probabilistic assertions of the following form: if $n$ is chosen uniformly at random from the set $\{ x \leq n \leq 2x: n = b\ (W)\}$, then the random variables $\theta(n+h_i)$ and $\lambda_{F_j}(n+h_j) \lambda_{G_j}(n+h_j)$ for $i,j=1,\dots,k$ have mean $(1+o(1)) \frac{W}{\phi(W)}$ and $(\int_0^1 F'_j(t) G'_j(t)\ dt + o(1)) B^{-1}$ respectively, and furthermore these random variables enjoy a limited amount of independence, except for the fact (as can be seen from \eqref{lambdan-prime}) that $\theta(n+h_i)$ and $\lambda_{F_i}(n+h_i) \lambda_{G_i}(n+h_i)$ are highly correlated.  Note though that we do not have asymptotics for any sum which involves two or more factors of $\theta$, as such estimates are of a difficulty at least as great as that of the twin prime conjecture (which is equivalent to the divergence of the sum $\sum_n \theta(n) \theta(n+2)$).
\end{remark}

Theorems \ref{prime-asym}, \ref{nonprime-asym} may be combined with Lemma \ref{crit} to reduce the task of establishing estimates of the form $\DHL[k,m+1]$ to that of obtaining sufficiently good solutions to certain variational problems.  For instance, in Section \ref{may-sec} we reprove the following result of Maynard \cite[Proposition 4.2]{maynard-new}:

\begin{theorem}[Sieving on the standard simplex]\label{maynard-thm}  Let $k \geq 2$ and $m \geq 1$ be fixed integers. For any fixed compactly supported square-integrable function $F: [0,+\infty)^k \to \R$, define the functionals
\begin{equation}\label{i-def}
I(F) := \int_{[0,+\infty)^k} F(t_1,\dots,t_k)^2\ dt_1 \dots dt_k
\end{equation}
and
\begin{equation}\label{ji-def}
 J_i(F) := \int_{[0,+\infty)^{k-1}} \left(\int_0^\infty F(t_1,\dots,t_k)\ dt_i\right)^2 dt_1 \dots dt_{i-1} dt_{i+1} \dots dt_k
\end{equation}
for $i=1,\dots,k$, and let $M_k$ be the supremum
\begin{equation}\label{mk4}
 M_k := \sup \frac{\sum_{i=1}^k J_i(F)}{I(F)}
\end{equation}
over all square-integrable functions $F$ that are supported on the simplex
$$ {\mathcal R}_k := \{ (t_1,\dots,t_k) \in [0,+\infty)^k: t_1+\dots+t_k \leq 1 \}$$
and are not identically zero (up to almost everywhere equivalence, of course).  Suppose that there is a fixed $0 < \vartheta < 1$ such that $\EH[\vartheta]$ holds, and such that
$$ M_k > \frac{2m}{\vartheta}.$$
Then $\DHL[k,m+1]$ holds.
\end{theorem}

Parts (vii)-(xi) of Theorem \ref{main-dhl} (and hence Theorem \ref{main}) are then immediate from the following results, proven in Sections \ref{asymptotics-sec}, \ref{h1-sec}, and ordered by increasing value of $k$:

\begin{theorem}[Lower bounds on $M_k$]\label{mlower}\
\begin{itemize}
\item[(vii)] $M_{54} > 4.00238$.
\item[(viii)] $M_{5511} > 6$.
\item[(ix)] $M_{41588} > 8$.
\item[(x)] $M_{309661} > 10$.
\item[(xi)] One has $M_k \geq \log k - C$ for all $k \geq C$, where $C$ is an absolute (and effective) constant.
\end{itemize}
\end{theorem}

For sake of comparison, in \cite[Proposition 4.3]{maynard-new} it was shown that $M_5 > 2$, $M_{105} > 4$, and $M_k \geq \log k - 2\log\log k - 2$ for all sufficiently large $k$.  As remarked in that paper, the sieves used on the bounded gap problem prior to the work in \cite{maynard-new} would essentially correspond, in this notation, to the choice of functions $F$ of the special form $F(t_1,\dots,t_k) := f(t_1+\dots+t_k)$, which severely limits the size of the ratio in
\eqref{mk4} (in particular, the analogue of $M_k$ in this special case cannot exceed $4$, as shown in \cite{sound}).

In the converse direction, in Corollary \ref{mk-upper} we will also show the upper bound $M_k \leq \frac{k}{k-1} \log k$ for all $k \geq 2$, which shows in particular that the bounds in (vii) and (xi) of the above theorem cannot be significantly improved.  We  remark that Theorem \ref{mlower}(vii) and the Bombieri-Vinogradov theorem also gives a weaker version $\DHL[54,2]$ of Theorem \ref{main-dhl}(i).

We also have a variant of Theorem \ref{maynard-thm} which can accept inputs of the form $\MPZ[\varpi,\delta]$:

\begin{theorem}[Sieving on a truncated simplex]\label{maynard-trunc}  Let $k \geq 2$ and $m \geq 1$ be fixed integers. Let $0<\varpi<1/4$ and $0 <\delta < 1/2$ be such that $\MPZ[\varpi,\delta]$ holds.  For any $\alpha>0$, let $M_k^{[\alpha]}$ be defined as in \eqref{mk4}, but where the supremum now ranges over all square-integrable functions $F$ supported in the \emph{truncated} simplex
\begin{equation}\label{ttk}
 \{ (t_1,\dots,t_k) \in [0,\alpha]^k: t_1+\dots+t_k \leq 1 \}
\end{equation}
and are not identically zero.  If
$$ M_k^{[\frac{\delta}{1/4+\varpi}]} > \frac{m}{1/4+\varpi},$$
then $\DHL[k,m+1]$ holds.
\end{theorem}

In Section \ref{asymptotics-sec} we will establish the following variant of Theorem \ref{mlower}, which when combined with Theorem \ref{mpz-poly}, allows one to use Theorem \ref{maynard-trunc} to establish parts (ii)-(vi) of Theorem \ref{main-dhl} (and hence Theorem \ref{main}):

\begin{theorem}[Lower bounds on $M_k^{[\alpha]}$]\label{mlower-var}\
\begin{itemize}
\item[(ii)] There exist $\delta,\varpi>0$ with $600 \varpi + 180 \delta < 7$ and $M_{\num{35410}}^{[\frac{\delta}{1/4+\varpi}]} > \frac{2}{1/4+\varpi}$.
\item[(iii)] There exist $\delta,\varpi>0$ with $600 \varpi + 180 \delta < 7$ and $M_{\num{1649821}}^{[\frac{\delta}{1/4+\varpi}]} > \frac{3}{1/4+\varpi}$.
\item[(iv)] There exist $\delta,\varpi>0$ with $600 \varpi + 180 \delta < 7$ and $M_{\num{75845707}}^{[\frac{\delta}{1/4+\varpi}]} > \frac{4}{1/4+\varpi}$.
\item[(v)] There exist $\delta,\varpi>0$ with $600 \varpi + 180 \delta < 7$ and $M_{\num{3473955908}}^{[\frac{\delta}{1/4+\varpi}]} > \frac{5}{1/4+\varpi}$.
\item[(vi)] For all $k \geq C$, there exist $\delta,\varpi>0$ with $600 \varpi + 180 \delta < 7$, $\varpi \geq \frac{7}{600} - \frac{C}{\log k}$, and $M_k^{[\frac{\delta}{1/4+\varpi}]} \geq \log k - C$ for some absolute (and effective) constant $C$.
\end{itemize}
\end{theorem}

The implication is clear for (ii)-(v).  For (vi), observe that from Theorem \ref{mlower-var}(vi), Theorem \ref{mpz-poly}, and Theorem \ref{maynard-trunc}, we see that $\DHL[k,m+1]$ holds whenever $k$ is sufficiently large and
$$ m \leq (\log k - C) \left(\frac{1}{4} + \frac{7}{600} - \frac{C}{\log k}\right)$$
which is in particular implied by
$$ m \leq \frac{\log k}{4 - \frac{28}{157}} - C'$$
for some absolute constant $C'$, giving Theorem \ref{main-dhl}(vi).

Now we give a more flexible variant of Theorem \ref{maynard-thm}, in which the support of $F$ is enlarged, at the cost of reducing the range of integration of the $J_i$.

\begin{theorem}[Sieving on an epsilon-enlarged simplex]\label{epsilon-trick}  Let $k \geq 2$ and $m \geq 1$ be fixed integers, and let $0 < \eps < 1$ be fixed also. For any fixed compactly supported square-integrable function $F: [0,+\infty)^k \to \R$, define the functionals
$$ J_{i,1-\eps}(F) := \int_{(1-\eps) \cdot {\mathcal R}_{k-1}} \left(\int_0^\infty F(t_1,\dots,t_k)\ dt_i\right)^2 dt_1 \dots dt_{i-1} dt_{i+1} \dots dt_k$$
for $i=1,\dots,k$, and let $M_{k,\eps}$ be the supremum
$$ M_{k,\eps} := \sup \frac{\sum_{i=1}^k J_{i,1-\eps}(F)}{I(F)}$$
over all square-integrable functions $F$ that are supported on the simplex
$$ (1+\eps) \cdot {\mathcal R}_k = \{ (t_1,\dots,t_k) \in [0,+\infty)^k: t_1+\dots+t_k \leq 1+\eps \}$$
and are not identically zero.  Suppose that there is a fixed $0 < \vartheta < 1$, such that one of the following two hypotheses holds:
\begin{itemize}
\item[(i)]  $\EH[\vartheta]$ holds, and $1+\eps < \frac{1}{\vartheta}$.
\item[(ii)] $\GEH[\vartheta]$ holds, and $\eps < \frac{1}{k-1}$.
\end{itemize}
If
$$ M_{k,\eps} > \frac{2m}{\vartheta}$$
then $\DHL[k,m+1]$ holds.
\end{theorem}

We prove this theorem in Section \ref{trick-sec}.  We remark that due to the continuity of $M_{k,\eps}$ in $\eps$, the strict inequalities in (i), (ii) of this theorem may be replaced by non-strict inequalities. Parts (i), (xiii) of Theorem \ref{main-dhl}, and a weaker version $\DHL[4,2]$ of part (xii), then follow from Theorem \ref{bv-thm} and the following computations, proven in Sections \ref{mkeps-sec}, \ref{4d}:

\begin{theorem}[Lower bounds on $M_{k,\eps}$]\label{mke-lower}\
\begin{itemize}
\item[(i)] $M_{50,1/25} > 4.0043$.
\item[(xii$'$)] $M_{4,0.168} > 2.00558$.
\item[(xiii)] $M_{51,1/50} > 4.00156$.
\end{itemize}
\end{theorem}

We remark that computations in the proof of Theorem \ref{mke-lower}(xii$'$) are simple enough that the bound may be checked by hand, without use of a computer.  The computations used to establish the full strength of Theorem \ref{main-dhl}(xii) are however significantly more complicated.  

In fact, we may enlarge the support of $F$ further.  We give a version corresponding to part (ii) of Theorem \ref{epsilon-trick}; there is also a version corresponding to part (i), but we will not give it here as we will not have any use for it.

\begin{theorem}[Going beyond the epsilon enlargement]\label{epsilon-beyond}  Let $k \geq 2$ and $m \geq 1$ be fixed integers, let $0 <\vartheta < 1$ be a fixed quantity such that $\GEH[\vartheta]$ holds, and let $0 < \eps < \frac{1}{k-1}$ be fixed also. Suppose that there is a fixed non-zero square-integrable function $F: [0,+\infty)^k \to \R$ supported in $\frac{k}{k-1} \cdot {\mathcal R}_k$, such that for $i=1,\dots,k$ one has the vanishing marginal condition
\begin{equation}\label{vanishing-marginal}
\int_0^\infty F(t_1,\dots,t_k)\ dt_i = 0
\end{equation}
 whenever $t_1,\dots,t_{i-1},t_{i+1},\dots,t_k \geq 0$ are such that
$$ t_1+\dots+t_{i-1}+t_{i+1}+\dots+t_k > 1+\eps.$$
Suppose that we also have the inequality
$$ \frac{\sum_{i=1}^k J_{i,1-\eps}(F)}{I(F)} > \frac{2m}{\vartheta}.$$
Then $\DHL[k,m+1]$ holds.
\end{theorem}

This theorem is proven in Section \ref{beyond-sec}.  Theorem \ref{main-dhl}(xii) is then an immediate consequence of Theorem \ref{epsilon-beyond} and the following numerical fact, established in Section \ref{3d}.

\begin{theorem}[A piecewise polynomial cutoff]\label{piece}  Set $\eps := \frac{1}{4}$.  Then there exists a piecewise polynomial function $F: [0,+\infty)^3 \to \R$ supported on the simplex
$$ \frac{3}{2} \cdot {\mathcal R}_3 = \left\{ (t_1,t_2,t_3) \in [0,+\infty)^3: t_1+t_2+t_3 \leq \frac{3}{2}\right\}$$
and symmetric in the $t_1,t_2,t_3$ variables, such that $F$ is not identically zero and obeys the vanishing marginal condition
$$
\int_0^\infty F(t_1,t_2,t_3)\ dt_3 = 0$$
whenever $t_1,t_2 \geq 0$ with $t_1+t_2 > 1+\eps$, and such that
$$ \frac{3 \int_{t_1+t_2 \leq 1-\eps} (\int_0^\infty F(t_1,t_2,t_3)\ dt_3)^2\ dt_1 dt_2}{\int_{[0,\infty)^3} F(t_1,t_2,t_3)^2\ dt_1 dt_2 dt_3} > 2.$$
\end{theorem}

There are several other ways to combine Theorems \ref{prime-asym}, \ref{nonprime-asym} with equidistribution theorems on the primes to obtain results of the form $\DHL[k,m+1]$, but all of our attempts to do so either did not improve the numerology, or else were numerically infeasible to implement. 


\section{Multidimensional Selberg sieves}\label{sieving-sec}

In this section we prove Theorems \ref{prime-asym} and \ref{nonprime-asym}.  A key asymptotic used in both theorems is the following:

\begin{lemma}[Asymptotic]\label{mul-asym}  Let $k \geq 1$ be a fixed integer,
 and let $N$ be a natural number coprime to $W$ with $\log{N}=O(\log^{O(1)}{x})$. Let $F_1,\dots,F_k,G_1,\dots,G_k: [0,+\infty)\to \R$ be fixed smooth compactly supported functions.  Then
\begin{equation}\label{multisum}
 \sum_{\substack{d_1,\dots,d_k,d'_1,\dots,d'_k \\ [d_1,d'_1],\dots,[d_k,d'_k], W, N \text{ coprime}}} \prod_{j=1}^k \frac{\mu(d_j) \mu(d'_j) F_j( \log_x d_j ) G_j( \log_x d'_j)}{[d_j,d'_j]} = (c+o(1)) B^{-k}\frac{N^k}{\phi(N)^k}
\end{equation}
where $B$ was defined in \eqref{bnorm}, and
$$ c := \prod_{j=1}^k \int_0^\infty F'_j(t_j) G'_j(t_j)\ dt_j.$$

The same claim holds if the denominators $[d_j,d'_j]$ are replaced by $\phi([d_j,d'_j])$.
\end{lemma}

Such asymptotics are standard in the literature; see e.g. \cite{gy} for some similar computations.  In older literature, it is common to establish these asymptotics via contour integration (e.g. via Perron's formula), but we will use the Fourier-analytic approach here.  Of course, both approaches ultimately use the same input, namely the simple pole of the Riemann zeta function at $s=1$. 


\begin{proof}  We begin with the first claim.
For $j=1,\dots,k$, the functions $t \mapsto e^t F_j(t)$, $t \mapsto e^t G_j(t)$ may be extended to smooth compactly supported functions on all of $\R$, and so we have Fourier expansions
\begin{equation}\label{etf}
 e^t F_j(t) = \int_\R e^{-it\xi} f_j(\xi)\ d\xi
\end{equation}
and
$$ e^t G_j(t) = \int_\R e^{-it\xi} g_j(\xi)\ d\xi$$
for some fixed functions $f_j, g_j: \R \to \C$ that are smooth and rapidly decreasing in the sense that $f_j(\xi), g_j(\xi) = O( (1+|\xi|)^{-A} )$ for any fixed $A>0$ and all $\xi \in \R$ 
(here the implied constant is independent of $\xi$ and depends only on $A$).  

We may thus write
$$ F_j( \log_x d_j ) = \int_\R \frac{f_j(\xi_j)}{d_j^{\frac{1+i\xi_j}{\log x}}}\ d\xi_j$$
and
$$ G_j( \log_x d'_j ) = \int_\R \frac{g_j(\xi'_j)}{(d'_j)^{\frac{1+i\xi'_j}{\log x}}}\ d\xi'_j$$
for all $d_j,d'_j \geq 1$.
We note that
\begin{align*}
 \sum_{d_j,d_j'}\frac{|\mu(d_j)\mu(d_j')|}{[d_j,d_j']d_j^{1/\log{x}}(d_j')^{1/\log{x}}} &= \prod_{p}\Bigl(1+\frac{2}{p^{1+1/\log{x}}}+\frac{1}{p^{1+2/\log{x}}}\Bigr) \\
&\leq \zeta\left(1+\frac{1}{\log x}\right)^3 \\
&\ll \log^3 x.
\end{align*}
Therefore, if we substitute the Fourier expansions into the left-hand side of \eqref{multisum}, the resulting expression is absolutely convergent. Thus we can apply Fubini's theorem, and the left-hand side of \eqref{multisum} can thus be rewritten as
\begin{equation}\label{fg-int}
 \int_\R \dots \int_\R K(\xi_1,\dots,\xi_k,\xi'_1,\dots,\xi'_k)\ \prod_{j=1}^k f_j(\xi_j) g_j(\xi'_j) d\xi_j d\xi'_j,
\end{equation}
where
$$ K(\xi_1,\dots,\xi_k,\xi'_1,\dots,\xi'_k) := 
\sum_{\substack{d_1, \dots, d_k, d'_1, \dots, d'_k \\ [d_1,d'_1], \dots, [d_k,d'_k], W, N \text{ coprime}}} 
\prod_{j=1}^k \frac{\mu(d_j) \mu(d'_j)}{[d_j,d'_j] d_j^{\frac{1+i\xi_j}{\log x}} (d'_j)^{\frac{1+i\xi'_j}{\log x}}}.$$
This latter expression factorizes as an Euler product 
$$ K = \prod_{p\nmid WN} K_p, $$
where the local factors $K_p$ are given by
\begin{equation}\label{kp}
 K_p(\xi_1,\dots,\xi_k,\xi'_1,\dots,\xi'_k)
 := 1 + \frac{1}{p}
\sum_{\substack{d_1,\dots,d_k,d'_1,\dots,d'_k \\ [d_1, \dots, d_k, d'_1, \dots, d'_k]=p \\ [d_1,d'_1],\dots,[d_k,d'_k] \text{ coprime}}} 
\prod_{j=1}^k \frac{\mu(d_j) \mu(d'_j)}{d_j^{\frac{1+i\xi_j}{\log x}} (d'_j)^{\frac{1+i\xi'_j}{\log x}}}.
\end{equation}
We can estimate each Euler factor as
\begin{equation}\label{kp-est}
 K_p(\xi_1,\dots,\xi_k,\xi'_1,\dots,\xi'_k) = \Bigl(1+O(\frac{1}{p^2})\Bigr) \prod_{j=1}^k \frac{\left(1 - p^{-1-\frac{1+i\xi_j}{\log x}}\right)\left(1 - p^{-1-\frac{1+i\xi'_j}{\log x}}\right)}{1 - p^{-1-\frac{2+i\xi_j+i\xi'_j}{\log x}}}.
\end{equation}
Since 
$$ 
\prod_{p: p>w} \Bigl(1 + O(\frac{1}{p^2})\Bigr) = 1 + o(1),$$
we have
$$
 K(\xi_1,\dots,\xi_k,\xi'_1,\dots,\xi'_k) = (1+o(1)) \prod_{j=1}^k \frac{ \zeta_{W N}( 1 + \frac{2+i\xi_j+i\xi'_j}{\log x}) }{ \zeta_{W N}(1 + \frac{1+i\xi_j}{\log x}) \zeta_{W N}(1 + \frac{1+i\xi'_j}{\log x}) } $$
where the modified zeta function $\zeta_{WN}$ is defined by the formula
$$ \zeta_{W N}(s) := \prod_{p\nmid W N} \left(1-\frac{1}{p^s}\right)^{-1}$$
for $\Re(s) > 1$.  

For $\Re(s) \geq 1 + \frac{1}{\log x}$ we have the crude bounds
\begin{align*}
 |\zeta_{W N}(s)|, |\zeta_{W N}(s)|^{-1} &\leq \zeta( 1 + \frac{1}{\log x}) \\
&\ll \log x
\end{align*}
where the first inequality comes from comparing the factors in the Euler product.
Thus
$$
 K(\xi_1,\dots,\xi_k,\xi'_1,\dots,\xi'_k) = O( \log^{3k} x ).$$
Combining this with the rapid decrease of $f_j, g_j$, we see that the contribution to \eqref{fg-int} outside of the cube $\{\max(|\xi_1|,\dots,|\xi_k|,|\xi'_1|,\dots,|\xi'_k|) \leq \sqrt{\log x}\}$ (say) is negligible.  Thus it will suffice to show that
$$
 \int_{-\sqrt{\log x}}^{\sqrt{\log x}} \dots \int_{-\sqrt{\log x}}^{\sqrt{\log x}} K(\xi_1,\dots,\xi_k,\xi'_1,\dots,\xi'_k)\ \prod_{j=1}^k f_j(\xi_j) g_j(\xi'_j) d\xi_j d\xi'_j = (c+o(1)) B^{-k}\frac{N^k}{\phi(N)^k}.
$$
When $|\xi_j| \leq \sqrt{\log x}$, we see from the simple pole of the Riemann zeta function $\zeta(s) = \prod_p (1-\frac{1}{p^s})^{-1}$ at $s=1$ that
$$ \zeta\left(1 + \frac{1+i\xi_j}{\log x}\right) = (1+o(1)) \frac{\log x}{1+i\xi_j}.$$
For $-\sqrt{\log{x}}\le \xi_j\le \sqrt{\log{x}}$, we see that
$$ 1-\frac{1}{p^{1+\frac{1+i\xi_j}{\log{x}}}}=1-\frac{1}{p}+O\Bigl(\frac{\log{p}}{p\sqrt{\log{x}}}\Bigr).$$
Since $\log(WN)\ll \log^{O(1)}{x}$, this gives
\begin{align*}
 \prod_{p|WN} \Bigl(1 - \frac{1}{p^{1+\frac{1+i\xi_j}{\log x}}}\Bigr) &= \frac{\phi(WN)}{WN}\exp\Bigl(O\Bigl(\sum_{p|WN}\frac{\log{p}}{p\sqrt{\log{x}}}\Bigr)\Bigr) = (1+o(1))\frac{\phi(WN)}{WN},
\end{align*}
since the sum is maximized when $WN$ is composed only of primes $p\ll \log^{O(1)}{x}$. Thus
$$ \zeta_{WN}\Bigl(1 + \frac{1+i\xi_j}{\log x}\Bigr) = \frac{(1+o(1)) B \phi(N)}{(1 + i\xi_j)N}.$$
Similarly with $1+i\xi_j$ replaced by $1+i\xi'_j$ or $2+i\xi_j+i\xi'_j$.  We conclude that
\begin{equation}\label{kt}
 K(\xi_1,\dots,\xi_k,\xi'_1,\dots,\xi'_k) = (1+o(1)) B^{-k}\frac{N^k}{\phi(N)^k} \prod_{j=1}^k \frac{(1+i\xi_j) (1+i\xi'_j)}{2+i\xi_j+i\xi'_j}.
\end{equation}
Therefore it will suffice to show that
$$
 \int_\R \dots \int_\R \prod_{j=1}^k \frac{(1+i\xi_j) (1+i\xi'_j)}{2+i\xi_j+i\xi'_j}  f_j(\xi_j) g_j(\xi'_j) d\xi_j d\xi'_j
= c,$$
since the errors caused by the $1+o(1)$ multiplicative factor in \eqref{kt} or the truncation $|\xi_j|, |\xi'_j| \leq \sqrt{\log x}$ can be seen to be negligible using the rapid decay of $f_j,g_j$.  By Fubini's theorem, it suffices to show that
$$ 
\int_\R \int_\R \frac{(1+i\xi) (1+i\xi')}{2+i\xi+i\xi'} f_j(\xi) g_j(\xi')\ d\xi d\xi' = \int_0^{+\infty} F_j'(t) G_j'(t)\ dt$$
for each $j=1,\dots,k$.
But from dividing \eqref{etf} by $e^t$ and differentiating under the integral sign, we have
$$F'_j(t) = - \int_\R (1+i\xi) e^{-t(1+i\xi)} f_j(\xi)\ d\xi,$$
and the claim then follows from Fubini's theorem.

Finally, suppose that we replace the denominators $[d_j,d'_j]$ with $\phi([d_j,d'_j])$.  An inspection of the above argument shows that the only change that occurs is that the $\frac{1}{p}$ term in \eqref{kp} is replaced by $\frac{1}{p-1}$; but this modification may be absorbed into the $1+O(\frac{1}{p^2})$ factor in \eqref{kp-est}, and the rest of the argument continues as before.
\end{proof}

\subsection{The trivial case}\label{triv-sec}

We can now prove the easiest case of the two theorems, namely case (i) of Theorem \ref{nonprime-asym}; a closely related estimate also appears in \cite[Lemma 6.2]{maynard-new}.  We may assume that $x$ is sufficiently large depending on all fixed quantities. By \eqref{lambdaf-def}, the left-hand side of \eqref{lflg} may be expanded as
\begin{equation}\label{lflg-expand}
 \sum_{d_1,\dots,d_k,d'_1,\dots,d'_k} \left(\prod_{i=1}^k \mu(d_i) \mu(d'_i) F_i(\log_x d_i) G_i(\log_x d'_i)\right) S(d_1,\dots,d_k,d'_1,\dots,d'_k)
\end{equation}
where
$$ S(d_1,\dots,d_k,d'_1,\dots,d'_k) := 
 \sum_{\substack{x \leq n \leq 2x\\ n = b\ (W) \\ n + h_i = 0\ ([d_i,d'_i])\ \forall i\\}} 1.
$$
By hypothesis, $b+h_i$ is coprime to $W$ for all $i=1,\dots,k$, and $|h_i-h_j| < w$ for all distinct $i,j$.  Thus, $S(d_1,\dots,d_k,d'_1,\dots,d'_k)$ vanishes unless the $[d_i,d'_i]$ are coprime to each other and to $W$.  In this case, $S(d_1,\dots,d_k,d'_1,\dots,d'_k)$ is summing the constant function $1$ over an arithmetic progression in $[x,2x]$ of spacing $W [d_1,d'_1] \dots [d_k,d'_k]$, and so  
$$ S(d_1,\dots,d_k,d'_1,\dots,d'_k) = \frac{x}{W [d_1,d'_1] \dots [d_k,d'_k]} + O(1).$$
By Lemma \ref{mul-asym}, the contribution of the main term
$\frac{x}{W [d_1,d'_1] \dots [d_k,d'_k]}$ to \eqref{lflg} is $(c+o(1)) B^{-k} \frac{x}{W}$; note that the restriction of the integrals in \eqref{c-def} to $[0,1]$ instead of $[0,+\infty)$ is harmless since $S(F_i), S(G_i) < 1$ for all $i$.  Meanwhile, the contribution of the $O(1)$ error is then bounded by
$$
O\Bigl( \sum_{d_1,\dots,d_k,d'_1,\dots,d'_k} (\prod_{i=1}^k |F_i(\log_x d_i)| |G_i(\log_x d'_i)|)\Bigr).$$
By the hypothesis in Theorem \ref{nonprime-asym}(i), we see that for $d_1,\dots,d_k,d'_1,\dots,d'_k$ contributing a non-zero term here, one has
$$
[d_1,d'_1] \dots [d_k,d'_k] \lessapprox x^{1-\eps}$$
for some fixed $\eps>0$. From the divisor bound \eqref{divisor-bound} we see that each choice of $[d_1,d'_1] \dots [d_k,d'_k]$ arises from $\lessapprox 1$ choices of $d_1,\dots,d_k,d'_1,\dots,d'_k$.  We conclude that the net contribution of the $O(1)$ error to \eqref{lflg} is $\lessapprox x^{1-\eps}$, and the claim follows.

\subsection{The Elliott-Halberstam case}\label{eh-case}

Now we show case (i) of Theorem \ref{prime-asym}.  For sake of notation we take $i_0=k$, as the other cases are similar.  We use \eqref{lambdaf-def} to rewrite the left-hand side of \eqref{theta-oo} as
\begin{equation}\label{theta-oo2}
 \sum_{d_1,\dots,d_{k-1},d'_1,\dots,d'_{k-1}} \Bigl(\prod_{i=1}^{k-1} \mu(d_i) \mu(d'_i) F_i(\log_x d_i) G_i(\log_x d'_i)\Bigr) \tilde S(d_1,\dots,d_{k-1},d'_1,\dots,d'_{k-1})
\end{equation}
where
$$ \tilde S(d_1,\dots,d_{k-1},d'_1,\dots,d'_{k-1}) := 
 \sum_{\substack{x \leq n \leq 2x\\ n = b\ (W) \\ n + h_i = 0\ ([d_i,d'_i])\ \forall i=1,\dots,k-1}} \theta(n+h_k).
$$
As in the previous case, $\tilde S(d_1,\dots,d_{k-1},d'_1,\dots,d'_{k-1})$ vanishes unless the $[d_i,d'_i]$ are coprime to each other and to $W$, and so the summand in \eqref{theta-oo2} vanishes unless the modulus $q_{W,d_1,\dots,d'_{k-1}}$ defined by
\begin{equation}\label{q-def}
q_{W,d_1,\dots,d'_{k-1}} := W [d_1,d'_1] \dots [d_{k-1},d'_{k-1}]
\end{equation}
is squarefree.  In that case, we may use the Chinese remainder theorem to concatenate the congruence conditions  on $n$ into a single primitive congruence condition  
$$n+h_k = a_{W,d_1,\dots,d'_{k-1}} \ (q_{W,d_1,\dots,d'_{k-1}})$$
for some $a_{W,d_1,\dots,d'_{k-1}}$ depending on $W, d_1,\dots,d_{k-1},d'_1,\dots,d'_{k-1}$, and conclude using \eqref{disc-def} that
\begin{equation}\label{ts}
\begin{split}
 \tilde S(d_1,\dots,d_{k-1},d'_1,\dots,d'_{k-1}) &= \frac{1}{\phi(q_{W,d_1,\dots,d'_{k-1}})} \sum_{x+h_k \leq n \leq 2x+h_k} \theta(n)\\
&\quad + \Delta( \onef_{[x+h_k,2x+h_k]} \theta;  a_{W,d_1,\dots,d'_{k-1}}\ (q_{W,d_1,\dots,d'_{k-1}})).
\end{split}
\end{equation}
From the prime number theorem we have
$$ \sum_{x+h_k \leq n \leq 2x+h_k} \theta(n) = (1+o(1)) x$$
and this expression is clearly independent of $d_1,\dots,d'_{k-1}$.  Thus by Lemma \ref{mul-asym}, the contribution of the main term in \eqref{ts} to \eqref{theta-oo2} is $(c+o(1)) B^{1-k} \frac{x}{\phi(W)}$.  By \eqref{W-bound} and \eqref{bnorm}, it thus suffices to show that for any fixed $A$ we have
\begin{equation}\label{sosmall}
 \sum_{d_1,\dots,d_{k-1},d'_1,\dots,d'_{k-1}} \Bigl(\prod_{i=1}^{k-1} |F_i(\log_x d_i)| |G_i(\log_x d'_i)|\Bigr) 
|\Delta( \onef_{[x+h_k,2x+h_k]} \theta;  a\ (q))| \ll x \log^{-A} x,
\end{equation}
where $a=a_{W,d_1,\dots,d'_{k-1}}$ and $q=q_{W,d_1,\dots,d'_{k-1}}$.  For future reference we note that we may restrict the summation here to those $d_1,\dots,d'_{k-1}$ for which $q_{W,d_1,\dots,d'_{k-1}}$ is square-free.

From the hypotheses of Theorem \ref{prime-asym}(i), we have
$$ q_{W,d_1,\dots,d'_{k-1}} \lessapprox x^\vartheta$$
whenever the summand in \eqref{theta-oo2} is non-zero, and each choice $q$ of $q_{W,d_1,\dots,d'_{k-1}}$ is associated to $O( \tau(q)^{O(1)} )$ choices of $d_1,\dots,d_{k-1},d'_1,\dots,d'_{k-1}$.  Thus this contribution is
$$
\ll \sum_{q \lessapprox x^\vartheta} \tau(q)^{O(1)} \sup_{a \in (\Z/q\Z)^\times} |\Delta( \onef_{[x+h_k,2x+h_k]} \theta; a\ (q) )|.$$
Using the crude bound
$$ |\Delta( \onef_{[x+h_k,2x+h_k]} \theta; a\ (q) )| \ll \frac{x}{q} \log^{O(1)} x$$
and \eqref{divisor-2}, we have
$$ \sum_{q \lessapprox x^\vartheta} \tau(q)^C \sup_{a \in (\Z/q\Z)^\times} |\Delta( \onef_{[x+h_k,2x+h_k]} \theta; a\ (q) )| \ll x \log^{O(1)} x$$
for any fixed $C>0$. By the Cauchy-Schwarz inequality it suffices to show that
$$ \sum_{q \lessapprox x^\vartheta} \sup_{a \in (\Z/q\Z)^\times} |\Delta( \onef_{[x+h_k,2x+h_k]} \theta; a\ (q) )| \ll x \log^{-A} x$$
for any fixed $A>0$.  However, since $\theta$ only differs from $\Lambda$ on powers $p^j$ of primes with $j>1$, it is not difficult to show that
$$  |\Delta( \onef_{[x+h_k,2x+h_k]} \theta; a\ (q) ) -  \Delta( \onef_{[x+h_k,2x+h_k]} \Lambda; a\ (q) )| \lessapprox \sqrt{\frac{x}{q}}, $$
so the net error in replacing $\theta$ here by $\Lambda$ is $\lessapprox x^{1 - (1-\vartheta)/2}$, which is certainly acceptable.  The claim now follows from the hypothesis $\EH[\vartheta]$, thanks to Claim \ref{eh-def}.

\subsection{The Motohashi-Pintz-Zhang case}

Now we show case (ii) of Theorem \ref{prime-asym}.  We repeat the arguments from Section \ref{eh-case}, with the only difference being in the derivation of \eqref{sosmall}.  As observed previously, we may restrict $q_{W,d_1,\dots,d'_{k-1}}$ to be squarefree.  From the hypotheses in Theorem \ref{prime-asym}(ii), we also see that
$$ q_{W,d_1,\dots,d'_{k-1}} \lessapprox x^{1/2+2\varpi}$$
and that all the prime factors of $q_{W,d_1,\dots,d'_{k-1}}$ are at most $x^\delta$.  Thus, if we set $I := [1,x^\delta]$, we see (using the notation from Claim \ref{mpz-claim}) that $q_{W,d_1,\dots,d'_{k-1}}$ lies in $\Scal_I$, and is thus a factor of $P_I$.  If we then let ${\mathcal A} \subset \Z/P_I\Z$ denote all the primitive residue classes $a\ (P_I)$ with the property that $a = b\ (W)$, and such that for each prime $w < p \leq x^\delta$, one has $a + h_i = 0\ (p)$ for some $i=1,\dots,k$, then we see that $a_{W,d_1,\dots,d'_{k-1}}$ lies in the projection of ${\mathcal A}$ to $\Z/q_{W,d_1,\dots,d'_{k-1}}\Z$.  Each $q \in \Scal_I$ is equal to $q_{W,d_1,\dots,d'_{k-1}}$ for $O(\tau(q)^{O(1)})$ choices of $d_1,\dots,d'_{k-1}$.  Thus the left-hand side of \eqref{sosmall} is
$$ \ll \sum_{q \in \Scal_I: q \lessapprox x^{1/2+2\varpi}} \tau(q)^{O(1)} \sup_{a \in {\mathcal A}}
|\Delta( \onef_{[x+h_k,2x+h_k]} \theta;  a\ (q))|.$$
Note from the Chinese remainder theorem that for any given $q$, if one lets $a$ range uniformly in ${\mathcal A}$, then $a\ (q)$ is uniformly distributed among $O( \tau(q)^{O(1)})$ different moduli.  Thus we have
$$
\sup_{a \in {\mathcal A}}
|\Delta( \onef_{[x+h_k,2x+h_k]} \theta;  a\ (q))| \ll \frac{\tau(q)^{O(1)}}{|{\mathcal A}|} \sum_{a \in {\mathcal A}} |\Delta( \onef_{[x+h_k,2x+h_k]} \theta;  a\ (q))|, $$
and so it suffices to show that
$$
\sum_{q \in \Scal_I: q \lessapprox x^{1/2+2\varpi}} \frac{\tau(q)^{O(1)}}{|{\mathcal A}|} \sum_{a \in {\mathcal A}}
|\Delta( \onef_{[x+h_k,2x+h_k]} \theta;  a\ (q))|\ll x \log^{-A} x$$
for any fixed $A>0$.  We see it suffices to show that
$$
\sum_{q \in \Scal_I: q \lessapprox x^{1/2+2\varpi}} \tau(q)^{O(1)}
|\Delta( \onef_{[x+h_k,2x+h_k]} \theta;  a\ (q))|\ll x \log^{-A} x$$
for any given $a \in {\mathcal A}$.  But this follows from the hypothesis $MPZ[\varpi,\delta]$ by repeating the arguments of Section \ref{eh-case}.

\subsection{Crude estimates on divisor sums}

To proceed further, we will need some additional information on the divisor sums $\lambda_F$ (defined in \eqref{lambdaf-def}), namely that these sums are concentrated on ``almost primes''; results of this type have also appeared in \cite{pintz-szemeredi}.

\begin{proposition}[Almost primality]\label{almostprime} Let $k \geq 1$ be fixed, let $(h_1,\dots,h_k)$ be a fixed admissible $k$-tuple, and let $b\ (W)$ be such that $b+h_i$ is coprime to $W$ for each $i=1,\dots,k$.  Let $F_1,\dots,F_k: [0,+\infty) \to \R$ be fixed smooth compactly supported functions, and let $m_1,\dots,m_k \geq 0$ and $a_1,\dots,a_k \geq 1$ be fixed natural numbers.  Then
\begin{equation}\label{lambdatau}
\sum_{x \leq n \leq 2x: n = b\ (W)} \prod_{j=1}^k \Bigl(|\lambda_{F_j}(n+h_j)|^{a_j} \tau(n+h_j)^{m_j}\Bigr) \ll B^{-k} \frac{x}{W}.
\end{equation}
Furthermore, if $1 \leq j_0 \leq k$ is fixed and $p_0$ is a prime with $p_0 \leq x^{\frac{1}{10k}}$, then we have the variant
\begin{equation}\label{lambdatau-fix}
\sum_{x \leq n \leq 2x: n = b\ (W)} \prod_{j=1}^k \Bigl(|\lambda_{F_j}(n+h_j)|^{a_j} \tau(n+h_j)^{m_j}\Bigr) \onef_{p_0|n+h_{j_0}} \ll \frac{\log_x p_0}{p_0} B^{-k} \frac{x}{W}.
\end{equation}
As a consequence, we have
\begin{equation}\label{lambdatau-fix2}
\sum_{x \leq n \leq 2x: n = b\ (W)} \prod_{j=1}^k \Bigl(|\lambda_{F_j}(n+h_j)|^{a_j} \tau(n+h_j)^{m_j}\Bigr) \onef_{p(n+h_{j_0}) \leq x^\eps} \ll \eps B^{-k} \frac{x}{W},
\end{equation}
for any $\eps > 0$, where $p(n)$ denotes the least prime factor of $n$.
\end{proposition}

The exponent $\frac{1}{10k}$ can certainly be improved here, but for our purposes any fixed positive exponent depending only on $k$ will suffice.

\begin{proof}  The strategy is to estimate the alternating divisor sums $\lambda_{F_j}(n+h_j)$ by non-negative expressions involving prime factors of $n+h_j$, which can then be bounded combinatorially using standard tools.

We first prove \eqref{lambdatau}.
As in the proof of Proposition \ref{mul-asym}, we can use Fourier expansion to write
$$ F_j( \log_x d ) = \int_\R \frac{f_j(\xi)}{d^{\frac{1+i\xi}{\log x}}}\ d\xi$$
for some rapidly decreasing $f_j: \R \to \C$ and all natural numbers $d$.  Thus
$$ \lambda_{F_j}(n) = \int_\R \Bigl(\sum_{d|n} \frac{\mu(d)}{d^{\frac{1+i\xi}{\log x}}}\Bigr) f_j(\xi)\ d\xi, $$
which factorizes using Euler products as
$$ \lambda_{F_j}(n) = \int_\R \prod_{p|n} \Bigl(1 - \frac{1}{p^{\frac{1+i\xi}{\log x}}}\Bigr) f_j(\xi)\ d\xi.$$
The function $s \mapsto p^{\frac{-s}{\log x}}$ has a magnitude of $O(1)$ and a derivative of $O( \log_x p )$ when $\Re(s) > 1$, and thus
$$ 1 - \frac{1}{p^{\frac{1+i\xi}{\log x}}} = O\Bigl( \min( (1+|\xi|) \log_x p, 1) \Bigr).$$
From the rapid decrease of $f_j$ and the triangle inequality, we conclude that
$$ |\lambda_{F_j}(n)| \ll \int_\R \Bigl(\prod_{p|n} O\Bigl( \min( (1+|\xi|) \log_x p, 1)\Bigr)\Bigr) \frac{d\xi}{(1+|\xi|)^A}$$
for any fixed $A > 0$. Thus, noting that $\prod_{p|n}O(1)\ll \tau(n)^{O(1)}$, we have
$$ |\lambda_{F_j}(n)|^{a_j} \ll \tau(n)^{O(1)}\int_\R \dots \int_\R \Bigl(\prod_{p|n}  \prod_{l=1}^{a_j} \min( (1+|\xi_l|) \log_x p, 1)\Bigr) \frac{d\xi_1 \dots d\xi_{a_j}}{(1+|\xi_1|)^A \dots (1+|\xi_{a_j}|)^A}
$$
for any fixed $a_j,A$. However, we have
$$ \prod_{i=1}^{a_j}\min( (1+|\xi_i|) \log_x p, 1)) \leq \min( (1+|\xi_1|+\dots+|\xi_{a_j}|) \log_x p, 1 ), $$
and so
$$ |\lambda_{F_j}(n)|^{a_j} \ll \tau(n)^{O(1)}\int_\R \dots \int_\R \frac{(\prod_{p|n} \min( (1+|\xi_1|+\dots+|\xi_{a_j}|) \log_x p, 1)) d\xi_1 \dots d \xi_{a_j}}{(1+|\xi_1|+\dots+|\xi_{a_j}|)^A}.$$
Making the change of variables $\sigma := 1+|\xi_1|+\dots+|\xi_{a_j}|$, we obtain
$$ |\lambda_{F_j}(n)|^{a_j} \ll \tau(n)^{O(1)} \int_1^\infty \Bigl(\prod_{p|n} \min(\sigma \log_x p, 1)\Bigr) \frac{d\sigma}{\sigma^A}$$
for any fixed $A>0$.  In view of this bound and the Fubini-Tonelli theorem, it suffices to show that
$$
\sum_{x \leq n \leq 2x: n = b\ (W)} \prod_{j=1}^k \Bigl(\tau(n+h_j)^{O(1)} \prod_{p|n+h_j} \min(\sigma_j \log_x p,1)\Bigr) \ll B^{-k} \frac{x}{W} (\sigma_1+\dots+\sigma_k)^{O(1)}$$
for all $\sigma_1,\dots,\sigma_k \geq 1$. By setting $\sigma := \sigma_1+\dots+\sigma_k$, it suffices to show that
\begin{equation}\label{xnx}
\sum_{x \leq n \leq 2x: n = b\ (W)} \prod_{j=1}^k \Bigl(\tau(n+h_j)^{O(1)} \prod_{p|n+h_j} \min(\sigma \log_x p,1)\Bigr) \ll B^{-k} \frac{x}{W} \sigma^{O(1)}
\end{equation}
for any $\sigma \geq 1$.

To proceed further, we factorize $n+h_j$ as a product
$$ n+h_j = p_1 \dots p_r$$
of primes $p_1 \leq \dots \leq p_r$ in increasing order, and then write 
$$ n+h_j = d_j m_j$$
where $d_j := p_1 \dots p_{i_j}$ and $i_j$ is the largest index for which $p_1 \dots p_{i_j} < x^{\frac{1}{10k}}$, and $m_j := p_{i_j+1} \dots p_r$.  By construction, we see that $0 \leq i_j < r$, $d_j \leq x^{\frac{1}{10k}}$.  Also, we have
$$ p_{i_j+1} \geq (p_1 \dots p_{i_j+1})^{\frac{1}{i_j+1}} \geq x^{\frac{1}{10k(i_j+1)}}.$$
Since $n \leq 2x$, this implies that
$$ r = O( i_j + 1 )$$
and so
$$ \tau(n+h_j) \leq 2^{O(1+\Omega(d_j))},$$
where we recall that $\Omega(d_j) = i_j$ denotes the number of prime factors of $d_j$, counting multiplicity. We also see that
$$ p(m_j) \geq x^{\frac{1}{10k(1+\Omega(d_j))}}\geq x^{\frac{1}{10k(1+\Omega(d_1\dots d_k))}}=:R,$$
where $p(n)$ denotes the least prime factor of $n$. Finally, we have that
$$ \prod_{p|n+h_j} \min(\sigma \log_x p,1) \leq \prod_{p|d_j} \min( \sigma \log_x p, 1 ),$$
and we see the $d_1,\dots,d_k,W$ are coprime.  We may thus estimate the left-hand side of \eqref{xnx} by
$$ 
\ll \sum_* \Bigl(\prod_{j=1}^k 2^{O(1+\Omega(d_j))} \prod_{p|d_j} \min( \sigma \log_x p, 1 )\Bigr) \sum_{**} 1$$
where the outer sum $\sum_*$ is over $d_1,\dots,d_k \leq x^{\frac{1}{10k}}$ with $d_1,\dots,d_k,W$ coprime, and the inner sum $\sum_{**}$ is over $x \leq n \leq 2x$ with $n = b\ (W)$ and $n + h_j = 0\ (d_j)$ for each $j$, with $p( \frac{n+h_j}{d_j}) \geq R$ for each $j$.

We bound the inner sum $\sum_{**} 1$  using a Selberg sieve upper bound. Let $G$ be a smooth  function supported on $[0,1]$ with $G(0)=1$, and let $d=d_1\dots d_k$. We see that
$$\sum_{**} 1\le \sum_{\substack{x\le n\le 2x\\ n+h_i = 0\ (d_i) \\ n\equiv b\ (W)}}\prod_{i=1}^k\Bigl(\sum_{\substack{e|n+h_i\\ (e,dW)=1}}\mu(e)G(\log_R{e})\Bigr)^2,$$
since the product is $G(0)^{2k}=1$ if $p( \frac{n+h_j}{d_j}) \geq R$, and non-negative otherwise. The right hand side may be expanded as
$$ \sum_{\substack{e_1,\dots,e_k,e_1',\dots,e_k'\\ (e_ie_i',dW)=1\forall i}}\Bigl(\prod_{i=1}^k\mu(e_i)\mu(e_i')G(\log_R{e_i})G(\log_R{e_i'})\Bigr)\sum_{\substack{x\le n\le 2x \\ n+h_i = 0\ (d_i[e_i,e_i']) \\ n= b\ (W)}}1.$$
As in Section \ref{triv-sec}, the inner sum vanishes unless the $e_ie_i'$ are coprime to each other and $dW$, in which case it is
$$\frac{x}{dW[e_1,e_1']\dots [e_k,e_k']}+O(1).$$
The $O(1)$ term contributes $\lessapprox R^k\lessapprox x^{1/10}$, which is negligible. By Lemma \ref{mul-asym}, if $\Omega(d)\ll \log^{1/2}{x}$ then the main term contributes
$$\ll \Bigl(\frac{d}{\phi(d)}\Bigr)^k\frac{x}{d W}(\log{R})^{-k}\ll 2^{\Omega(d)}B^{-k}\frac{x}{d W}.$$
We see that this final bound applies trivially if $\Omega(d)\gg \log^{1/2}{x}$. The bound \eqref{xnx} thus reduces to
\begin{equation}\label{xnx-2}
\sum_* \Bigl(\prod_{j=1}^k \frac{2^{O(1+\Omega(d_j))}}{d_j} \prod_{p|d_j} \min( \sigma \log_x p, 1 )\Bigr) \ll \sigma^{O(1)}.
\end{equation}
Ignoring the coprimality conditions on the $d_j$ for an upper bound, we see this is bounded by
$$ \prod_{w<p\le x^{\frac{1}{10k}}}\Bigl(1+\frac{O(\min(\sigma\log_x(p),1))}{p}\sum_{j\ge 0}\frac{O(1)^j}{p^j}\Bigr)^k \ll \exp\Bigl(O\Bigl(\sum_{p\le x}\frac{(\min(\sigma\log_x(p),1))}{p}\Bigr)\Bigr).$$
But from Mertens' theorem we have
$$ \sum_{p \leq x} \frac{\min(\sigma \log_x p, 1)}{p} = O\Bigl( \log \frac{1}{\sigma} \Bigr),$$
and the claim \eqref{lambdatau} follows.

The proof of \eqref{lambdatau-fix} is a minor modification of the argument above used to prove \eqref{lambdatau}.  
Namely, the variable $d_{j_0}$ is now replaced by $[d_0,p_0]<x^{1/5k}$, which upon factoring out $p_0$ has the effect of multiplying the upper bound for \eqref{xnx-2} by $O( \frac{\sigma \log_x p_0}{p_0})$ (at the negligible cost of deleting the prime $p_0$ from the sum $\sum_{p \leq x}$), giving the claim; we omit the details.

Finally, \eqref{lambdatau-fix2} follows immediately from \eqref{lambdatau} when $\eps > \frac{1}{10k}$, and from \eqref{lambdatau-fix} and Mertens' theorem when $\eps \leq \frac{1}{10k}$.
\end{proof}

\begin{remark} As in \cite{pintz-szemeredi}, one can use Proposition \ref{almostprime}, together with the observation that the quantity $\lambda_F(n)$ is bounded whenever $n = O(x)$ and $p(n) \geq x^\eps$, to conclude that whenever the hypotheses of Lemma \ref{crit} are obeyed for some $\nu$ of the form \eqref{nuform}, then there exists a fixed $\eps>0$ such that for all sufficiently large $x$, there are $\gg \frac{x}{\log^k x}$ elements $n$ of $[x,2x]$ such that $n+h_1,\dots,n+h_k$ have no prime factor less than $x^\eps$, and that at least $m$ of the $n+h_1,\dots,n+h_k$ are prime.
\end{remark}

\subsection{The generalized Elliott-Halberstam case}\label{geh-case}

Now we show case (ii) of Theorem \ref{nonprime-asym}.  For sake of notation we shall take $i_0=k$, as the other cases are similar; thus we have
\begin{equation}\label{ik1}
\sum_{i=1}^{k-1} (S(F_{i}) + S(G_{i})) < \vartheta.
\end{equation}

The basic idea is to view the sum \eqref{lflg} as a variant of \eqref{theta-oo}, with the role of the function $\theta$ now being played by the product divisor sum $\lambda_{F_k} \lambda_{G_k}$, and to repeat the arguments in Section \ref{eh-case}.  To do this we rely on Proposition \ref{almostprime} to restrict $n+h_i$ to the almost primes.

We turn to the details.  Let $\eps > 0$ be an arbitrary fixed quantity.  From \eqref{lambdatau-fix2} and Cauchy-Schwarz one has
$$
 \sum_{\substack{x \leq n \leq 2x\\ n = b\ (W)}} \Bigl(\prod_{i=1}^k \lambda_{F_{i}}(n+h_{i}) \lambda_{G_{i}}(n+h_{i})\Bigr) \onef_{p(n+h_k) \leq x^\eps} = O\left( \eps B^{-k} \frac{x}{W} \right)
$$
with the implied constant uniform in $\eps$,
so by the triangle inequality and a limiting argument as $\eps \to 0$ it suffices to show that
\begin{equation}\label{lltrunc}
 \sum_{\substack{x \leq n \leq 2x\\ n = b\ (W)}} \Bigl( \prod_{i=1}^k \lambda_{F_{i}}(n+h_{i}) \lambda_{G_{i}}(n+h_{i}) \Bigr)\onef_{p(n+h_k) > x^\eps} = (c_\eps + o(1)) B^{-k} \frac{x}{W} 
\end{equation}
where $c_\eps$ is a quantity depending on $\eps$ but not on $x$, such that
$$ \lim_{\eps \to 0} c_\eps = \prod_{i=1}^k \int_0^1 F'_i(t) G'_i(t)\ dt.$$

We use \eqref{lambdaf-def} to expand out $\lambda_{F_{i}}, \lambda_{G_{i}}$ for $i=1,\dots,k-1$, but \emph{not} for $i=k$, so that the left-hand side of \eqref{lflg} becomes
\begin{equation}\label{ddd}
 \sum_{d_1,\dots,d_{k-1},d'_1,\dots,d'_{k-1}} \Bigl(\prod_{i=1}^k \mu(d_{i}) \mu(d'_{i}) F_{i}(\log_x d_{i}) G_{i}(\log_x d'_{i}) \Bigr) S'(d_1,\dots,d_{k-1},d'_1,\dots,d'_{k-1})
\end{equation}
where
$$ S'(d_1,\dots,d_{k-1},d'_1,\dots,d'_{k-1}) := 
 \sum_{\substack{x \leq n \leq 2x\\ n = b\ (W) \\ n + h_{i} = 0\ ([d_{i},d'_{i}])\ \forall i=1,\dots,k-1}}
\hspace{0pt minus 1fil}
 \lambda_{F_k}(n+h_k) \lambda_{G_k}(n+h_k) \onef_{p(n+h_k) 
> x^\eps}.
$$
As before, the summand in \eqref{ddd} vanishes unless the modulus\footnote{In the $k=1$ case, we of course just have $q_{W,d_1,\dots,d'_{k-1}} = W$.}
$q_{W,d_1,\dots,d'_{k-1}}$ defined in \eqref{q-def} is squarefree, in which case we have the analogue
\begin{align}\label{tsp}
 S'(d_1,\dots,d_{k-1},d'_1,\dots,d'_{k-1}) &= \frac{1}{\phi(q)} \sum_{\substack{x+h_k \leq n \leq 2x+h_k\\ (n,q)=1}} \lambda_{F_k}(n) \lambda_{G_k}(n) \onef_{p(n)>x^\eps}\nonumber\\
& + \Delta( \onef_{[x+h_k,2x+h_k]} \lambda_{F_k} \lambda_{G_k} \onef_{p(\cdot) > x^\eps};  a\ (q))
\end{align}
of \eqref{ts}. Here we have put $q=q_{W,d_1,\dots,d'_{k-1}}$ and $a=a_{W,d_1,\dots,d'_{k-1}}$ for convenience. We thus split
$$ S' = S'_1 - S'_2 + S'_3,$$
where,
\begin{align}
 S'_1(d_1,\dots,d_{k-1},d'_1,\dots,d'_{k-1}) &= \frac{1}{\phi(q)} \sum_{x+h_k \leq n \leq 2x+h_k} \lambda_{F_k}(n) \lambda_{G_k}(n) \onef_{p(n)>x^\eps} 
\label{sp1-def},\\
 S'_2(d_1,\dots,d_{k-1},d'_1,\dots,d'_{k-1}) &= \frac{1}{\phi(q)} \sum_{x+h_k \leq n \leq 2x+h_k; (n,q) > 1} \lambda_{F_k}(n) \lambda_{G_k}(n) \onef_{p(n)>x^\eps} 
\label{sp2-def},\\
 S'_3(d_1,\dots,d_{k-1},d'_1,\dots,d'_{k-1}) &= \Delta( \onef_{[x+h_k,2x+h_k]} \lambda_{F_k} \lambda_{G_k} \onef_{p(\cdot) > x^\eps};  a\ (q)),
\label{sp3-def}
\end{align}
when $q=q_{W,d_1,\dots,d'_{k-1}}$ is squarefree, with $S'_1=S'_2=S'_3=0$ otherwise.

For $j\in\{1,2,3\}$, let
\begin{equation}
 \Sigma_j= \sum_{d_1,\dots,d_{k-1},d'_1,\dots,d'_{k-1}}\Bigl(\prod_{i=1}^k \mu(d_{i}) \mu(d'_{i}) F_{i}(\log_x d_{i}) G_{i}(\log_x d'_{i})\Bigr) S'_j(d_1,\dots,d_{k-1},d'_1,\dots,d'_{k-1}).
\label{eq:SigmaDef}
\end{equation}
To show \eqref{lltrunc}, it thus suffices to show the main term estimate
\begin{equation}\label{ll-main}
 \Sigma_1= (c_\eps + o(1)) B^{-k} \frac{x}{W},
\end{equation}
the first error term estimate
\begin{equation}\label{ll-error1}
\Sigma_2 \lessapprox x^{1-\eps},
\end{equation}
and the second error term estimate
\begin{equation}\label{ll-error2}
\Sigma_3 \ll x \log^{-A} x
\end{equation}
for any fixed $A>0$.

We begin with \eqref{ll-error1}.  Observe that if $p(n) > x^\eps$, then the only way that $(n,q_{W,d_1,\dots,d'_{k-1}})$ can exceed $1$ is if there is a prime $x^\eps < p \ll x$ which divides both $n$ and one of $d_1,\dots,d'_{k-1}$; in particular, this case can only occur when $k > 1$.  For sake of notation we will just consider the contribution when there is a prime that divides $n$ and $d_1$, as the other $2k-3$ cases are similar.  By \eqref{sp2-def}, this contribution to $\Sigma_2$ can then be crudely bounded (using \eqref{divisor-bound}) by
\begin{align*}
 \Sigma_2&\lessapprox \sum_{x^\eps < p \ll x} \sum_{d_1,\dots,d_{k-1},d'_1,\dots,d'_{k-1} \leq x; p|d_1} \frac{1}{[d_1,d'_1] \dots [d_{k-1},d'_{k-1}]} \sum_{n \ll x: p|n} 1 \\
&\lessapprox \sum_{x^\eps < p \ll x} \frac{x}{p} \Bigl(\sum_{e_1 \leq x^2; p|e_1} \frac{\tau(e_1)}{e_1}\Bigr) \prod_{i=2}^{k-1} \Bigl(\sum_{e_i \leq x^2} \frac{\tau(e_i)}{e_i}\Bigr) \\
&\lessapprox \sum_{x^\eps < p \ll x} \frac{x}{p^2} \\
&\lessapprox x^{1-\eps}
\end{align*}
as required, where we have made the change of variables $e_i := [d_i,d'_i]$, using the divisor bound to control the multiplicity.  

Now we show \eqref{ll-error2}.  From the hypothesis \eqref{eh-upper} we have $q_{W,d_1,\dots,d'_{k-1}} \lessapprox x^\vartheta$ whenever the summand in \eqref{ll-error2} is non-zero.  From the divisor bound, for each $q \lessapprox x^\vartheta$ there are $O( \tau(q)^{O(1)})$ choices of $d_1,\dots,d'_{k-1}$ with $q_{W,d_1,\dots,d'_{k-1}} = q$. We see the product in \eqref{eq:SigmaDef} is $O(1)$. Thus by \eqref{sp3-def}, we may bound $\Sigma_3$ by
$$
\Sigma_3 \ll \sum_{q \lessapprox x^\vartheta} \tau(q)^{O(1)} \sup_{a \in (\Z/q\Z)^\times} |\Delta( \onef_{[x+h_k,2x+h_k]} \lambda_{F_k} \lambda_{G_k} \onef_{p(\cdot) > x^\eps};  a\ (q))|.
$$
From \eqref{divisor-2} we easily obtain the bound
$$
\Sigma_3\ll \sum_{q \lessapprox x^\vartheta} \tau(q)^{O(1)} \sup_{a \in (\Z/q\Z)^\times} |\Delta( \onef_{[x+h_k,2x+h_k]} \lambda_{F_k} \lambda_{G_k} \onef_{p(\cdot) > x^\eps};  a\ (q))| \ll x \log^{O(1)} x,
$$
so by Cauchy-Schwarz it suffices to show that
\begin{equation}\label{local}
\sum_{q \lessapprox x^\vartheta} \sup_{a \in (\Z/q\Z)^\times} |\Delta( \onef_{[x+h_k,2x+h_k]} \lambda_{F_k} \lambda_{G_k} \onef_{p(\cdot) > x^\eps};  a\ (q))| \ll x \log^{-A} x
\end{equation}
for any fixed $A>0$.

If we had the additional hypothesis $S(F_k) + S(G_k) < 1$, then this would follow easily from the hypothesis $\GEH[\vartheta]$ thanks to Claim \ref{geh-def}, since one can write $\lambda_{F_k} \lambda_{G_k} \onef_{p(\cdot) > x^\eps} = \alpha \star \beta$ with
$$ \alpha(n) := \onef_{p(n) > x^\eps} \sum_{d,d': [d,d'] = n} \mu(d) F_k(\log_x d) \mu(d') G_k(\log_x d')$$
and
$$ \beta(n) := \onef_{p(n) > x^\eps}.$$
But even in the absence of the hypothesis $S(F_k) + S(G_k) < 1$, we can still invoke $\GEH[\vartheta]$ after appealing to the fundamental theorem of arithmetic.  Indeed, if $n \in [x+h_k,2x+h_k]$ with $p(\cdot) > \eps$, then we have
$$ n = p_1 \dots p_r$$
for some primes $x^\eps < p_1 \leq \dots \leq p_r \leq 2x+h_k$, which forces $r \leq \frac{1}{\eps}+1$.  If we then partition $[x^\eps,2x+h_k]$ by $O( \log^{A+1} x )$ intervals $I_1,\dots,I_m$, with each $I_j$ contained in an interval of the form $[N, (1+\log^{-A} x) N]$, then we have $p_i \in I_{j_i}$ for some $1 \leq j_1 \leq \dots \leq j_r \leq m$, with the product interval $I_{j_1} \cdot \dots \cdot I_{j_r}$ intersecting $[x+h_k, 2x+h_k]$.  For fixed $r$, there are $O( \log^{Ar+r} x)$ such tuples $(j_1,\dots,j_r)$, and a simple application of the prime number theorem with classical error term (and crude estimates on the discrepancy $\Delta$) shows that each tuple contributes $O( x \log^{-Ar+O(1)} x)$ to \eqref{local} (here, and for the rest of this section, implied constants will be independent of $A$ unless stated otherwise). In particular, the $O(\log^{A(r-1)}{x})$ tuples $(j_1,\dots,j_r)$ with one repeated $j_i$, or for which the interval $I_{j_1} \cdot \dots \cdot I_{j_r}$ meets the boundary of $[x+h_k, 2x+h_k]$, contribute a total of $O(\log^{-A+O(1)}{x})$. This is an acceptable error to \eqref{local}, and so these tuples may be removed.  Thus it suffices to show that 
$$
\sum_{q \lessapprox x^\vartheta} \sup_{a \in (\Z/q\Z)^\times} |\Delta( \lambda_{F_k} \lambda_{G_k} \onef_{A_{j_1,\dots,j_r}};  a\ (q))| \ll x \log^{-A(r+1)+O(1)} x
$$
for any $1 \leq r \leq \frac{1}{\eps}+1$ and $1 \leq j_1 < \dots < j_r \leq m$ with $I_{j_1} \cdot \dots \cdot I_{j_r}$ contained in $[x+h_k,x+2h_k]$, where $A_{j_1,\dots,j_r}$ is the set of all products $p_1 \dots p_r$ with $p_i \in I_{j_i}$ for $i=1,\dots,r$, and where we allow implied constants in the $\ll$ notation to depend on $\eps$.  But for $n$ in $A_{j_1,\dots,j_r}$, the $2^r$ factors of $n$ are just the products of subsets of $\{p_1,\dots,p_r\}$, and from the smoothness of $F_k,G_k$ we see that $\lambda_{F_k}(n)$ is equal to some bounded constant (depending on $j_1,\dots,j_r$, but independent of $p_1,\dots,p_r$), plus an error of $O(\log^{-A} x)$.  As before, the contribution of this error is $O( \log^{-A(r+1)+O(1)} x)$, so it suffices to show that
$$
\sum_{q \lessapprox x^\vartheta} \sup_{a \in (\Z/q\Z)^\times} |\Delta( \onef_{A_{j_1,\dots,j_r}};  a\ (q))| \ll x \log^{-A(r+1)+O(1)} x.
$$
But one can write $\onef_{A_{j_1,\dots,j_r}}$ as a convolution $\onef_{A_{j_1}} \star \dots \star \onef_{A_{j_r}}$, where $A_{j_i}$ denotes the primes in $I_{j_i}$; assigning $A_{j_r}$ (for instance) to be $\beta$ and the remaining portion of the convolution to be $\alpha$, the claim now follows from the hypothesis $\GEH[\vartheta]$, thanks to the Siegel-Walfisz theorem (see e.g. \cite[Satz 4]{siebert} or~\cite[Th. 5.29]{ik}).

Finally, we show \eqref{ll-main}.  By Lemma \ref{mul-asym} we have
$$
 \sum_{\substack{ d_1,\dots,d_{k-1},d'_1,\dots,d'_{k-1}\\ d_1d'_1,\dots,d_{k-1}d'_{k-1}, W \text{ coprime}}} \hspace{0pt minus 1fil}\frac{\prod_{i=1}^{k-1} \mu(d_{i}) \mu(d'_{i}) F_{i}(\log_x d_{i}) G_{i}(\log_x d'_{i}) }{\phi(q_{W,d_1,\dots,d'_{k-1}})} = \frac{1}{\phi(W)} (c'+o(1)) B^{-k+1},$$
where
$$ c' := \prod_{i=1}^{k-1} \int_0^1 F'_i(t) G'_i(t)\ dt$$
(note that $F_i, G_i$ are supported on $[0,1]$ by hypothesis), so by \eqref{sp1-def} it suffices to show that
\begin{equation}\label{cpeps}
\sum_{x+h_k \leq n \leq 2x+h_k} \lambda_{F_k}(n) \lambda_{G_k}(n) \onef_{p(n)>x^\eps} = (c''_\eps + o(1)) \frac{x}{\log x},
\end{equation}
where $c''_{\eps}$ is a quantity depending on $\eps$ but not on $x$ such that
$$ \lim_{\eps \to 0} c''_{\eps} = \int_0^1 F'_k(t) G'_k(t)\ dt.$$
In the case $S(F_k)+S(G_k) < 1$, this would follow easily from (the $k=1$ case of) Theorem \ref{nonprime-asym}(i) and Proposition \ref{almostprime}.  In the general case, we may appeal once more to the fundamental theorem of arithmetic.  As before, we may factor $n = p_1 \dots p_r$ for some $x^\eps \leq p_1 \leq \dots \leq p_r \leq 2x+h_k$ and $r \leq \frac{1}{\eps}+1$.  The contribution of those $n$ with a repeated prime factor $p_i = p_{i+1}$ can easily be shown to be $\lessapprox x^{1-\eps}$ in the same manner we dealt with $\Sigma_2$, so we may restrict attention to the square-free $n$, for which the $p_i$ are strictly increasing.  In that case, one can write
$$ \lambda_{F_k}(n) = (-1)^r \partial_{(\log_x p_1)} \dots \partial_{(\log_x p_r)} F_k(0)$$
and
$$ \lambda_{G_k}(n) = (-1)^r \partial_{(\log_x p_1)} \dots \partial_{(\log_x p_r)} G_k(0)$$
where $\partial_{(h)} F(x) := F(x+h)-F(x)$.  On the other hand, a standard application of Mertens' theorem and the prime number theorem (and an induction on $r$) shows that for any fixed $r \geq 1$ and any fixed continuous function $f: \R^r \to \R$, we have
$$ \sum_{x^\eps \leq p_1 < \dots < p_r: x+h_k \leq p_1 \dots p_r \leq 2x+h_k} f(\log_x p_1,\dots, \log_x p_r) = (c_f + o(1)) \frac{x}{\log x} $$
where $c_f$ is the quantity
$$ c_f := \int_{\eps \leq t_1 < \dots < t_r: t_1 + \dots + t_r = 1} f( t_1,\dots,t_r)\ \frac{dt_1 \dots dt_{r-1}}{t_1 \dots t_r}$$
where we lift Lebesgue measure $dt_1 \dots dt_{r-1}$ up to the hyperplane $t_1+\dots+t_r=1$, thus
$$ \int_{t_1+\dots+t_r=1} F(t_1,\dots,t_r)\ dt_1 \dots dt_{r-1} := \int_{\R^{r-1}} F(t_1,\dots,t_{r-1},1-t_1-\dots-t_{r-1}) dt_1 \dots dt_{r-1}.$$
Putting all this together, we see that we obtain an asymptotic \eqref{cpeps} with
$$ c''_\eps := \sum_{1 \leq r \leq \frac{1}{\eps}+1} 
\int_{\eps \leq t_1 < \dots < t_r: t_1 + \dots + t_r = 1} \partial_{(t_1)} \dots \partial_{(t_r)} F_k(0)
\partial_{(t_1)} \dots \partial_{(t_r)} G_k(0)\ \frac{dt_1 \dots dt_{r-1}}{t_1 \dots t_r}.$$
Comparing \eqref{cpeps} with the first part of Proposition \ref{almostprime} we see that $c''_\eps = O(1)$ uniformly in $\eps$; subtracting two instances of \eqref{cpeps} and comparing with the last part of Proposition \ref{almostprime} we see that $|c''_{\eps_1} - c''_{\eps_2}| \ll \eps_1 + \eps_2$ for any $\eps_1,\eps_2 > 0$.  We conclude that $c''_\eps$ converges to a limit as $\eps\rightarrow 0$ for any $F,G$.  This implies the absolute convergence
\begin{equation}\label{absconv}
 \sum_{r>0} \int_{0 < t_1 < \dots < t_r: t_1 + \dots + t_r = 1} |\partial_{(t_1)} \dots \partial_{(t_r)} F_k(0)|
|\partial_{(t_1)} \dots \partial_{(t_r)} G_k(0)|\ \frac{dt_1 \dots dt_{r-1}}{t_1 \dots t_r} < \infty;
\end{equation}
indeed, by the Cauchy-Schwarz inequality it suffices to establish this for $F=G$, at which point we may remove the absolute value signs and use the boundedness of $c''_\eps$.  By the dominated convergence theorem, it therefore suffices to establish the identity
\begin{equation}\label{condconv}
 \sum_{r>0} \int_{0 < t_1 < \dots < t_r: t_1 + \dots + t_r = 1} \partial_{(t_1)} \dots \partial_{(t_r)} F_k(0)
\partial_{(t_1)} \dots \partial_{(t_r)} G_k(0)\ \frac{dt_1 \dots dt_{r-1}}{t_1 \dots t_r} = \int_0^1 F'_k(t) G'_k(t)\ dt.
\end{equation}
It will suffice to show the identity
\begin{equation}\label{depol}
 \sum_{r>0} \int_{0 < t_1 < \dots < t_r: t_1 + \dots + t_r = 1} |\partial_{(t_1)} \dots \partial_{(t_r)} F(0)|^2\ \frac{dt_1 \dots dt_{r-1}}{t_1 \dots t_r} = \int_0^1 |F'(t)|^2\ dt
\end{equation}
for any smooth $F: [0,+\infty) \to \R$, since \eqref{condconv} follows by replacing $F$ with $F_k+G_k$ and $F_k-G_k$ and then subtracting.

At this point we use the following identity:

\begin{lemma}  For any positive reals $t_1,\dots,t_r$ with $r \geq 1$, we have
\begin{equation}\label{star}
 \frac{1}{t_1 \dots t_r} = \sum_{\sigma \in S_r} \frac{1}{\prod_{i=1}^r (\sum_{j=i}^r t_{\sigma(j)})}.
\end{equation}
\end{lemma}
Thus, for instance, when $r=2$ we have
$$ \frac{1}{t_1 t_2} = \frac{1}{(t_1+t_2) t_1} + \frac{1}{(t_1+t_2)t_2}.$$

\begin{proof}  If the right-hand side of \eqref{star} is denoted $f_r( t_1,\dots,t_r )$, then one easily verifies the identity
$$ f_r(t_1,\dots,t_r) = \frac{1}{t_1+\dots+t_r} \sum_{i=1}^r f_{r-1}(t_1,\dots,t_{i-1},t_{i+1},\dots,t_r)$$
for any $r > 1$; but the left-hand side of \eqref{star} also obeys this identity, and the claim then follows from induction.
\end{proof}

From this lemma and symmetrisation, we may rewrite the left-hand side of \eqref{depol} as
$$
 \sum_{r>0} \int_{\substack{t_1,\dots,t_r \geq 0\\ t_1+\dots+t_r =1}} |\partial_{(t_1)} \dots \partial_{(t_r)} F(0)|^2\ \frac{dt_1 \dots dt_{r-1}}{\prod_{i=1}^r(\sum_{j=i}^rt_i)}.$$
Let
$$ I_a(F) := \int_0^a F'(t)^2\ dt,$$
and
$$ J_a(F) := (\partial_{(a)} F(0))^2.$$
One can then rewrite \eqref{depol} as the identity
\begin{equation}\label{depol-2}
 I_1(F) = \sum_{r=1}^\infty K_{1,r}(F),
\end{equation}
where
$$ K_{a,r}(F) := \int_{\substack{t_1,\dots,t_r \geq 0\\ t_1+\dots+t_r =a}} J_{t_r}( \partial_{(t_1)} \dots \partial_{(t_{r-1})} F) \frac{dt_1 \dots dt_{r-1}}{a(a-t_1) \dots (a-t_1-\dots-t_{r-1})}.$$
To prove this, we first observe the identity
$$ I_a(F) = \frac{1}{a} J_a(F) + \int_{0 \leq t \leq a} I_{a-t}( \partial_{(t)} F ) \frac{dt}{a}$$
for any $a>0$; indeed, we have
\begin{align*}
\int_{0 \leq t \leq a} I_{a-t}( \partial_{(t)} F ) \frac{dt}{a} &=
\int_{0 \leq t \leq a; 0 \leq u \leq a-t} |F'(t+u) - F'(t)|^2\ \frac{du dt}{a} \\
&= \int_{0 \leq t \leq s \leq a} |F'(s) - F'(t)|^2\ \frac{ds dt}{a} \\
&= \frac{1}{2} \int_0^a \int_0^a |F'(s) - F'(t)|^2\ \frac{ds dt}{a} \\
&= \int_0^a |F'(s)|^2\ ds - \frac{1}{a} \left(\int_0^a F'(s)\ ds\right) \left(\int_0^a F'(t)\ dt\right) \\
&= I_a(F) - \frac{1}{a} J_a(F),
\end{align*}
and the claim follows.  Iterating this identity $k$ times, we see that
\begin{equation}\label{iak}
I_a(F) = \sum_{r=1}^k K_{a,r}(F) + L_{a,k}(F)
\end{equation}
for any $k \geq 1$, where
$$ L_{a,k}(F) := \int_{\substack{t_1,\dots,t_k \geq 0\\ t_1+\dots+t_k \leq a}} I_{1-t_1-\dots-t_k}( \partial_{(t_1)} \dots \partial_{(t_k)} F) \frac{dt_1 \dots dt_k}{a(a-t_1) \dots (a-t_1-\dots-t_{k-1})}.$$
In particular, dropping the $L_{a,k}(F)$ term and sending $k \to \infty$ yields the lower bound
\begin{equation}\label{krsum}
\sum_{r=1}^\infty K_{a,r}(F) \leq I_a(F).
\end{equation}
On the other hand, we can expand $L_{a,k}(F)$ as
$$ \int_{\substack{t_1,\dots,t_k,t \geq 0\\t_1+\dots+t_k+t \leq a}} |\partial_{(t_1)} \dots \partial_{(t_k)} F'(t)|^2 \frac{dt_1 \dots dt_k dt}{a(a-t_1) \dots (a-t_1-\dots-t_{k-1})}.$$
Writing $s := t_1 + \dots + t_k$, we obtain the upper bound
$$ L_{a,k}(F) \leq \int_{s,t \geq 0: s+t \leq a} K_{s,k}( F'_t )\ dt,$$
where $F_t(x) := F(x+t)$.  Summing this and using \eqref{krsum} and the monotone convergence theorem, we conclude that
$$ \sum_{k=1}^\infty L_{a,k}(F) \leq \int_{s,t \geq 0: s+t \leq a} I_{s}( F_t )\ dt < \infty,$$
and in particular $L_{a,k}(F) \to 0$ as $k \to \infty$.  Sending $k \to \infty$ in \eqref{iak}, we obtain \eqref{depol-2} as desired.

\section{Reduction to a variational problem}\label{variational-sec}

Now that we have proven Theorems \ref{prime-asym} and \ref{nonprime-asym}, we can establish Theorems \ref{maynard-thm}, \ref{maynard-trunc}, \ref{epsilon-trick}, \ref{epsilon-beyond}.  The main technical difficulty is to take the multidimensional measurable functions $F$ appearing in these functions and approximate them by tensor products of smooth functions, for which Theorems \ref{prime-asym} and \ref{nonprime-asym} may be applied.

\subsection{Proof of Theorem \ref{maynard-thm}}\label{may-sec}

We now prove Theorem \ref{maynard-thm}.  Let $k, m, \vartheta$ obey the hypotheses of that theorem, thus we may find a fixed square-integrable function $F: [0,+\infty)^k \to \R$ supported on the simplex
$$ {\mathcal R}_k := \{ (t_1,\dots,t_k) \in [0,+\infty)^k: t_1+\dots+t_k \leq 1 \}$$
and not identically zero and with
\begin{equation}\label{jbig}
 \frac{\sum_{i=1}^k J_i(F)}{I(F)} > \frac{2m}{\vartheta}.
\end{equation}
We now perform a number of technical steps to further improve the structure of $F$.  Our arguments here will be somewhat convoluted, and are not the most efficient way to prove Theorem \ref{maynard-thm} (which in any event was already established in \cite{maynard-new}), but they will motivate the similar arguments given below to prove the more difficult results in Theorems \ref{maynard-trunc}, \ref{epsilon-trick}, \ref{epsilon-beyond}.  In particular, we will use regularisation techniques which are compatible with the vanishing marginal condition \eqref{vanishing-marginal} that is a key hypothesis in Theorem \ref{epsilon-beyond}.

We first need to rescale and retreat a little bit from the slanted boundary of the simplex ${\mathcal R}_k$.
Let $\delta_1 > 0$ be a sufficiently small fixed quantity, and write $F_1: [0,+\infty)^k \to \R$ to be the rescaled function
$$ F_1(t_1,\dots,t_k) := F( \frac{t_1}{\vartheta/2-\delta_1}, \dots, \frac{t_k}{\vartheta/2-\delta_1} ).$$
Thus $F_1$ is a fixed square-integrable measurable function supported on the rescaled simplex
$$ (\vartheta/2-\delta_1) \cdot {\mathcal R}_k = \{ (t_1,\dots,t_k) \in [0,+\infty)^k: t_1+\dots+t_k \leq \vartheta/2-\delta_1 \}.$$
From \eqref{jbig}, we see that if $\delta_1$ is small enough, then $F_1$ is not identically zero and
\begin{equation}\label{jbig-2}
 \frac{\sum_{i=1}^k J_i(F_1)}{I(F_1)} > m.
\end{equation}

Let $\delta_1$ and $F_1$ be as above. Next, let $\delta_2 > 0$ be a sufficiently small fixed quantity (smaller than $\delta_1$), and write $F_2: [0,+\infty)^k \to \R$ to be the shifted function, defined by setting
$$ F_2(t_1,\dots,t_k) := F_1( t_1-\delta_2, \dots, t_k-\delta_2 )$$
when $t_1,\dots,t_k \geq \delta_2$, and $F_2(t_1,\dots,t_k)=0$ otherwise.
As $F_1$ was square-integrable, compactly supported, and not identically zero, and because spatial translation is continuous in the strong operator topology on $L^2$, it is easy to see that we will have $F_2$ not identically zero and that
\begin{equation}\label{jbig-3}
 \frac{\sum_{i=1}^k J_i(F_2)}{I(F_2)} > m
\end{equation}
for $\delta_2$ small enough (after restricting $F_2$ back to $[0,+\infty)^k$, of course).  For $\delta_2$ small enough, this function will be supported on the region
$$
\{ (t_1,\dots,t_k) \in \R^k: t_1 \dots + t_k \leq \vartheta/2-\delta_2; t_1,\dots,t_k \geq \delta_2 \},$$
thus the support of $F_2$ stays away from all the boundary faces of ${\mathcal R}_k$.

By convolving $F_2$ with a smooth approximation to the identity that is supported sufficiently close to the origin, one may then find a \emph{smooth} function $F_3: [0,+\infty)^k \to \R$, supported on
$$
\{ (t_1,\dots,t_k) \in \R^k: t_1 \dots + t_k \leq \vartheta/2-\delta_2/2; t_1,\dots,t_k \geq \delta_2/2 \},$$
which is not identically zero, and such that
\begin{equation}\label{jbig-4}
 \frac{\sum_{i=1}^k J_i(F_3)}{I(F_3)} > m.
\end{equation}

We extend $F_3$ by zero to all of $\R^k$, and then define the function $f_3: \R^k \to \R$ by
$$ f_3(t_1,\dots,t_k) := \int_{s_1 \geq t_1, \dots, s_k \geq t_k} F_3(s_1,\dots,s_k)\ ds_1 \dots ds_k,$$
thus $f_3$ is smooth, not identically zero and supported on the region
\begin{equation}\label{ftj-set}
 \{ (t_1,\dots,t_k) \in \R^k: \sum_{i=1}^k \max(t_i, \delta_2/2) \leq \vartheta/2 - \delta_2/2 \}.
\end{equation}
From the fundamental theorem of calculus we have
\begin{equation}\label{ftj}
 F_3(t_1,\dots,t_k) := (-1)^k \frac{\partial^k}{\partial t_1 \dots \partial t_k} f_3(t_1,\dots,t_k),
\end{equation}
and so $I(F_3) = \tilde I(f_3)$ and $J_i(F_3) = \tilde J_i(f_3)$ for $i=1,\dots,k$, where
\begin{equation}\label{tidef}
 \tilde I(f_3) := \int_{[0,+\infty)^k} \left|\frac{\partial^k}{\partial t_1 \dots \partial t_k} f_3(t_1,\dots,t_k)\right|^2\ dt_1 \dots dt_k
\end{equation}
and
\begin{equation}\label{tjdef}
 \tilde J_i(f_3) := \int_{[0,+\infty)^{k-1}} \left|\frac{\partial^{k-1}}{\partial t_1 \dots \partial t_{i-1} \partial t_{i+1} \dots \partial t_k} f_3(t_1,\dots,t_{i-1}, 0, t_{i+1}, \dots, t_k)\right|^2\ dt_1 \dots dt_{i-1} dt_{i+1} \dots dt_k.
\end{equation}
In particular, 
\begin{equation}\label{jbig-5}
 \frac{\sum_{i=1}^k \tilde J_i(f_3)}{\tilde I(f_3)} > m.
\end{equation}

Now we approximate $f_3$ by linear combinations of tensor products.  By the Stone-Weierstrass theorem, we may express $f_3$ (on $[0,+\infty)^k$) as the uniform limit of functions
of the form
\begin{equation}\label{cff}
 (t_1,\dots,t_k) \mapsto \sum_{j=1}^J c_j f_{1,j}(t_1) \dots f_{k,j}(t_k)
\end{equation}
where $c_1,\dots,c_J$ are real scalars, and $f_{i,j}: \R \to \R$ are smooth compactly supported functions.  Since $f_3$ is supported in \eqref{ftj-set}, we can ensure that all the components $f_{1,j}(t_1) \dots f_{k,j}(t_k)$ are supported in the slightly larger region
$$ \{ (t_1,\dots,t_k) \in \R^k: \sum_{i=1}^k \max(t_i, \delta_2/4) \leq \vartheta/2 - \delta_2/4 \}.$$
Observe that if one convolves a function of the form \eqref{cff} with a smooth approximation to the identity which is of tensor product form $(t_1,\dots,t_k) \mapsto \varphi_1(t_1) \dots \varphi_1(t_k)$, one obtains another function of this form.  Such a convolution converts a uniformly convergent sequence of functions to a \emph{uniformly smoothly} convergent sequence of functions (that is to say, all derivatives of the functions converge uniformly).  From this, we conclude that $f_3$ can be expressed (on $[0,+\infty)^k$) as the \emph{smooth} limit of functions of the form \eqref{cff}, with each component $f_{1,j}(t_1) \dots f_{k,j}(t_k)$ supported in the region
$$ \{ (t_1,\dots,t_k) \in \R^k: \sum_{i=1}^k \max(t_i, \delta_2/8) \leq \vartheta/2 - \delta_2/8 \}.$$
Thus, we may find such a linear combination
\begin{equation}\label{f4}
 f_4(t_1,\dots,t_k) = \sum_{j=1}^J c_j f_{1,j}(t_1) \dots f_{k,j}(t_k)
\end{equation}
with $J$, $c_j$, $f_{i,j}$ fixed and $f_4$ not identically zero, with
\begin{equation}\label{jbig-6}
 \frac{\sum_{i=1}^k \tilde J_i(f_4)}{\tilde I(f_4)} > m.
\end{equation}
Furthermore, by construction we have
\begin{equation}\label{sfg}
S(f_{1,j}) + \dots + S(f_{k,j}) < \frac{\vartheta}{2} \leq \frac{1}{2}
\end{equation}
for all $j=1,\dots,J$, where $S()$ was defined in \eqref{S-def}.

Now we construct the sieve weight $\nu: \N \to \R$ by the formula
\begin{equation}\label{nu-def}
 \nu(n) := \left( \sum_{j=1}^J c_j \lambda_{f_{1,j}}(n+h_1) \dots \lambda_{f_{k,j}}(n+h_k) \right)^2,
\end{equation}
where the divisor sums $\lambda_f$ were defined in \eqref{lambdaf-def}.

Clearly $\nu$ is non-negative.  Expanding out the square and using Theorem \ref{nonprime-asym}(i) and \eqref{sfg}, we see that
$$ \sum_{\substack{x \leq n \leq 2x\\ n = b\ (W)}} \nu(n) = (\alpha + o(1)) B^{-k} \frac{x}{\log x}$$
where
$$ \alpha := \sum_{j=1}^J \sum_{j'=1}^J c_j c_{j'} \prod_{i=1}^k \int_0^\infty f'_{i,j}(t_i) f'_{i,j'}(t_i)\ dt_i$$
which factorizes using \eqref{f4}, \eqref{tidef} as
\begin{align*}
\alpha &= \int_{[0,+\infty)^k} \left|\frac{\partial^{k-1}}{\partial t_1 \dots \partial t_k} f_4(t_1,\dots,t_k)\right|^2\ dt_1 \dots dt_k \\
&= \tilde I(f_4).
\end{align*}
Now consider the sum
$$ \sum_{\substack{x \leq n \leq 2x\\ n = b\ (W)}} \nu(n) \theta(n+h_k).$$
By \eqref{lambdan-prime}, one has
$$ \lambda_{f_{k,j}}(n+h_k) = f_{k,j}(0)$$
whenever $n$ gives a non-zero contribution to the above sum.  Expanding out the square in \eqref{nu-def} again and using Theorem \ref{prime-asym}(i) and \eqref{sfg} (and the hypothesis $\EH[\vartheta]$), we thus see that
$$ \sum_{\substack{x \leq n \leq 2x\\ n = b\ (W)}} \nu(n) \theta(n+h_k) = (\beta_k + o(1)) B^{1-k} \frac{x}{\phi(W)}$$
where
$$ \beta_k := \sum_{j=1}^J \sum_{j'=1}^J c_j c_{j'} f_{i,j}(0) f_{i,j'}(0) \prod_{i=1}^{k-1} \int_0^\infty f'_{i,j}(t_i) f'_{i,j'}(t_i)\ dt_i$$
which factorizes using \eqref{f4}, \eqref{tjdef} as
\begin{align*}
\beta_k &= \int_{[0,+\infty)^k} \left|\frac{\partial^k}{\partial t_1 \dots \partial t_{k-1}} f_4(t_1,\dots,t_{k-1},0)\right|^2\ dt_1 \dots dt_{k-1} \\
&= \tilde J_k(f_4).
\end{align*}
More generally, we see that
$$ \sum_{\substack{x \leq n \leq 2x\\ n = b\ (W)}} \nu(n) \theta(n+h_i) = (\beta_i + o(1)) B^{1-k} \frac{x}{\phi(W)}$$
for $i=1,\dots,k$, with $\beta_i := \tilde J_i(f_4)$.  Applying Lemma \ref{crit} and \eqref{jbig-4}, we obtain $\DHL[k,m+1]$ as required.

\subsection{Proof of Theorem \ref{maynard-trunc}}\label{trunc-sec}

Now we prove Theorem \ref{maynard-trunc}, which uses a very similar argument to that of the previous section.  Let $k, m, \varpi, \delta, F$ be as in Theorem \ref{maynard-trunc}.  By performing the same rescaling as in the previous section (but with $1/2 + 2\varpi$ playing the role of $\vartheta$), we see that we can find
a fixed square-integrable measurable function $F_1$ supported on the rescaled truncated simplex
$$ \{ (t_1,\dots,t_k) \in [0,+\infty)^k: t_1+\dots+t_k \leq \frac{1}{4} + \varpi - \delta_1; t_1,\dots,t_k < \delta - \delta_1 \}$$
for some sufficiently small fixed $\delta_1>0$, such that \eqref{jbig-2} holds.  By repeating the arguments of the previous section we may eventually arrive at a smooth function $f_4: \R^k \to \R$ of the form \eqref{f4}, which is not identically zero and obeys \eqref{jbig-6}, and such that each component $f_{1,j}(t_1) \dots f_{k,j}(t_k)$ is supported in the region
$$ \{ (t_1,\dots,t_k) \in \R^k: \sum_{i=1}^k \max(t_i, \delta_2/8) \leq \frac{1}{4}+\varpi - \delta_2/8; t_1,\dots,t_k < \delta - \delta_2/8 \}$$
for some sufficiently small $\delta_2>0$.  In particular, one has
$$
S(f_{1,j}) + \dots + S(f_{k,j}) < \frac{1}{4}+\varpi \leq \frac{1}{2}$$
and
$$ S(f_{1,j}),\dots,S(f_{k,j}) < \delta$$
for all $j=1,\dots,J$.   If we then define $\nu$ by \eqref{nu-def} as before, and repeat all of the above arguments (but use Theorem \ref{prime-asym}(ii) and $\MPZ[\varpi,\delta]$ in place of Theorem \ref{prime-asym}(i) and $\EH[\vartheta]$), we obtain the claim; we leave the details to the interested reader.

\subsection{Proof of Theorem \ref{epsilon-trick}}\label{trick-sec}

Now we prove Theorem \ref{epsilon-trick}. Let $k, m, \eps, \vartheta$ be as in that theorem.  Then one may find a square-integrable function $F: [0,+\infty)^k \to \R$ supported on $(1+\eps) \cdot {\mathcal R}_k$ which is not identically zero, and with
$$\frac{\sum_{i=1}^k J_{i,1-\eps}(F)}{I(F)} > \frac{2m}{\vartheta}.$$
By truncating and rescaling as in Section \ref{may-sec}, we may find a fixed bounded measurable function $F_1: [0,+\infty)^k \to \R$ on the simplex $(1+\eps) (\frac{\vartheta}{2}-\delta_1) \cdot {\mathcal R}_k$ such that 
$$\frac{\sum_{i=1}^k J_{i,(1-\eps) \frac{\vartheta}{2}}(F_1)}{I(F_1)} > m.$$

By repeating the arguments in Section \ref{may-sec}, we may eventually arrive at a smooth function $f_4: \R^k \to \R$ of the form \eqref{f4}, which is not identically zero and obeys 
\begin{equation}\label{f4-ratio}
\frac{\sum_{i=1}^k \tilde J_{i,(1-\eps) \frac{\vartheta}{2}}(f_4)}{\tilde I(f_4)} > m
\end{equation}
with
\begin{align*}
 \tilde J_{i, (1-\eps) \frac{\vartheta}{2}}(f_4) &:= \int_{(1-\eps) \frac{\vartheta}{2} \cdot {\mathcal R}_{k-1}} \left|\frac{\partial^{k-1}}{\partial t_1 \dots \partial t_{i-1} \partial t_{i+1} \dots \partial t_k} f_4(t_1,\dots,t_{i-1}, 0, t_{i+1}, \dots, t_k)\right|^2\\
&\quad \ dt_1 \dots dt_{i-1} dt_{i+1} \dots dt_k,
\end{align*}
and such that each component $f_{1,j}(t_1) \dots f_{k,j}(t_k)$ is supported in the region
$$ \left\{ (t_1,\dots,t_k) \in \R^k: \sum_{i=1}^k \max(t_i, \delta_2/8) \leq (1+\eps) \frac{\vartheta}{2} - \frac{\delta_2}{8}\right\}$$
for some sufficiently small $\delta_2>0$.  In particular, we have
\begin{equation}\label{sff}
S(f_{1,j}) + \dots + S(f_{k,j}) \leq (1+\eps) \frac{\vartheta}{2} - \frac{\delta_2}{8}
\end{equation}
for all $1 \leq j \leq J$.  

Let $\delta_3 > 0$ be a sufficiently small fixed quantity (smaller than $\delta_1$ or $\delta_2$).  By a smooth partitioning, we may assume that all of the $f_{i,j}$ are supported in intervals of length at most $\delta_3$, while keeping the sum
\begin{equation}\label{abs-sum}
\sum_{j=1}^J |c_j| |f_{1,j}(t_1)| \dots |f_{k,j}(t_k)|
\end{equation}
bounded uniformly in $t_1,\dots,t_k$ and in $\delta_3$.

Now let $\nu$ be as in \eqref{nu-def}, and consider the expression
$$ \sum_{\substack{x \leq n \leq 2x\\ n = b\ (W)}} \nu(n).$$
This expression expands as a linear combination of the expressions
$$ \sum_{\substack{x \leq n \leq 2x\\ n = b\ (W)}} \prod_{i=1}^k \lambda_{f_{i,j}}(n+h_i) \lambda_{f_{i,j'}}(n+h_i)$$
for various $1 \leq j,j' \leq J$.  We claim that this sum is equal to
$$ \left(\prod_{i=1}^k \int_0^1 f'_{i,j}(t_i) f'_{i,j'}(t_i)\ dt_i + o(1)\right) B^{-k} \frac{x}{W}.$$
To see this, we divide into two cases.  First suppose that hypothesis (i) from Theorem \ref{epsilon-trick} holds.  Then from \eqref{sff} we have
$$ \sum_{i=1}^k (S(f_{i,j}) + S(f_{i,j'})) < (1+\eps) \vartheta < 1$$
and the claim follows from Theorem \ref{nonprime-asym}(i).  Now suppose instead that hypothesis (ii) from Theorem \ref{epsilon-trick} holds, then from \eqref{sff} one has
$$ \sum_{i=1}^k (S(f_{i,j}) + S(f_{i,j'})) < (1+\eps) \vartheta < \frac{k}{k-1} \vartheta,$$
and so from the pigeonhole principle we have
$$ \sum_{1 \leq i \leq k: i \neq i_0} (S(f_{i,j}) + S(f_{i,j'})) < \vartheta$$
for some $1 \leq i_0 \leq k$.  The claim now follows from Theorem \ref{nonprime-asym}(ii).

Putting this together as in Section \ref{may-sec}, we conclude that
$$ \sum_{\substack{x \leq n \leq 2x\\ n = b\ (W)}} \nu(n) = (\alpha + o(1)) B^{-k} \frac{x}{W}$$
where 
$$ \alpha := \tilde I(f_4).$$

Now we consider the sum
\begin{equation}\label{theta-sum}
 \sum_{\substack{x \leq n \leq 2x\\ n = b\ (W)}} \nu(n) \theta(n+h_k).
\end{equation}
From Proposition \ref{geh-eh} we see that we have $\EH[\vartheta]$ as a consequence of the hypotheses of Theorem \ref{epsilon-trick}.  However, this combined with Theorem \ref{prime-asym} is not strong enough to obtain an asymptotic for the sum \eqref{theta-sum}, as there is an epsilon loss in \eqref{sff}.  But observe that Lemma \ref{crit} only requires a \emph{lower} bound on the sum \eqref{theta-sum}, rather than an asymptotic. 

To obtain this lower bound, we partition $\{1,\dots,J\}$ into ${\mathcal J}_1 \cup {\mathcal J}_2$, where ${\mathcal J}_1$ consists of those indices $j \in \{1,\dots,J\}$ with
\begin{equation}\label{sff-minus}
S(f_{1,j}) + \dots + S(f_{k-1,j}) < (1-\eps) \frac{\vartheta}{2}
\end{equation}
and ${\mathcal J}_2$ is the complement.  From the elementary inequality
$$ (x_1 + x_2)^2 = x_1^2 + 2x_1 x_2 + x_2^2 \geq (x_1+2x_2) x_1 $$
we obtain the pointwise lower bound
$$ \nu(n) \geq 
\left( (\sum_{j \in {\mathcal J}_1} + 2 \sum_{j \in {\mathcal J}_2}) c_j \lambda_{f_{1,j}}(n+h_1) \dots \lambda_{f_{k,j}}(n+h_k) \right)
\left( \sum_{j' \in {\mathcal J}_1} c_{j'} \lambda_{f_{1,j'}}(n+h_1) \dots \lambda_{f_{k,j'}}(n+h_k) \right).$$
The point of performing this lower bound is that if $j \in {\mathcal J}_1 \cup {\mathcal J}_2$ and $j' \in {\mathcal J}_1$, then from \eqref{sff}, \eqref{sff-minus} one has
$$ \sum_{i=1}^{k-1} (S(f_{i,j}) + S(f_{i,j'})) < \vartheta$$
which makes Theorem \ref{prime-asym}(i) available for use.  Indeed, for any $j \in \{1,\dots,J\}$ and $i=1,\dots,k$, we have from \eqref{sff} that
$$ S(f_{i,j}) \leq (1+\eps) \frac{\vartheta}{2} < \vartheta < 1$$
and so by \eqref{lambdan-prime} we have
\begin{equation}\label{nutheta}
\begin{split}
 \nu(n) \theta(n+h_k) &\geq 
\left( (\sum_{j \in {\mathcal J}_1} + 2 \sum_{j \in {\mathcal J}_2}) c_j \lambda_{f_{1,j}}(n+h_1) \dots \lambda_{f_{k-1,j}}(n+h_{k-1}) f_{k,j}(0) \right)\\
&\quad \times \left( \sum_{j' \in {\mathcal J}_1} c_{j'} \lambda_{f_{1,j'}}(n+h_1) \dots \lambda_{f_{k-1,j'}}(n+h_{k-1}) f_{k,j'}(0) \right) \theta(n+h_k)
\end{split}
\end{equation}
for $x \leq n \leq 2x$.  If we then apply Theorem \ref{prime-asym}(i) and the hypothesis $\EH[\vartheta]$, we obtain the lower bound
$$ 
 \sum_{\substack{x \leq n \leq 2x\\ n = b\ (W)}} \nu(n) \theta(n+h_k) \geq (\beta_k - o(1)) B^{1-k} \frac{x}{\phi(W)}$$
with
$$ \beta_k :=
(\sum_{j \in {\mathcal J}_1} + 2 \sum_{j \in {\mathcal J}_2}) \sum_{j' \in {\mathcal J}_1} c_j c_{j'} f_{k,j}(0) f_{k,j'}(0)
\prod_{i=1}^{k-1} \int_0^\infty f'_{i,j}(t_i) f'_{i,j'}(t_i)\ dt_i $$
which we can rearrange as
\begin{align*}
\beta_k &= \int_{[0,+\infty)^{k-1}} \left(\frac{\partial^{k-1}}{\partial t_1 \dots \partial t_{k-1}} f_{4,1}(t_1,\dots,t_{k-1},0) + 2\frac{\partial^{k-1}}{\partial t_1 \dots \partial t_{k-1}} f_{4,2}(t_1,\dots,t_{k-1},0)\right)\\
&\quad\quad  \frac{\partial^{k-1}}{\partial t_1 \dots \partial t_{k-1}} f_{4,1}(t_1,\dots,t_{k-1},0)\ dt_1 \dots dt_{k-1} 
\end{align*}
where
$$ f_{4,l}(t_1,\dots,t_k) :=\sum_{j \in {\mathcal J}_l} c_j f_{1,j}(t_1) \dots f_{k,j}(t_k)$$
for $l=1,2$.  Note that $f_{4,1}, f_{4,2}$ are both bounded pointwise by \eqref{abs-sum}, and their supports only overlap on a set of measure $O( \delta_3 )$.  We conclude that
$$ \beta_k = \tilde J_k( f_{4,1} ) + O(\delta_3)$$
with the implied constant independent of $\delta_3$, and thus
$$ \beta_k = \tilde J_{k, (1-\eps) \frac{\vartheta}{2}}( f_4 ) + O(\delta_3).$$
A similar argument gives
$$ 
 \sum_{\substack{x \leq n \leq 2x\\ n = b\ (W)}} \nu(n) \theta(n+h_i) \geq (\beta_i - o(1)) B^{1-k} \frac{x}{\phi(W)}$$
for $i=1,\dots,k$ with
$$ \beta_i = \tilde J_{i, (1-\eps) \frac{\vartheta}{2}}( f_4 ) + O(\delta_3).$$
If we choose $\delta_3$ small enough, then the claim $\DHL[k,m+1]$ now follows from Lemma \ref{crit} and \eqref{f4-ratio}.

\subsection{Proof of Theorem \ref{epsilon-beyond}}\label{beyond-sec}

Finally, we prove Theorem \ref{epsilon-beyond}. Let $k, m, \eps, F$ be as in that theorem.  By rescaling as in previous sections, we may find
a square-integrable function $F_1: [0,+\infty)^k \to \R$ supported on 
$(\frac{k}{k-1} \frac{\vartheta}{2}-\delta_1) \cdot {\mathcal R}_k$ for some sufficiently small fixed $\delta_1 > 0$, which is not identically zero, which obeys the bound
$$
\frac{\sum_{i=1}^k J_{i,(1-\eps) \frac{\vartheta}{2}}(F_1)}{I(F_1)} > m
$$
and also obeys the vanishing marginal condition \eqref{vanishing-marginal} whenever $t_1,\dots,t_{i-1},t_{i+1},\dots,t_k \geq 0$ are such that
$$ t_1+\dots+t_{i-1}+t_{i+1}+\dots+t_k > (1+\eps) \frac{\vartheta}{2} - \delta_1.$$

As before, we pass from $F_1$ to $F_2$ by a spatial translation, and from $F_2$ to $F_3$ by a regularisation; crucially, we note that both of these operations interact well with the vanishing marginal condition \eqref{vanishing-marginal}, with the end product being that we obtain a smooth function $F_3: [0,+\infty)^k \to \R$, supported on the region
$$
\{ (t_1,\dots,t_k) \in \R^k: t_1 \dots + t_k \leq \frac{k}{k-1} \frac{\vartheta}{2}-\frac{\delta_2}{2}; t_1,\dots,t_k \geq \frac{\delta_2}{2} \}$$
for some sufficiently small $\delta_2>0$, which is not identically zero, obeying the bound
$$
\frac{\sum_{i=1}^k J_{i,(1-\eps) \frac{\vartheta}{2}}(F_3)}{I(F_3)} > m
$$
and also obeying the vanishing marginal condition \eqref{vanishing-marginal} whenever
$t_1,\dots,t_{i-1},t_{i+1},\dots,t_k \geq 0$ are such that
$$ t_1+\dots+t_{i-1}+t_{i+1}+\dots+t_k > (1+\eps) \frac{\vartheta}{2} - \frac{\delta_2}{2}.$$

As before, we now define the function $f_3: \R^k \to \R$ by
$$ f_3(t_1,\dots,t_k) := \int_{s_1 \geq t_1, \dots, s_k \geq t_k} F_3(s_1,\dots,s_k)\ ds_1 \dots ds_k,$$
thus $f_3$ is smooth, not identically zero and supported on the region
$$ \left\{ (t_1,\dots,t_k) \in \R^k: \sum_{i=1}^k \max(t_i, \delta_2/2) \leq \frac{k}{k-1} \frac{\vartheta}{2} - \frac{\delta_2}{2} \right\}.$$
Furthermore, from the vanishing marginal condition we see that we also have
$$ f_3(t_1,\dots,t_k) =0 $$
whenever we have some $1 \leq i \leq k$ for which $t_i \leq \delta_2/2$ and
$$ t_1 + \dots + t_{i-1} + t_{i+1} + \dots + t_k \geq (1+\eps) \frac{\vartheta}{2} - \frac{\delta_2}{2}.$$

From the fundamental theorem of calculus as before, we have
$$
\frac{\sum_{i=1}^k \tilde J_{i,(1-\eps) \frac{\vartheta}{2}}(f_3)}{\tilde I(f_3)} > m.
$$
Using the Stone-Weierstrass theorem as before, we can then find a function $f_4$ of the form
\begin{equation}\label{cff-again}
 (t_1,\dots,t_k) \mapsto \sum_{j=1}^J c_j f_{1,j}(t_1) \dots f_{k,j}(t_k)
\end{equation}
where $c_1,\dots,c_J$ are real scalars, and $f_{i,j}: \R \to \R$ are smooth functions supported on intervals of length at most $\delta_3>0$ for some sufficiently small $\delta_3>0$, with each component
$f_{1,j}(t_1) \dots f_{k,j}(t_k)$ supported in the region
$$ \left\{ (t_1,\dots,t_k) \in \R^k: \sum_{i=1}^k \max(t_i, \delta_2/8) \leq \frac{k}{k-1} \frac{\vartheta}{2} - \delta_2/8 \right\}$$
and avoiding the regions
$$ \left\{ (t_1,\dots,t_k) \in \R^k: t_i \leq \delta_2/8; \quad  t_1 + \dots + t_{i-1} + t_{i+1} + \dots + t_k \geq (1+\eps) \frac{\vartheta}{2} - \delta_2/8 \right\}$$
for each $i=1,\dots,k$, and such that
$$
\frac{\sum_{i=1}^k \tilde J_{i,(1-\eps) \frac{\vartheta}{2}}(f_4)}{\tilde I(f_4)} > m.
$$
In particular, for any $j=1,\dots,J$ we have
\begin{equation}\label{sfk-sum}
S(f_{1,j}) + \dots + S(f_{k,j}) < \frac{k}{k-1} \frac{\vartheta}{2} < \frac{1}{2} \frac{k}{k-1} \leq 1
\end{equation}
and for any $i=1,\dots,k$ with $f_{k,i}$ not vanishing at zero, we have
\begin{equation}\label{sfk-sum-2} 
S(f_{1,j}) + \dots + S(f_{k,i-1}) + S(f_{k,i+1}) + \dots + S(f_{k,j}) < (1+\eps) \frac{\vartheta}{2}.
\end{equation}

Let $\nu$ be defined by \eqref{nu-def}.
From \eqref{sfk-sum}, the hypothesis $\GEH[\vartheta]$, and the argument from the previous section used to prove Theorem \ref{epsilon-trick}(ii), we have
$$ \sum_{\substack{x \leq n \leq 2x\\ n = b\ (W)}} \nu(n) = (\alpha + o(1)) B^{-k} \frac{x}{W}$$
where 
$$ \alpha := \tilde I(f_4).$$

Similarly, from \eqref{sfk-sum-2} (and the upper bound $S(f_{i,j}) < 1$ from \eqref{sfk-sum}), the hypothesis $\EH[\vartheta]$ (which is available by Proposition \ref{geh-eh}), and the argument from the previous section we have
$$ 
 \sum_{\substack{x \leq n \leq 2x\\ n = b\ (W)}} \nu(n) \theta(n+h_i) \geq (\beta_i - o(1)) B^{1-k} \frac{x}{\phi(W)}$$
for $i=1,\dots,k$ with
$$ \beta_i = \tilde J_{i, (1-\eps) \frac{\vartheta}{2}}( f_4 ) + O(\delta_3).$$

Setting $\delta_3$ small enough, the claim $\DHL[k,m+1]$ now follows from Lemma \ref{crit}.

\section{Asymptotic analysis}\label{asymptotics-sec}

We now establish upper and lower bounds on the quantity $M_k$ defined in \eqref{mk4}, as well as for the related quantities appearing in Theorem \ref{maynard-trunc}.

To obtain an upper bound on $M_k$, we use the following consequence of the Cauchy-Schwarz inequality.

\begin{lemma}[Cauchy-Schwarz]\label{cs}  Let $k \geq 2$, and suppose that there exist positive measurable functions $G_i: {\mathcal R}_k \to (0,+\infty)$ for $i=1,\dots,k$ such that
\begin{equation}\label{gi}
 \int_0^\infty G_i(t_1,\dots,t_k)\ dt_i \leq 1
\end{equation}
for all $t_1,\dots,t_{i-1},t_{i+1},\dots,t_k \geq 0$, where we extend $G_i$ by zero to all of $[0,+\infty)^k$.  Then we have
\begin{equation}\label{mk-bound}
M_k \leq \operatorname{ess} \sup_{(t_1,\dots,t_k) \in {\mathcal R}_k} \sum_{i=1}^k \frac{1}{G_i(t_1,\dots,t_k)}.
\end{equation}
Here $\operatorname{ess} \sup$ refers to essential supremum (thus, we may ignore a subset of ${\mathcal R}_k$ of measure zero in the supremum).
\end{lemma}

\begin{proof}  Let $F: [0,+\infty)^k \to \R$ be a square-integrable function supported on ${\mathcal R}_k$.  From the Cauchy-Schwarz inequality and \eqref{gi}, we have
$$ \left(\int_0^\infty F(t_1,\dots,t_k)\ dt_i\right)^2 \leq \int_0^\infty \frac{F(t_1,\dots,t_k)^2}{G_i(t_1,\dots,t_k)} \ dt_i$$
for any $t_1,\dots,t_{i-1},t_{i+1},\dots,t_k \geq 0$, with $F^2/G$ extended by zero outside of ${\mathcal R}_k$. Inserting this into \eqref{ji-def} and integrating, we conclude that
$$ J_i(F) \leq \int_{{\mathcal R}_k} \frac{F(t_1,\dots,t_k)^2 }{G_i(t_1,\dots,t_k)}\ dt_1\dots dt_k.$$
Summing in $i$ and using \eqref{i-def}, \eqref{mk4}, \eqref{mk-bound} we obtain the claim.
\end{proof}

As a corollary, we can compute $M_k$ exactly if we can locate a positive eigenfunction:

\begin{corollary}\label{ef}  Let $k \geq 2$, and suppose that there exists a positive function $F: {\mathcal R}_k \to (0,+\infty)$ obeying the eigenfunction equation
\begin{equation}\label{lf}
 \lambda F(t_1,\dots,t_k) = \sum_{i=1}^k \int_0^\infty F(t_1,\dots,t_{i-1},t'_i, t_{i+1},\dots,t_k)\ dt'_i
\end{equation}
for some $\lambda > 0$ and all $(t_1,\dots,t_k) \in {\mathcal R}_k$, where we extend $F$ by zero to all of $[0,+\infty)^k$.  Then $\lambda = M_k$.
\end{corollary}

\begin{proof}  On the one hand, if we integrate \eqref{lf} against $F$ and use \eqref{i-def}, \eqref{ji-def} we see that
$$ \lambda I(F) = \sum_{i=1}^k J_i(F)$$
and thus by \eqref{mk4} we see that $M_k \geq \lambda$.  On the other hand, if we apply Lemma \ref{cs} with
$$ G_i(t_1,\dots,t_k) \coloneqq \frac{F(t_1,\dots,t_k)}{\int_0^\infty F(t_1,\dots,t_{i-1},t'_i, t_{i+1},\dots,t_k)\ dt'_i}$$
we see that $M_k \leq \lambda$, and the claim follows.
\end{proof}

This allows for an exact calculation of $M_2$:

\begin{corollary}[Computation of $M_2$]\label{m2-comp}
We have 
$$M_2 = \frac{1}{1 - W(1/e)} = 1.38593\dots$$
where the Lambert $W$-function $W(x)$ is defined for positive $x$ as the unique positive solution to $x = W(x) e^{W(x)}$.
\end{corollary}

\begin{proof}  If we set $\lambda \coloneqq \frac{1}{1-W(1/e)} = 1.38593\dots$, then a brief calculation shows that
\begin{equation}\label{ll1}
 2\lambda - 1 = \lambda \log \lambda - \lambda \log(\lambda-1).
\end{equation}
Now if we define the function $f: [0,1] \to [0,+\infty)$ by the formula
$$ f(x) \coloneqq \frac{1}{\lambda-1+x} + \frac{1}{2\lambda-1} \log \frac{\lambda-x}{\lambda-1+x}$$
then a further brief calculation shows that
$$ \int_0^{1-x} f(y)\ dy = \frac{\lambda-1+x}{2\lambda-1} \log \frac{\lambda-x}{\lambda-1+x} + \frac{\lambda \log \lambda - \lambda \log(\lambda-1)}{2\lambda-1}$$
for any $0 \leq x \leq 1$, and hence by \eqref{ll1} that
$$ \int_0^{1-x} f(y)\ dy = (\lambda-1+x) f(x).$$
If we then define the function $F: {\mathcal R}_2 \to (0,+\infty)$ by $F(x,y) \coloneqq f(x) + f(y)$, we conclude that
$$ \int_0^{1-x} F(x',y)\ dx' + \int_0^{1-y} F(x,y')\ dy' = \lambda F(x,y)$$
for all $(x,y) \in {\mathcal R}_2$, and the claim now follows from Corollary \ref{ef}.
\end{proof}

We conjecture that a positive eigenfunction for $M_k$ exists for all $k \geq 2$, not just for $k=2$; however, we were unable to produce any such eigenfunctions for $k>2$.  Nevertheless, Lemma \ref{cs} still gives us a general upper bound:

\begin{corollary}\label{mk-upper}  We have $M_k \leq \frac{k}{k-1} \log k$ for any $k \geq 2$.
\end{corollary}

Thus for instance one has $M_2 \leq 2 \log 2 = 1.38629\dots$, which compares well with Corollary \ref{m2-comp}.  On the other hand, Corollary \ref{mk-upper} also gives
$$ M_4 \leq \frac{4}{3} \log 4 = 1.8454\dots, $$
so that one cannot hope to establish $\DHL[4,2]$ (or $\DHL[3,2]$) solely through Theorem \ref{maynard-thm} even when assuming GEH, and must rely instead on more sophisticated criteria for $\DHL[k,m]$ such as Theorem \ref{epsilon-trick} or Theorem \ref{epsilon-beyond}. 

\begin{proof}  If we set $G_i: {\mathcal R}_k \to (0,+\infty)$ for $i=1,\dots,k$ to be the functions
$$ G_i(t_1,\dots,t_k) \coloneqq \frac{k-1}{\log k} \frac{1}{1-t_1-\dots-t_k + kt_i} $$
then direct calculation shows that
$$ \int_0^\infty G_i(t_1,\dots,t_k)\ dt_i \leq 1$$
for all $t_1,\dots,t_{i-1},t_{i+1},\dots,t_k \geq 0$, where we extend $G_i$ by zero to all of $[0,+\infty)^k$.  On the other hand, we have
$$ \sum_{i=1}^k \frac{1}{G_i(t_1,\dots,t_k)} = \frac{k}{k-1} \log k$$
for all $(t_1,\dots,t_k) \in {\mathcal R}_k$.  The claim now follows from Lemma \ref{cs}.
\end{proof}

The upper bound arguments for $M_k$ can be extended to other quantities such as $M_{k,\eps}$, although the bounds do not appear to be as sharp in that case.  For instance, we have the following variant of Lemma \ref{mk-upper}, which shows that the improvement in constants when moving from $M_k$ to $M_{k,\eps}$ is asymptotically modest:

\begin{proposition}\label{mkeps}  For any $k \geq 2$ and $0 \leq \eps < 1$ we have
$$ M_{k,\eps} \leq \frac{k}{k-1} \log(2k-1).$$
\end{proposition}

\begin{proof}  
Let $F: [0,+\infty)^k \to \R$ be a square-integrable function supported on $(1+\eps) \cdot {\mathcal R}_k$.  If $i=1,\dots,k$ and $(t_1,\dots,t_{i-1},t_{i+1},\dots,t_k) \in (1-\eps) \cdot {\mathcal R}_k$, then if we write $s := 1-t_1-\dots-t_{i-1}-t_{i+1}-\dots-t_k$, we have $s \geq \eps$ and hence
\begin{align*}
\int_0^{1-t_1-\dots-t_{i-1}-t_{i+1}-\dots-t_k+\eps} \frac{1}{1-t_1-\dots-t_k + kt_i}\ dt_i 
&= \int_0^{s+\eps} \frac{1}{s+(k-1)t_i}\ dt_i \\
&= \frac{1}{k-1} \log \frac{ks + (k-1)\eps}{s} \\
&\leq \frac{1}{k-1} \log(2k-1).
\end{align*}
By Cauchy-Schwarz, we conclude that
$$ \left(\int_0^\infty F(t_1,\dots,t_k)\ dt_i\right)^2 \leq \frac{1}{k-1} \log(2k-1) \int_0^\infty (1-t_1-\dots-t_k + k t_i) F(t_1,\dots,t_k)^2\ dt_i.$$
Integrating in $t_1,\dots,t_{i-1},t_{i+1},\dots,t_k$ and summing in $i$, we obtain the claim.
\end{proof}

\begin{remark} The same argument, using the weight $1 + a(-t_1-\dots-t_k + kt_i)$, gives the more general inequality
$$ M_{k,\eps} \leq \frac{k}{a(k-1)} \log\left(k + \frac{(a(1+\eps)-1)(k-1)}{1-a(1-\eps)} \right)$$
whenever $\frac{1}{1+\eps} < a < \frac{1}{1-\eps}$; the case $a=1$ is Proposition \ref{mkeps}, and the limiting case $a=\frac{1}{1+\eps}$ recovers Lemma \ref{mk-upper} when one sends $\eps$ to zero.
\end{remark}

One can also adapt the computations in Corollary \ref{m2-comp} to obtain exact expressions for $M_{2,\eps}$, although the calculations are rather lengthy and will only be summarized here.  For fixed $0 < \eps < 1$, the eigenfunctions $F$ one seeks should take the form
$$ F(x,y) \coloneqq f(x) + f(y)$$
for $x,y \geq 0$ and $x+y \leq 1+\eps$, where
$$ f(x) := \onef_{x \leq 1-\eps} \int_0^{1+\eps-x} F(x,t)\ dt.$$
In the regime $0 < \eps < 1/3$, one can calculate that $f$ will (up to scalar multiples) take the form
\begin{align*}
 f(x) &= \onef_{x \leq 2\eps} \frac{C_1}{\lambda-1-\eps+x} \\
&\quad + \onef_{2\eps \leq x \leq 1-\eps} \left(\frac{\log(\lambda-x)-\log(\lambda-1-\eps+x)}{2\lambda-1-\eps} + \frac{1}{\lambda-1-\eps+x} \right)
\end{align*}
where
$$ C_1 := \frac{\log(\lambda-2\eps) - \log(\lambda-1+\eps)}{1 - \log(\lambda-1+\eps) + \log(\lambda-1-\eps)} $$
and $\lambda$ is the largest root of the equation
\begin{align*}
 1 &= C_1 ( \log(\lambda-1+\eps) - \log(\lambda-1-\eps)) - \log(\lambda-1+\eps) \\
&\quad + \frac{ (\lambda-1+\eps) \log(\lambda-1+\eps) - (\lambda-2\eps) \log(\lambda-2\eps) }{2\lambda-1-\eps}.
\end{align*}
In the regime $1/3 \leq \eps < 1$, the situation is significantly simpler, and one has the exact expressions
$$ f(x) = \frac{\onef_{x \leq 1-\eps}}{\lambda - 1 - \eps + x}$$
and 
$$ \lambda = \frac{e(1+\eps)-2\eps}{e-1}.$$
In both cases, a variant of Corollary \ref{ef} can be used to show that $M_{2,\eps}$ will be equal to $\lambda$; thus for instance
$$ M_{2,\eps} = \frac{e(1+\eps)-2\eps}{e-1}$$
for $1/3 \leq \eps < 1$.  In particular, $M_{2,\eps}$ increases to $2$ in the limit $\eps \to 1$; the lower bound $\liminf_{\eps \to 1} M_{2,\eps} \geq 2$ can also be established by testing with the function $F(x,y) := \onef_{x \leq \delta, y \leq 1+\eps-\delta} + \onef_{y \leq \delta, x \leq 1+\eps-\delta}$ for some sufficiently small $\delta>0$.

Now we turn to lower bounds on $M_k$, which are of more relevance for the purpose of establishing results such as Theorem \ref{mlower}.  If one restricts attention to those functions $F: {\mathcal R}_k \to \R$ of the special form $F(t_1,\dots,t_k) = f(t_1+\dots+t_k)$ for some function $f: [0,1] \to \R$ then the resulting variational problem has been optimized in previous works \cite{revesz}, \cite{polymath8a-unabridged} (and originally in unpublished work of Conrey), giving rise to the lower bound
$$ M_k \geq \frac{4k(k-1)}{j_{k-2}^2}$$
where $j_{k-2}$ is the first positive zero of the Bessel function $J_{k-2}$.  This lower bound is reasonably strong for small $k$; for instance, when $k=2$ it shows that
$$ M_2 \geq 1.383\dots$$
which compares well with Corollary \ref{m2-comp}, and also shows that $M_6 > 2$, recovering the result of Goldston, Pintz, and Y{\i}ld{\i}r{\i}m that $\DHL[6,2]$ (and hence $H_1 \leq 16$) was true on the Elliott-Halberstam conjecture.  However, one can show that $\frac{4k(k-1)}{j_{k-2}^2} < 4$ for all $k$ (see \cite{sound}), so this lower bound cannot be used to force $M_k$ to be larger than $4$.

In \cite{maynard-new} the lower bound
\begin{equation}\label{klog}
 M_k \geq \log k - 2\log\log k - 2
\end{equation}
was established for all sufficiently large $k$.  In fact, the arguments in \cite{maynard-new} can be used to show this bound for all $k \geq 200$ (for $k<200$, the right-hand side of \eqref{klog} is either negative or undefined).  Indeed, if we use the bound \cite[(7.19)]{maynard-new} with $A$ chosen so that $A^2 e^A = k$, then $3 < A < \log k$ when $k \geq 200$, hence $e^A = k/A^2 > k / \log^2 k$ and so $A \geq \log k - 2 \log\log k$.  By using the bounds $\frac{A}{e^A-1} < \frac{1}{6}$ (since $A >3$) and $e^A/k = 1/A^2 < 1/9$, we see that the right-hand side of \cite[(8.17)]{maynard-new} exceeds $A - \frac{1}{(1-1/6 - 1/9)^2} \geq A-2$, which gives \eqref{klog}.

We will remove the $\log\log k$ term in \eqref{klog} via the following explicit estimate.

\begin{theorem}\label{explicit}  Let $k \geq 2$, and let $c,T,\tau > 0$ be parameters.  Define the function $g: [0,T] \to \R$ by
\begin{equation}\label{g-def}
 g(t) \coloneqq \frac{1}{c + (k-1) t}
\end{equation}
and the quantities
\begin{align}
m_2 &\coloneqq \int_0^T g(t)^2\ dt \label{m2-def}\\
\mu &\coloneqq \frac{1}{m_2} \int_0^T t g(t)^2\ dt \label{mu-def}\\
\sigma^2 &\coloneqq \frac{1}{m_2} \int_0^T t^2 g(t)^2\ dt - \mu^2.\label{sigma-def}
\end{align}
Assume the inequalities
\begin{align}
k\mu &\leq 1-\tau \label{tau-bound}\\
k\mu &< 1-T \label{T-bound}\\
k\sigma^2 &< (1+\tau-k\mu)^2. \label{ksb}
\end{align}
Then one has
\begin{equation}\label{bigbound}
 \frac{k}{k-1} \log k - M_k^{[T]} \leq \frac{k}{k-1} \frac{Z + Z_3 + WX + VU}{(1+\tau/2) (1 - \frac{k\sigma^2}{(1+\tau-k\mu)^2})}
\end{equation}
where $Z, Z_3, W, X, V, U$ are the explicitly computable quantities
\begin{align}
Z &\coloneqq \frac{1}{\tau} \int_1^{1+\tau}\left( r\left(\log\frac{r-k\mu}{T} + \frac{k\sigma^2}{4(r-k\mu)^2 \log \frac{r-k\mu}{T}} \right) + \frac{r^2}{4kT}\right)\ dr\label{z-def}\\
Z_3 &\coloneqq \frac{1}{m_2} \int_0^T kt \log(1+\frac{t}{T}) g(t)^2\ dt \label{z3-def}\\
W &\coloneqq \frac{1}{m_2} \int_0^T \log(1+\frac{\tau}{kt}) g(t)^2\ dt\label{W-def}\\
X &\coloneqq \frac{\log k}{\tau} c^2 \label{X-def}\\
V &\coloneqq \frac{c}{m_2} \int_0^T \frac{1}{2c + (k-1)t} g(t)^2\ dt \label{V-def} \\
U &\coloneqq \frac{\log k}{c} \int_0^1 \left((1 + u\tau- (k-1)\mu- c)^2 + (k-1) \sigma^2\right)\ du.\label{U-def}
\end{align}
Of course, since $M_k^{[T]} \leq M_k$, the bound \eqref{bigbound} also holds with $M_k^{[T]}$ replaced by $M_k$.
\end{theorem}

\begin{proof}  From \eqref{mk4} we have
$$ \sum_{i=1}^k J_i(F) \leq M_k^{[T]} I(F)$$
whenever $F: [0,+\infty)^k \to \R$ is square-integrable and supported on $[0,T]^k \cap {\mathcal R}_k$.  By rescaling, we conclude that
$$ \sum_{i=1}^k J_i(F) \leq r M_k^{[T]} I(F)$$
whenever $r>0$ and $F: [0,+\infty)^k \to \R$ is square-integrable and supported on $[0,rT]^k \cap r \cdot {\mathcal R}_k$.  We apply this inequality with the function
$$ F(t_1,\dots,t_k) \coloneqq \onef_{t_1+\dots+t_k \leq r} g(t_1) \dots g(t_k)$$
where $r>1$ is a parameter which we will eventually average over, and $g$ is extended by zero to $[0,+\infty)$.  We thus have
$$ I(F) = m_2^k \int_0^\infty \dots \int_0^\infty \onef_{t_1+\dots+t_k \leq r} \prod_{i=1}^k \frac{g(t_i)^2\ dt_i}{m_2}.$$
We can interpret this probabilistically as
$$ I(F) = m_2^k \P( X_1 + \dots + X_k \leq r )$$
where $X_1,\dots,X_k$ are independent random variables taking values in $[0,T]$ with probability distribution $\frac{1}{m_2} g(t)^2\ dt$.  In a similar fashion, we have
$$ J_k(F) = m_2^{k-1} \int_0^\infty \dots \int_0^\infty\left (\int_{[0,r-t_1-\dots-t_{k-1}]} g(t)\ dt\right)^2 \prod_{i=1}^{k-1} \frac{g(t_i)^2\ dt_i}{m_2},$$
where we adopt the convention that $\int_{[a,b]}$ vanishes when $b<a$.  In probabilistic language, we thus have
$$ J_k(F) = m_2^{k-1} \E \left(\int_{[0,r-X_1-\dots-X_{k-1}]} g(t)\ dt\right)^2$$
where we adopt the convention that the expectation operator $\E$ applies to the entire expression to the right of that operator unless explicitly restricted by parentheses.
Also by symmetry we see that $J_i(F)=J_k(F)$ for all $i=1,\dots,k$.  Putting all this together, we conclude that
$$\E \left(\int_0^{r-X_1-\dots-X_{k-1}} g(t)\ dt\right)^2 \leq \frac{m_2 M_k^{[T]} r}{k} \P( X_1 + \dots + X_k \geq r )$$
for all $r>1$.  Writing $S_i \coloneqq X_1 + \dots + X_i$, we abbreviate this as
\begin{equation}\label{rg1}
\E \left(\int_{[0,r-S_{k-1}]} g(t)\ dt\right)^2 \leq \frac{m_2 M_k^{[T]} r}{k} \P( S_k \geq r ).
\end{equation}
Now we run a variant of the Cauchy-Schwarz argument used to prove Corollary \ref{mk-upper}.  If, for fixed $r>0$, we introduce the random function $h: (0,+\infty) \to \R$ by the formula
\begin{equation}\label{hdef}
 h(t) \coloneqq \frac{1}{r-S_{k-1}+(k-1)t} \onef_{S_{k-1} < r}
\end{equation}
and observe that whenever $S_{k-1} < r$, we have
\begin{equation}\label{h-int}
 \int_{[0,r-S_{k-1}]} h(t)\ dt = \frac{\log k}{k-1}
\end{equation}
and thus by the Legendre identity we have
$$ \left(\int_{[0,r-S_{k-1}]} g(t)\ dt\right)^2 = \frac{\log k}{k-1} \int_{[0,r-S_{k-1}]} \frac{g(t)^2}{h(t)}\ dt - \frac{1}{2} \int_{[0,r-S_{k-1}]} \int_{[0,r-S_{k-1}]} \frac{(g(s)h(t)-g(t)h(s))^2}{h(s)h(t)}\ ds dt$$
for $S_{k-1} < r$; but the claim also holds when $r \leq S_{k-1}$ since all integrals vanish in that case.  On the other hand, we have
\begin{align*}
\E \int_{[0,r-S_{k-1}]} \frac{g(t)^2}{h(t)}\ dt &= m_2 \E (r - S_{k-1} + (k-1) X_k) \onef_{X_k \leq r - S_{k-1} }\\
&= m_2 \E (r - S_k + k X_k) \onef_{S_k \leq r} \\
&= m_2 \E r \onef_{S_k \leq r} \\
&= m_2 r \P( S_k \leq r)
\end{align*}
where we have used symmetry to get the third equality.  We conclude that
$$ \E (\int_{[0,r-S_{k-1}]} g(t)\ dt)^2 = \frac{\log k}{k-1} m_2 r \P(S_k \leq r) - \frac{1}{2} \E \int_{[0,r-S_{k-1}]} \int_{[0,r-S_{k-1}]} \frac{(g(s)h(t)-g(t)h(s))^2}{h(s)h(t)}\ ds dt.$$
Combining this with \eqref{rg1}, we conclude that
$$
\Delta r \P( S_k \leq r) \leq \frac{k}{2m_2}
\E \int_{[0,r-S_{k-1}]} \int_{[0,r-S_{k-1}]} \frac{(g(s)h(t)-g(t)h(s))^2}{h(s)h(t)}\ ds dt
$$
where
$$ \Delta \coloneqq \frac{k}{k-1} \log k - M_k^{[T]}.$$
Splitting into regions where $s,t$ are less than $T$ or greater than $T$, and noting that $g(s)$ vanishes for $s > T$, we conclude that
$$
\Delta r \P( S_k \leq r) \leq Y_1(r) + Y_2(r)$$
where 
$$ Y_1(r) \coloneqq \frac{k}{m_2} \E \int_{[0,T]} \int_{[T,r-S_{k-1}]} \frac{g(t)^2}{h(t)} h(s)\ ds dt $$
and
$$ Y_2(r) \coloneqq \frac{k}{2m_2} \E \int_{[0,\min(T,r-S_{k-1})]} \int_{[0,\min(T,r-S_{k-1})]} \frac{(g(s)h(t)-g(t)h(s))^2}{h(s)h(t)}\ ds dt.$$
We average this from $r=1$ to $r=1+\tau$, to conclude that
$$
\Delta (\frac{1}{\tau} \int_1^{1+\tau} r \P( S_k \leq r)\ dr) \leq \frac{1}{\tau} \int_1^{1+\tau} Y_1(r)\ dr + 
\frac{1}{\tau} \int_1^{1+\tau} Y_2(r)\ dr.$$
Thus to prove \eqref{bigbound}, it suffices (by \eqref{ksb}) to establish the bounds
\begin{equation}\label{denom-bound}
\frac{1}{\tau} \int_1^{1+\tau} r \P( S_k \leq r)\ dr \geq 
(1+\tau/2) \left(1 - \frac{k\sigma^2}{(1+\tau-k\mu)^2}\right),
\end{equation}
\begin{equation}\label{y1-bound}
\frac{k}{k-1} Y_1(r) \leq Z + Z_3
\end{equation}
for all $1 < r \leq 1+\tau$, and
\begin{equation}\label{y2-bound}
\frac{1}{\tau} \int_1^{1+\tau} Y_2(r)\ dr \leq \frac{k}{k-1}( WX+VU ).
\end{equation}

We begin with \eqref{denom-bound}.  Since
$$ \frac{1}{\tau} \int_1^{1+\tau} r\ dr = 1+\frac{\tau}{2}$$
it suffices to show that
$$
\frac{1}{\tau} \int_1^{1+\tau} r \P( S_k > r) \leq (1+\frac{\tau}{2}) \frac{k\sigma^2}{(1+\tau-k\mu)^2}.
$$
But, from \eqref{mu-def}, \eqref{sigma-def}, we see that each $X_i$ has mean $\mu$ and variance $\sigma^2$, so $S_k$ has mean $k\mu$ and variance $k\sigma^2$.  It thus suffices to show the pointwise bound
$$
\frac{1}{\tau} \int_1^{1+\tau} r 1_{x>r} \leq (1+\frac{\tau}{2}) \frac{(x-k\mu)^2}{(1+\tau-k\mu)^2}
$$
for any $x$. It suffices to verify this in the range $1 \leq x \leq 1+\tau$.  But in this range, the left-hand side is convex, equals $0$ at $1$ and $1+\tau/2$ at $1+\tau$, while the right-hand side is convex, and equals $1+\tau/2$ at $1+\tau$ with slope at least $(1+\tau/2)/\tau$ there thanks to \eqref{tau-bound}.  The claim follows.

Now we show \eqref{y1-bound}.  The quantity $Y_1(r)$ is vanishing unless $r-S_{k-1} \geq T$.  Using the crude bound $h(s) \leq \frac{1}{(k-1)s}$ from \eqref{hdef}, we see that
$$ \int_{[T,r-S_{k-1}]} h(s)\ ds \leq \frac{1}{k-1} \log_+ \frac{r-S_{k-1}}{T}$$
where $\log_+(x) \coloneqq \max(\log x, 0)$.  We conclude that
$$ Y_1(r) \leq \frac{k}{k-1} \frac{1}{m_2} \E \int_{[0,T]} \frac{g(t)^2}{h(t)}\ dt \log_+ \frac{r-S_{k-1}}{T}.$$
We can rewrite this as
$$ Y_1(r) \leq \frac{k}{k-1} \E \frac{\onef_{S_k \leq r}}{h(X_k)} \log_+ \frac{r-S_{k-1}}{T}.$$
By \eqref{hdef}, we have
$$ \frac{\onef_{S_k \leq r}}{h(X_k)} = (r-S_k+kX_k) \onef_{S_k \leq r}.$$
Also, from the elementary bound $\log_+(x+y) \leq \log_+ x + \log(1+y)$ for any $x,y \geq 0$, we see that
$$ \log_+ \frac{r-S_{k-1}}{T} \leq \log_+ \frac{r-S_{k}}{T} + \log\left(1+\frac{X_k}{T}\right).$$
We conclude that
\begin{align*}
Y_1(r) &\leq \frac{k}{k-1} \E (r-S_k+kX_k) \left( \log_+ \frac{r-S_{k}}{T} + \log\left(1+\frac{X_k}{T}\right) \right) \onef_{S_k \leq r}\\
&\leq \frac{k}{k-1} \left( \E (r-S_k+kX_k) \log_+ \frac{r-S_{k}}{T} + \max(r-S_k,0) \frac{X_k}{T} + k X_k \log\left(1+\frac{X_k}{T}\right) \right)
\end{align*}
using the elementary bound $\log(1+y) \leq y$.  Symmetrizing in the $X_1,\dots,X_k$, we conclude that
\begin{equation}\label{y1r}
 Y_1(r) \leq \frac{k}{k-1} (Z_1(r) + Z_2(r) + Z_3)
\end{equation}
where
\begin{align*}
Z_1(r) &\coloneqq \E r \log_+ \frac{r-S_k}{T} \\
Z_2(r) &\coloneqq \E (r-S_k) \onef_{S_k \leq r} \frac{S_k}{kT} 
\end{align*}
and $Z_3$ was defined in \eqref{z3-def}.

For the minor error term $Z_2$, we use the crude bound $(r-S_k) \onef_{S_k \leq r} S_k \leq \frac{r^2}{4}$, so
\begin{equation}\label{z2r}
Z_2(r) \leq \frac{r^2}{4kT}.
\end{equation}
For $Z_1$, we upper bound $\log_+ x$ by a quadratic expression in $x$.  More precisely, we observe the inequality
$$ \log_+ x \leq \frac{(x-2a\log a-a)^2}{4a^2 \log a}$$
for any $a > 1$ and $x \in \R$, since the left-hand side is concave in $x$ for $x \geq 1$, while the right-hand side is convex in $x$, non-negative, and tangent to the left-hand side at $x=a$.  We conclude that
$$ \log_+ \frac{r-S_k}{T} \leq \frac{(r-S_k-2aT\log a-aT)^2}{4a^2 T^2 \log a}.$$
On the other hand, from \eqref{mu-def}, \eqref{sigma-def}, we see that each $X_i$ has mean $\mu$ and variance $\sigma^2$, so $S_k$ has mean $k\mu$ and variance $k\sigma^2$.  We conclude that
$$ Z_1(r) \leq r \frac{ (r-k\mu-2aT\log a-aT)^2 + k \sigma^2}{4a^2 T^2 \log a}$$
for any $a > 1$.

From \eqref{T-bound} and the assumption $r > 1$, we may choose $a \coloneqq \frac{r-k\mu}{T}$ here, leading to the simplified formula
\begin{equation}\label{z1r}
Z_1(r) \leq r \left( \log \frac{r-k\mu}{T} + \frac{k\sigma^2}{4(r-k\mu)^2 \log \frac{r-k\mu}{T}}\right).
\end{equation}
From \eqref{y1r}, \eqref{z2r}, \eqref{z1r}, \eqref{z-def} we conclude \eqref{y1-bound}.

Finally, we prove \eqref{y2-bound}.  Here, we finally use the specific form \eqref{g-def} of the function $g$.  Indeed, from \eqref{g-def}, \eqref{hdef} we observe the identity
$$ g(t) - h(t) = (r - S_{k-1} - c) g(t) h(t)$$
for $t \in [0, \min(r-S_{k-1},T)]$.  Thus
\begin{align*}
Y_2(r) &= \frac{k}{2 m_2} \E \int_{[0,\min(r-S_{k-1},T)]} \int_{[0,\min(r-S_{k-1},T)]} \frac{((g-h)(s) h(t)-(g-h)(t) h(s))^2}{h(s) h(t)}\ ds dt \\
&= \frac{k}{2 m_2} \E (r - S_{k-1} - c)^2 \int_{[0,\min(r-S_{k-1},T)]} \int_{[0,\min(r-S_{k-1},T)]} (g(s)-g(t))^2 h(s) h(t)\ ds dt.
\end{align*}
Using the crude bound $(g(s)-g(t))^2 \leq g(s)^2+g(t)^2$ and using symmetry, we conclude
$$Y_2(r) \leq \frac{k}{m_2} \E (r - S_{k-1} - c)^2 \int_{[0,\min(r-S_{k-1},T)]} \int_{[0,\min(r-S_{k-1},T)]} g(s)^2 h(s) h(t)\ ds dt.$$
From \eqref{h-int}, \eqref{hdef} we conclude that
$$ Y_2(r) \leq \frac{k}{k-1} Z_4(r)$$
where
$$ Z_4(r) \coloneqq \frac{\log k}{m_2} \E\left( (r - S_{k-1} - c)^2 \int_{[0,\min(r-S_{k-1},T)]} \frac{g(s)^2}{r-S_{k-1}+(k-1)s}\ ds \right) .$$
To prove \eqref{y2-bound}, it thus suffices (after making the change of variables $r = 1 + u \tau$) to show that
\begin{equation}\label{z4-bound}
 \int_0^1 Z_4(1+u\tau)\ du \leq WX+VU.
\end{equation}
We will exploit the averaging in $u$ to deal with the singular nature of the factor $\frac{1}{r-S_{k-1}+(k-1)s}$.  By Fubini's theorem, the left-hand side of \eqref{z4-bound} may be written as
$$ \frac{\log k}{m_2} \E \int_0^1 Q(u)\ du $$
where $Q(u)$ is the random variable
$$ Q(u) \coloneqq (1+u\tau - S_{k-1} - c)^2 \int_{[0,\min(1+u\tau-S_{k-1},T)]} \frac{g(s)^2}{1+u\tau-S_{k-1}+(k-1)s}\ ds.$$
Note that $Q(u)$ vanishes unless $1+u\tau-S_{k-1} > 0$.  Consider first the contribution of those $Q(u)$ for which
$$ 0 < 1+u\tau-S_{k-1} \leq 2c.$$
In this regime we may bound
$$ (1+u\tau-S_{k-1}-c)^2 \leq c^2,$$
so this contribution to \eqref{z4-bound} may be bounded by
$$ \frac{\log k}{m_2} c^2 \E \int_{[0,T]} g(s)^2 \left(\int_0^1 \frac{\onef_{1+u\tau-S_{k-1} \geq s}}{1+u\tau-S_{k-1}+(k-1)s}\ du\right)\ ds.$$
Observe on making the change of variables $v \coloneqq 1 + u\tau - S_{k-1} + (k-1)s$ that 
\begin{align*}
\int_0^1 \frac{\onef_{1+u\tau-S_{k-1} \geq s}}{1+u\tau-S_{k-1}+(k-1)s}\ du &= \frac{1}{\tau} \int_{[\max(ks, 1-S_{k-1}+(k-1)s), 1-S_{k-1}+\tau+(k-1)s]} \frac{dv}{v} \\
&\leq \frac{1}{\tau} \log \frac{ks+\tau}{ks}
\end{align*}
and so this contribution to \eqref{z4-bound} is bounded by $WX$, where $W,X$ are defined in \eqref{W-def}, \eqref{X-def}.

Now we consider the contribution to \eqref{z4-bound} when\footnote{One could obtain a small improvement to the bounds here by replacing the threshold $2c$ with a parameter to be optimized over.}
$$ 1+u\tau-S_{k-1} > 2c.$$
In this regime we bound
$$ \frac{1}{1+u\tau-S_{k-1}+(k-1)s} \leq \frac{1}{2c+(k-1)t},$$ 
and so this portion of $\int_0^1 Z_4[1+u\tau]\ du$ may be bounded by 
$$\int_0^1 \frac{\log k}{c} \E (1 + u\tau - S_{k-1} - c)^2 V\ du = VU$$
where $V, U$ are defined in \eqref{V-def}, \eqref{U-def}.  The proof of the theorem is now complete.
\end{proof}

We can now perform an asymptotic analysis in the limit $k \to \infty$ to establish Theorem \ref{mlower}(xi) and Theorem \ref{mlower-var}(vi).  For $k$ sufficiently large, we select the parameters
\begin{align*}
c &\coloneqq \frac{1}{\log k} + \frac{\alpha}{\log^2 k} \\
T &\coloneqq \frac{\beta}{\log k} \\
\tau &\coloneqq \frac{\gamma}{\log k}
\end{align*}
for some real parameters $\alpha \in \R$ and $\beta,\gamma > 0$ independent of $k$ to be optimized in later.  From \eqref{g-def}, \eqref{m2-def} we have
\begin{align*}
m_2 &= \frac{1}{k-1} \left( \frac{1}{c} - \frac{1}{c+(k-1)T} \right) \\
&= \frac{\log k}{k} \left(1 - \frac{\alpha}{\log k} + o( \frac{1}{\log k} ) \right)
\end{align*}
where we use $o(f(k))$ to denote a function $g(k)$ of $k$ with $g(k)/f(k) \to 0$ as $k \to \infty$.  On the other hand, we have from \eqref{g-def}, \eqref{mu-def} that
\begin{align*}
m_2 (c + (k-1) \mu) &= \int_0^T (c+(k-1)t) g(t)^2\ dt \\
&= \frac{1}{k-1} \log \frac{c+(k-1)T}{c} \\
&= \frac{\log k}{k} \left(1 + \frac{\log \beta}{\log k} + o( \frac{1}{\log k} ) \right)
\end{align*}
and thus
\begin{align*}
 k\mu &= \frac{k}{k-1} \left(1 + \frac{\log \beta + \alpha}{\log k} + o(\frac{1}{\log k} ) \right) - \frac{kc}{k-1} \\
&= 1 + \frac{\log \beta + \alpha}{\log k} + o\left( \frac{1}{\log k} \right) - \left( \frac{1}{\log k} + o\left( \frac{1}{\log k} \right) \right)\\
&= 1 + \frac{\log \beta + \alpha - 1}{\log k} + o\left( \frac{1}{\log k} \right).
\end{align*}
Similarly, from \eqref{g-def}, \eqref{mu-def}, \eqref{sigma-def} we have
\begin{align*}
m_2 (c^2 + 2c(k-1)\mu + (k-1)^2 (\mu^2+\sigma^2)) &= \int_0^T (c+(k-1)t)^2 g(t)^2\ dt \\
&= T
\end{align*}
and thus
\begin{align*}
k \sigma^2 &= \frac{k}{(k-1)^2} \left(\frac{T}{m_2} - c^2 - 2c(k-1)\mu\right) - k \mu^2 \\
&= \frac{\beta}{\log^2 k} + o( \frac{1}{\log^2 k} ).
\end{align*}

We conclude that the hypotheses \eqref{tau-bound}, \eqref{T-bound}, \eqref{ksb} will be obeyed for sufficiently large $k$ if we have
\begin{align*}
\log \beta + \alpha + \gamma &< 1 \\
\log \beta + \alpha + \beta &< 1 \\
\beta &< (1 + \gamma - \alpha - \log \beta)^2.
\end{align*}
These conditions can be simultaneously obeyed, for instance by setting $\beta=\gamma=1$ and $\alpha = -1$.

Now we crudely estimate the quantities $Z,Z_3,W,X,V,U$ in \eqref{z-def}-\eqref{U-def}.  For $1 \leq r \leq 1+\tau$, we have $r-k\mu \sim 1/\log k$, and so
$$ \frac{r-k\mu}{T} \sim 1; \quad \frac{k \sigma^2}{(r-k\mu)^2} \sim 1; \quad \frac{r^2}{4kT} = o(1)$$
and so by \eqref{z-def} $Z = O(1)$.  Using the crude bound $\log(1+\frac{t}{T}) = O(1)$ for $0 \leq t \leq T$, we see from \eqref{z3-def}, \eqref{mu-def} that $Z_3 = O( k \mu ) = O(1)$.  It is clear that $X = O(1)$, and using the crude bound $\frac{1}{2c+(k-1)t} \leq \frac{1}{c}$ we see from \eqref{V-def}, \eqref{m2-def} that $V = O(1)$.  For $0 \leq u \leq 1$ we have $1 + u\tau - (k-1)\mu - c = O(1/\log k)$, so from \eqref{U-def} we have $U=O(1)$.  Finally, from \eqref{W-def} and the change of variables $t = \frac{s}{k \log k}$ we have
\begin{align*}
W &= \frac{\log k}{k m_2} \int_0^{kT\log k} \log\left(1 + \frac{\gamma}{s}\right) \frac{ds}{(1 + \frac{\alpha}{\log k} + \frac{k-1}{k} s)^2} \\
&= O\left( \int_0^\infty \log\left(1+\frac{\gamma}{s}\right) \frac{ds}{(1+o(1))(1+s)^2} \right) \\
&= O(1).
\end{align*}
Finally we have
$$ 1 - \frac{k\sigma^2}{(1+\tau-k\mu)^2} \sim 1.$$
Putting all this together, we see from \eqref{bigbound} that
$$  M_k \geq M_k^{[T]} \geq \frac{k}{k-1} \log k -  O(1)$$
giving Theorem \ref{mlower}(xi).  Furthermore, if we set
$$ \varpi \coloneqq \frac{7}{600} - \frac{C}{\log k}$$
and
$$ \delta \coloneqq \left(\frac{1}{4} + \frac{7}{600}\right) \frac{\beta}{\log k}$$
then we will have $600 \varpi + 180 \delta < 7$ for $C$ large enough, and Theorem \ref{mlower-var}(vi) also follows (as one can verify from inspection that all implied constants here are effective).


Finally, Theorem \ref{mlower}(viii), (ix), (x) follow by setting
\begin{align*}
c &:= \frac{\theta}{\log k} \\
T &:= \frac{\beta}{\log k} \\
\tau &= 1 - k\mu
\end{align*}
with $\theta,\beta$ given by Table \ref{narrow-table}, with \eqref{bigbound} then giving the bound $M_k^{[T]} > M$ with $M$ as given by the table, after verifying of course that the conditions \eqref{tau-bound}, \eqref{T-bound}, \eqref{ksb} are obeyed.  Similarly, Theorem \ref{mlower-var} (ii), (iii), (iv), (v) follows with $\theta,\beta$ given by the same table, with $\varpi$ chosen so that
$$ M = \frac{m}{\frac{1}{4}+\varpi} $$
with $m=2,3,4,5$ for (ii), (iii), (iv), (v) respectively, and $\delta$ chosen by the formula
$$ \delta := T (\frac{1}{4} + \varpi).$$

\begin{table}
\centering
 \caption{Parameter choices for Theorems \ref{mlower}, \ref{mlower-var}.}
 \label{narrow-table}
  \begin{tabular}{llll}
   \toprule
     $k$    & $\theta$   & $\beta$          &M \\
   \midrule
           5511 & 0.965              & 0.973               & 6.000048609 \\ 
          35410 & 0.99479            & 0.85213             & 7.829849259 \\ [0.5ex]
          41588 & 0.97878  	         & 0.94319             & 8.000001401 \\
         309661 & 0.98627            & 0.92091             & 10.00000032 \\ [0.5ex]
				1649821 & 1.00422            & 0.80148             & 11.65752556 \\  
       75845707 & 1.00712            & 0.77003             & 15.48125090 \\  [0.5ex]
     3473955908 & 1.0079318          & 0.7490925           & 19.30374872 \\  
   \bottomrule
  \end{tabular}
\end{table}

\section{The case of small and medium dimension}\label{h1-sec}

In this section we establish lower bounds for $M_k$ (and related quantities, such as $M_{k,\eps}$) both for small values of $k$ (in particular, $k=3$ and $k=4$) and medium values of $k$ (in particular, $k=50$ and $k=54$).  Specifically, we will establish Theorem \ref{mlower}(vii), Theorem \ref{mke-lower}, and Theorem \ref{piece}.

\subsection{Bounding \texorpdfstring{$M_k$}{Mk} for medium \texorpdfstring{$k$}{k}}

We begin with the problem of lower bounding $M_k$.  We first formalize an observation\footnote{The arguments in \cite{maynard-new} are rigorous under the assumption of a positive eigenfunction as in Corollary \ref{ef}, but the existence of such an eigenfunction remains open for $k \geq 3$.} of Maynard \cite{maynard-new} that one may restrict without loss of generality to symmetric functions:

\begin{lemma}  For any $k \geq 2$, one has
$$ M_k \coloneqq \sup \frac{k J_1(F)}{I(F)}$$
where $F$ ranges over \emph{symmetric} square-integrable functions on ${\mathcal R}_k$ that are not identically zero.
\end{lemma}

\begin{proof}  Firstly, observe that if one replaces a square-integrable function $F: [0,+\infty)^k \to \R$ with its absolute value $|F|$, then $I(|F|) = I(F)$ and $J_i(|F|) \geq J_i(F)$.  Thus one may restrict the supremum in \eqref{mk4} to non-negative functions without loss of generality.  We may thus find a sequence $F_n$ of square-integrable non-negative functions on ${\mathcal R}_k$, normalized so that $I(F_n)=1$, and such that $\sum_{i=1}^k J_i(F_n) \to M_k$ as $n \to \infty$.

Now let
$$\overline{F_n}(t_1,\dots,t_k) \coloneqq \frac{1}{k!} \sum_{\sigma \in S_k} F_n( t_{\sigma(1)},\dots,t_{\sigma(k)} )$$
be the symmetrization of $F_n$.  Since the $F_n$ are non-negative with $I(F_n)=1$, we see that
$$ I( \overline{F_n} ) \geq I( \frac{1}{k!} F_n ) = \frac{1}{(k!)^2}$$
and so $I(\overline{F_n})$ is bounded away from zero.  Also, from \eqref{mk4}, we know that the quadratic form
$$ Q(F) \coloneqq M_k I(F) - \sum_{i=1}^k J_i(F)$$
is positive semi-definite and is also invariant with respect to symmetries, and so from the triangle inequality for inner product spaces we conclude that
$$ Q( \overline{F_n} ) \leq Q( F_n ).$$
By construction, $Q(F_n)$ goes to zero as $n \to \infty$, and thus $Q(\overline{F_n})$ also goes to zero.  We conclude that
$$ \frac{k J_1(\overline{F_n})}{I(\overline{F_n})} = \frac{\sum_{i=1}^k J_i(\overline{F_n})}{I(\overline{F_n})} \to M_k$$
as $n \to \infty$, and so
$$ M_k \geq \sup \frac{k J_1(F)}{I(F)}.$$
The reverse inequality is immediate from \eqref{mk4}, and the claim follows.
\end{proof}

To establish a lower bound of the form $M_k > C$ for some $C > 0$, one thus seeks to locate a symmetric function $F: [0,+\infty)^k \to \R$ supported on ${\mathcal R}_k$ such that
\begin{equation}\label{kjaf}
 k J_1(F) > C I(F).
\end{equation}
To do this numerically, we follow \cite{maynard-new} (see also \cite{gpy} for some related ideas) and can restrict attention to functions $F$ that are linear combinations
$$ F = \sum_{i=1}^n a_i b_i$$
of some explicit finite set of symmetric square-integrable functions $b_1,\dots,b_n: [0,+\infty)^k \to \R$ supported on ${\mathcal R}_k$, and some real scalars $a_1,\dots,a_n$ that we may optimize in.  The condition \eqref{kjaf} then may be rewritten as
\begin{equation}\label{ama}
 \mathbf{a}^T \mathbf{M}_2 \mathbf{a} - C \mathbf{a}^T \mathbf{M}_1 \mathbf{a} > 0
\end{equation}
where $\mathbf{a}$ is the vector
$$ \mathbf{a} \coloneqq \begin{pmatrix} a_1 \\ \vdots \\ a_n \end{pmatrix}
$$
and $\mathbf{M}_1,\mathbf{M}_2$ are the real symmetric and positive semi-definite $n \times n$ matrices
\begin{align}\label{Eq:M1}
\mathbf{M}_1 &= \left( \int_{\R^k} b_i(t_1,\dots,t_k) b_j(t_1,\dots,t_k)\ dt_1 \dots dt_k \right)_{1 \leq i,j \leq n} \\\label{Eq:M2}
\mathbf{M}_2 &= \left( k \int_{\R^{k+1}} b_i(t_1,\dots,t_k) b_j(t_1,\dots,t_{k-1},t'_k)\ dt_1 \dots dt_k dt'_k\right)_{1 \leq i,j \leq n}.
\end{align}
If the $b_1,\dots,b_n$ are linearly independent in $L^2({\mathcal R}_k)$, then $\mathbf{M}_1$ is strictly positive definite, and (as observed in \cite[Lemma 8.3]{maynard-new}), one can find $\mathbf{a}$ obeying \eqref{ama} if and only if the largest eigenvalue of $\mathbf{M}_2 \mathbf{M}_1^{-1}$ exceeds $C$.  This is a criterion that can be numerically verified for medium-sized values of $n$, if the $b_1,\dots,b_n$ are chosen so that the matrix coefficients of $\mathbf{M}_1,\mathbf{M}_2$ are explicitly computable.

In order to facilitate computations, it is natural to work with bases $b_1,\dots,b_n$ of symmetric polynomials.  We have the following basic integration identity:

\begin{lemma}[Beta function identity]\label{bfi}  For any non-negative $a,a_1,\dots,a_k$, we have
$$ \int_{{\mathcal R}_k} (1-t_1-\dots-t_k)^a t_1^{a_1} \dots t_k^{a_k}\ dt_1 \dots dt_k = \frac{\Gamma(a+1) \Gamma(a_1+1) \dots \Gamma(a_k+1)}{\Gamma(a_1+\dots+a_k+k+a+1)}$$
where $\Gamma(s) := \int_0^\infty t^{s-1} e^{-t}\ dt$ is the Gamma function.  In particular, if $a_1,\dots,a_k$ are natural numbers, then
$$ \int_{{\mathcal R}_k} (1-t_1-\dots-t_k)^a t_1^{a_1} \dots t_k^{a_k}\ dt_1 \dots dt_k = \frac{a! a_1! \dots a_k!}{(a_1+\dots+a_k+k+a)!}.$$
\end{lemma}

\begin{proof}
Since
$$ \int_{{\mathcal R}_k} (1-t_1-\dots-t_k)^a t_1^{a_1} \dots t_k^{a_k}\ dt_1 \dots dt_k = a \int_{{\mathcal R}_{k+1}} t_1^{a_1} \dots t_k^{a_k} t_{k+1}^{a-1}\ dt_1 \dots dt_{k+1}$$
we see that to establish the lemma it suffices to do so in the case $a=0$.

If we write
$$ X := \int_{t_1+\dots+t_k=1} t_1^{a_1} \dots t_k^{a_k}\ dt_1 \dots dt_{k-1}$$
then by homogeneity we have
$$ r^{a_1+\dots+a_k+k-1} X = \int_{t_1+\dots+t_k=r} t_1^{a_1} \dots t_k^{a_k}\ dt_1 \dots dt_{k-1}$$
for any $r > 0$, and hence on integrating $r$ from $0$ to $1$ we conclude that
$$ \frac{X}{a_1+\dots+a_k+k} = \int_{{\mathcal R}_k} t_1^{a_1} \dots t_k^{a_k}\ dt_1 \dots dt_k.$$
On the other hand, if we multiply by $e^{-r}$ and integrate $r$ from $0$ to $\infty$, we obtain instead
$$ \int_0^\infty r^{a_1+\dots+a_k+k-1} X e^{-r}\ dr = \int_{[0,+\infty)^k} t_1^{a_1} \dots t_k^{a_k} e^{-t_1-\dots-t_k}\ dt_1 \dots dt_k.$$
Using the definition of the Gamma function, this becomes
$$ \Gamma(a_1+\dots+a_k+k) X = \Gamma(a_1+1) \dots \Gamma(a_k+1)$$
and the claim follows.
\end{proof}

Define a \emph{signature} to be a non-increasing sequence $\alpha = (\alpha_1,\alpha_2,\dots,\alpha_k)$ of natural numbers; for brevity we omit zeroes, thus for instance if $k=6$, then $(2,2,1,1,0,0)$ will be abbreviated as $(2,2,1,1)$.  The number of non-zero elements of $\alpha$ will be called the \emph{length} of the signature $\alpha$, and as usual the \emph{degree} of $\alpha$ will be $\alpha_1+\dots+\alpha_k$.  For each signature $\alpha$, we then define the symmetric polynomials $P_\alpha = P^{(k)}_\alpha$ by the formula
$$
P_\alpha(t_1,\dots,t_k) = \sum_{a: s(a)=\alpha} t_1^{a_1} \dots t_k^{a_k}$$
where the summation is over all tuples $a = (a_1,\dots,a_k)$ whose non-increasing rearrangement $s(a)$ is equal to $\alpha$.  Thus for instance
\begin{align*}
P_{(1)}(t_1,\dots,t_k) &= t_1 + \dots + t_k \\
P_{(2)}(t_1,\dots,t_k) &= t_1^2 + \dots + t_k^2 \\
P_{(1,1)}(t_1,\dots,t_k) &= \sum_{1 \leq i < j \leq k} t_i t_j \\
P_{(2,1)}(t_1,\dots,t_k) &= \sum_{1 \leq i < j \leq k} t_i^2 t_j + t_i t_j^2
\end{align*}
and so forth.  Clearly, the $P_\alpha$ form a linear basis for the symmetric polynomials of $t_1,\dots,t_k$.  Observe that if $\alpha = (\alpha',1)$ is a signature containing $1$, then one can express $P_\alpha$ as $P_{(1)} P_{\alpha'}$ minus a linear combination of polynomials $P_\beta$ with the length of $\beta$ less than that of $\alpha$.  This implies that the functions $P_{(1)}^a P_\alpha$, with $a \geq 0$ and $\alpha$ avoiding $1$, are also a basis for the symmetric polynomials.  Equivalently, the functions $(1-P_{(1)})^a P_\alpha$ with $a \geq 0$ and $\alpha$ avoiding $1$ form a basis.

After extensive experimentation, we have discovered that a good basis $b_1,\dots,b_n$ to use for the above problem comes by setting the $b_i$ to be all the symmetric polynomials of the form $(1-P_{(1)})^a P_\alpha$, where $a \geq 0$ and $\alpha$ consists entirely of even numbers, whose total degree $a + \alpha_1+\dots +\alpha_k$ is less than or equal to some chosen threshold $d$.  For such functions, the coefficients of $\mathbf{M}_1,\mathbf{M}_2$ can be computed exactly using Lemma \ref{bfi}.

More explicitly, first we quickly compute a look-up table for the structure constants $c_{\alpha,\beta,\gamma}\in \Z$ derived from simple products of the form
\[
P_{\alpha}P_{\beta}=\sum_{\gamma}c_{\alpha,\beta,\gamma}P_{\gamma}
\]
where $\deg(\alpha)+\deg(\beta)\leq d$.  Using this look-up table we rewrite the integrands of the entries of the matrices in (\ref{Eq:M1}) and (\ref{Eq:M2}) as integer linear combinations of nearly ``pure'' monomials of the form $(1-P_{(1)})^a t_1^{a_1}\dots t_k^{a_k}$.  We then calculate the entries of $\mathbf{M}_1$ and $\mathbf{M}_2$, as exact rational numbers, using Lemma \ref{bfi}.

We next run a generalized eigenvector routine on (real approximations to) $\mathbf{M}_1$ and $\mathbf{M}_2$ to find a vector $\mathbf{a}'$ which nearly maximizes the quantity $C$ in (\ref{ama}).  Taking a rational approximation $\mathbf{a}$ to $\mathbf{a}'$, we then do the quick (and exact) arithmetic to verify that (\ref{ama}) holds for some constant $C>4$.  This generalized eigenvector routine is time-intensive when the sizes of $\mathbf{M}_1$ and $\mathbf{M}_2$ are large (say, bigger than $1500\times 1500$), and in practice  is the most computationally intensive step of our calculation.  When one does not care about an exact arithmetic proof that $C>4$, instead one can run a test for positive-definiteness for the matrix $C\mathbf{M}_1-\mathbf{M}_2$, which is usually much faster and less RAM intensive.

Using this method, we were able to demonstrate $M_{54} > 4.00238$, thus establishing Theorem \ref{mlower}(vii).  We took $d=23$ and imposed the restriction on signatures $\alpha$ that they be composed only of even numbers. It is likely that $d=22$ would suffice in the absence of this restriction on signatures, but we found that the gain in $M_{54}$ from lifting this restriction is typically only in the region of $0.005$, whereas the execution time is increased by a large factor. We do not have a good understanding of why this particular restriction on signatures is so inexpensive in terms of the trade-off between the accuracy of $M$-values and computational complexity.  The total run-time for this computation was under one hour.

We now describe a second choice for the basis elements $b_1,\dots,b_n$, which uses the Krylov subspace method; it gives faster and more efficient numerical results than the previous basis, but does not seem to extend as well to more complicated variational problems such as $M_{k,\eps}$.  We introduce the linear operator ${\mathcal L}: L^2({\mathcal R}_k) \to L^2({\mathcal R}_k)$ defined by
$$ {\mathcal L} f( t_1,\dots,t_k) \coloneqq \sum_{i=1}^k \int_0^{1-t_1-\dots-t_{i-1}-t_{i+1}-\dots-t_k} f(t_1,\dots,t_{i-1},t'_i,t_{i+1},\dots,t_k)\ dt'_i.$$
This is a self-adjoint and positive semi-definite operator on $L^2({\mathcal R}_k)$.  For symmetric $b_1,\dots,b_n \in L^2({\mathcal R}_k)$, one can then write
\begin{align*}
\mathbf{M}_1 &= \left( \langle b_i, b_j \rangle \right)_{1 \leq i,j \leq n} \\
\mathbf{M}_2 &= \left( \langle {\mathcal L} b_i, b_j \rangle \right)_{1 \leq i,j \leq n}.
\end{align*}
If we then choose
$$ b_i \coloneqq {\mathcal L}^{i-1} 1$$
where $1$ is the unit constant function on ${\mathcal R}_k$, then the matrices $\mathbf{M}_1,\mathbf{M}_2$ take the Hankel form
\begin{align*}
\mathbf{M}_1 &= \left( \langle {\mathcal L}^{i+j-2} 1, 1 \rangle \right)_{1 \leq i,j \leq n} \\
\mathbf{M}_2 &= \left( \langle {\mathcal L}^{i+j-1} 1, 1 \rangle \right)_{1 \leq i,j \leq n},
\end{align*}
and so can be computed entirely in terms of the $2n$ numbers $\langle {\mathcal L}^i 1, 1 \rangle$ for $i=0,\dots,2n-1$.

The operator ${\mathcal L}$ maps symmetric polynomials to symmetric polynomials; for instance, one has
\begin{align*}
{\mathcal L} 1 &= k - (k-1) 	P_{(1)} \\
{\mathcal L} P_{(1)} &= \frac{k}{2} - \frac{k-1}{2} P_{(2)} - (k-2) P_{(1,1)}
\end{align*}
and so forth.  From this and Lemma \ref{bfi}, the quantities $\langle {\mathcal L}^i 1, 1 \rangle$ are explicitly computable rational numbers; for instance, one can calculate
\begin{align*}
\langle 1, 1 \rangle &= \frac{1}{k!} \\
\langle {\mathcal L} 1, 1\rangle &= \frac{2k}{(k+1)!} \\
\langle {\mathcal L}^2 1, 1\rangle &= \frac{k (5k+1)}{(k+2)!} \\
\langle {\mathcal L}^3 1, 1\rangle &= \frac{2k^2 (7k+5)}{(k+3)!}
\end{align*}
and so forth.

With Maple, we were able to compute $\langle {\mathcal L}^i 1, 1 \rangle$ for $i \leq 50$ and $k \leq 100$, leading to lower bounds on $M_k$ for these values of $k$, a selection of which are given in Table \ref{tab}.

\begin{table}
\centering
 \caption{Selected lower bounds on $M_k$ obtained from the Krylov subspace method, with the $\frac{k}{k-1} \log k$ upper bound displayed for comparison.}
 \label{tab}
  \begin{tabular}{lll}
   \toprule
     $k$    & Lower bound on $M_k$   & $\frac{k}{k-1} \log k$   \\
   \midrule
           2 & 1.38593            & 1.38630                \\
           3 & 1.64644            & 1.64792              \\ [0.5ex]
           4 & 1.84540  	        & 1.84840             \\
           5 & 2.00714  	        & 2.01180             \\ [0.5ex]
          10 & 2.54547            & 2.55843              \\
				  20 & 3.12756            & 3.15341              \\  [0.5ex]
          30 & 3.48313            & 3.51849              \\
          40 & 3.73919            & 3.78347            \\  [0.5ex]
          50 & 3.93586            & 3.99187              \\
          53 & 3.98621            & 4.04665            \\  [0.5ex]
          54 & 4.00223            & 4.06425            \\
          60 & 4.09101            & 4.16375            \\  [0.5ex]
				 100 & 4.46424            & 4.65169           \\
   \bottomrule
  \end{tabular}
\end{table}

\subsection{Bounding \texorpdfstring{$M_{k,\eps}$}{Mk,epsilon} for medium \texorpdfstring{$k$}{k}}\label{mkeps-sec}

When bounding $M_{k,\eps}$, we have not been able to implement the Krylov method, because the analogue of ${\mathcal L}^i 1$ in this context is piecewise polynomial instead of polynomial, and we were only able to compute it explicitly for very small values of $i$, such as $i=1,2,3$, which are insufficient for good numerics.  Thus, we rely on the previously discussed approach, in which symmetric polynomials are used for the basis functions.  Instead of computing integrals over the region ${\mathcal R}_k$ we pass to the regions $(1\pm \varepsilon)\mathcal{R}_k$.  In order to apply Lemma \ref{bfi} over these regions, this necessitates working with a slightly different basis of polynomials.  We chose to work with those polynomials of the form $(1+\varepsilon-P_{(1)})^a P_{\alpha}$, where $\alpha$ is a signature with no 1's.  Over the region $(1+\varepsilon)\mathcal{R}_k$, a single change of variables converts the needed integrals into those of the form in Lemma \ref{bfi}, and we can then compute the entries of $\mathbf{M}_1$.

On the other hand, over the region $(1-\varepsilon)\mathcal{R}_k$ we instead want to work with polynomials of the form $(1-\varepsilon-P_{(1)})^aP_{\alpha}$.  Since $(1+\varepsilon-P_{(1)})^a=(2\varepsilon +(1-\varepsilon-P_{(1)}))^a$, an expansion using the binomial theorem allows us to convert from our given basis to polynomials of the needed form.

With these modifications, and calculating as in the previous section, we find that $M_{50,1/25}>4.00124$ if $d=25$ and $M_{50,1/25}>4.0043$ if $d=27$, thus establishing Theorem \ref{mke-lower}(i).  As before, we found it optimal to restrict signatures to contain only even entries, which greatly reduced execution time while only reducing $M$ by a few thousandths.

One surprising additional computational difficulty introduced by allowing $\varepsilon>0$ is that the ``complexity'' of $\varepsilon$ as a rational number affects the run-time of the calculations.  We found that choosing $\varepsilon=1/m$ (where $m\in \Z$ has only small prime factors) reduces this effect.

A similar argument gives $M_{51,1/50} > 4.00156$, thus establishing Theorem \ref{mke-lower}(xiii). In this case our polynomials were of maximum degree $d= 22$.

Code and data for these calculations may be found at 

\centerline{\url{www.dropbox.com/sh/0xb4xrsx4qmua7u/WOhuo2Gx7f/Polymath8b}.}

\subsection{Bounding \texorpdfstring{$M_{4,\eps}$}{M4,0.18}}\label{4d}

We now prove Theorem \ref{mke-lower}(xii$'$), which can be established by a direct numerical calculation.  We introduce the explicit function $F: [0,+\infty)^4 \to \R$ defined by
$$ F(t_1,t_2,t_3,t_4) \coloneqq (1 - \alpha (t_1+t_2+t_3+t_4)) \onef_{t_1+t_2+t_3+t_4 \leq 1+\eps}$$
with $\eps \coloneqq 0.168$ and $\alpha \coloneqq 0.784$.  As $F$ is symmetric in $t_1,t_2,t_3,t_4$, we have $J_{i,1-\eps}(F) = J_{1,1-\eps}(F)$, so to show Theorem \ref{mke-lower}(xii$'$) it will suffice to show that
\begin{equation}\label{jf}
 \frac{4 J_{1,1-\eps}(F)}{I(F)} > 2.00558.
\end{equation}
By making the change of variables $s=t_1+t_2+t_3+t_4$ we see that
\begin{align*}
I(F) &= \int_{t_1+t_2+t_3+t_4 \leq 1+\eps} (1-\alpha(t_1+t_2+t_3+t_4))^2\ dt_1 dt_2 dt_3 dt_4 \\
&= \int_0^{1+\eps} (1-\alpha s)^2 \frac{s^3}{3!}\ ds\\
&= \alpha^2 \frac{(1+\eps)^6}{36} - \alpha \frac{(1+\eps)^5}{15} + \frac{(1+\eps)^4}{24} \\
&= 0.00728001347\dots
\end{align*}
and similarly by making the change of variables $u = t_1+t_2+t_3$
\begin{align*}
J_{1,1-\eps}(F) &= \int_{t_1+t_2+t_3 \leq 1-\eps} (\int_0^{1+\eps-t_1-t_2-t_3} (1-\alpha(t_1+t_2+t_3+t_4))\ dt_4)^2 dt_1 dt_2 dt_3 \\
&= \int_0^{1-\eps} (\int_0^{1+\eps-u} (1-\alpha(u+t_4))\ dt_4)^2 \frac{u^2}{2!} du \\
&= \int_0^{1-\eps} (1+\eps-u)^2 (1-\alpha \frac{1+\eps+u}{2})^2 \frac{u^2}{2} du  \\
&= 0.003650160667\dots
\end{align*}
and so \eqref{jf} follows.

\begin{remark}  If one uses the truncated function
$$ \tilde F(t_1,t_2,t_3,t_4) \coloneqq F(t_1,t_2,t_3,t_4) \onef_{t_1,t_2,t_3,t_4 \leq 1}$$
in place of $F$, and sets $\eps$ to $0.18$ instead of $0.168$, one can compute that
$$ \frac{4 J_{1,1-\eps}(\tilde F)}{I(\tilde F)} > 2.00235.$$
Thus it is possible to establish Theorem \ref{mke-lower}(xii$'$) using a cutoff function $F'$ that is also supported in the unit cube $[0,1]^4$.  This allows for a slight simplification to the proof of $\DHL[4,2]$ assuming GEH, as one can add the additional hypothesis $S(F_{i_0})+S(G_{i_0}) < 1$ to Theorem \ref{nonprime-asym}(ii) in that case.
\end{remark}

\begin{remark} By optimising in $\eps$ and taking $F$ to be a symmetric polynomial of degree higher than $1$, one can get slightly better lower bounds for $M_{4,\eps}$; for instance setting $\eps = 5/21$ and choosing $F$ to be a cubic polynomial, we were able to obtain the bound $M_{4,\eps} \geq 2.05411$.  On the other hand, the best lower bound for $M_{3,\eps}$ that we were able to obtain was $1.91726$ (taking $\eps = 56/113$ and optimizing over cubic polynomials).  Again, see \url{www.dropbox.com/sh/0xb4xrsx4qmua7u/WOhuo2Gx7f/Polymath8b} for the relevant code and data.
\end{remark}

\subsection{Three-dimensional cutoffs}\label{3d}

In this section we establish Theorem \ref{piece}.  We relabel the variables $(t_1,t_2,t_3)$ as $(x,y,z)$, thus our task is to locate a piecewise polynomial function
$F: [0,+\infty)^3 \to \R$ supported on the simplex
$$ R := \left\{ (x,y,z) \in [0,+\infty)^3: x+y+z \leq \frac{3}{2}\right\}$$
and symmetric in the $x,y,z$ variables, obeying the vanishing marginal condition
\begin{equation}\label{fmor}
\int_0^\infty F(x,y,z)\ dz = 0
\end{equation}
whenever $x,y \geq 0$ with $x+y > 1+\eps$, and such that
\begin{equation}\label{tt1}
J(F) > 2 I(F)
\end{equation}
where
\begin{equation}\label{jf-def}
J(F) := 3\int_{x+y \leq 1-\eps} \left(\int_0^\infty F(x,y,z)\ dz\right)^2\ dx dy
\end{equation}
and
\begin{equation}\label{if-def}
I(F) := \int_R  F(x,y,z)^2\ dx dy dz
\end{equation}
and
$$\eps := 1/4.$$

Our strategy will be as follows.  We will decompose the simplex $R$ (up to null sets) into a carefully selected set of disjoint open polyhedra $P_1,\dots,P_m$ (in fact $m$ will be $60$), and on each $P_i$ we will take $F(x,y,z)$ to be a low degree polynomial $F_i(x,y,z)$ (indeed, the degree will never exceed $3$).  The left and right-hand sides of \eqref{tt1} become quadratic functions in the coefficients of the $F_i$.  Meanwhile, the requirement of symmetry, as well as the marginal requirement \eqref{fmor}, imposes some linear constraints on these coefficients.  In principle, this creates a finite-dimensional quadratic program, which one can try to solve numerically.  However, to make this strategy practical, one needs to keep the number of linear constraints imposed on the coefficients to be fairly small, as compared with the total number of coefficients.  To achieve this, the following properties on the polyhedra $P_i$ are desirable:

\begin{itemize}
\item (Symmetry) If $P_i$ is a polytope in the partition, then every reflection of $P_i$ formed by permuting the $x,y,z$ coordinates should also lie in the partition.
\item (Graph structure) Each polytope $P_i$ should be of the form
\begin{equation}\label{qi-form}
\{ (x,y,z): (x,y) \in Q_i; a_i(x,y) < z < b_i(x,y)\},
\end{equation}
where $a_i(x,y), b_i(x,y)$ are linear forms and $Q_i$ is a polygon.
\item (Epsilon splitting)  Each $Q_i$ is contained in one of the regions $\{ (x,y): x+y < 1-\eps\}$, $\{ (x,y): 1-\eps < x+y < 1+\eps \}$, or $\{ (x,y): 1+\eps < x+y < 3/2 \}$.
\end{itemize}

Observe that the vanishing marginal condition \eqref{fmor} now takes the form
\begin{equation}\label{vanish}
 \sum_{i: (x,y) \in Q_i} \int_{a_i(x,y)}^{b_i(x,y)} F_i(x,y,z)\ dz = 0
\end{equation}
for every $x,y > 0$ with $x+y > 1+\eps$.  If the set $\{i: (x,y) \in Q_i\}$ is fixed, then the left-hand side of \eqref{vanish} is a polynomial in $x,y$ whose coefficients depend linearly on the coefficients of the $F_i$, and thus \eqref{vanish} imposes a set of linear conditions on these coefficients for each possible set $\{ i: (x,y) \in Q_i\}$ with $x+y > 1+\eps$.

Now we describe the partition we will use.  This partition can in fact be used for all $\eps$ in the interval $[1/4,1/3]$, but the endpoint $\eps = 1/4$ has some simplifications which allowed for reasonably good numerical results.  To obtain the symmetry property, it is natural to split $R$ (modulo null sets) into six polyhedra $R_{xyz}, R_{xzy}, R_{yxz}, R_{yzx}, R_{zxy}, R_{zyx}$, where
\begin{align*}
 R_{xyz} &:= \{(x,y,z)\in R\ :\ x+y < y+z < z+x\} \\
&= \{(x,y,z): 0 < y < x < z; x+y+z \leq 3/2 \}
\end{align*}
and the other polyhedra are obtained by permuting the indices $x,y,z$, thus for instance
\begin{align*}
R_{yxz} &:= \{(x,y,z)\in R\ :\ y+x < x+z < z+y\} \\
&= \{(x,y,z): 0 < x < y < z; y+x+z \leq 3/2 \}.
\end{align*}

To obtain the epsilon splitting property, we decompose $R_{xyz}$ (modulo null sets) into eight sub-polytopes
\begin{align*}
A_{xyz} & =\{(x,y,z)\in R\ :\ x+y< y+z< z+x< 1-\eps\},\\
B_{xyz} & =\{(x,y,z)\in R\ :\ x+y< y+z< 1-\eps< z+x< 1+\eps\},\\
C_{xyz} & =\{(x,y,z)\in R\ :\ x+y< 1-\eps < y+z < z+x< 1+\eps\},\\
D_{xyz} & =\{(x,y,z)\in R\ :\ 1-\eps < x+y<  y+z < z+x< 1+\eps\},\\
E_{xyz} & =\{(x,y,z)\in R\ :\ x+y< y+z< 1-\eps <  1+\eps < z+x\},\\
F_{xyz} & =\{(x,y,z)\in R\ :\ x+y< 1-\eps < y+z <  1+\eps < z+x\},\\
G_{xyz} & =\{(x,y,z)\in R\ :\ x+y< 1-\eps  <  1+\eps < y+z < z+x\},\\
H_{xyz} & =\{(x,y,z)\in R\ :\ 1-\eps < x+y<  y+z < 1+\eps < z+x\};
\end{align*}
the other five polytopes $R_{xzy}, R_{yxz}, R_{yzx}, R_{zxy}, R_{zyx}$ are decomposed similarly, leading to a partition of $R$ into $6 \times 8 = 48$ polytopes.
This is almost the partition we will use; however there is a technical difficulty arising from the fact that some of the permutations of $F_{xyz}$ do not obey the graph structure property.  So we will split $F_{xyz}$ further, into the three pieces
\begin{align*}
S_{xyz} & = \{(x,y,z)\in F_{xyz}\ :\ z< 1/2+\eps\},\\
T_{xyz} & = \{(x,y,z)\in F_{xyz}\ :\ z> 1/2+\eps;x> 1/2-\eps \}, \\
U_{xyz} & = \{(x,y,z)\in F_{xyz}\ :\ x< 1/2-\eps \}.
\end{align*}
Thus $R_{xyz}$ is now partitioned into ten polytopes $A_{xyz},$ $B_{xyz},$ $C_{xyz},$ $D_{xyz}$, $E_{xyz}$, $S_{xyz}$, $T_{xyz}$, $U_{xyz}$, $G_{xyz}$, $H_{xyz}$, and similarly for permutations of $R_{xyz}$, leading to a decomposition of $R$ into $6 \times 10 = 60$ polytopes.


A symmetric piecewise polynomial function $F$ supported on $R$ can now be described (almost everywhere) by specifying a polynomial function $F\downharpoonright_{P}: P \to \R$ for the ten polytopes $P = A_{xyz}, B_{xyz}, C_{xyz}, D_{xyz}, E_{xyz},  S_{xyz}, T_{xyz}, U_{xyz}, G_{xyz}, H_{xyz}$, and then extending by symmetry, thus for instance
$$ F\downharpoonright_{A_{yzx}}(x,y,z) = F\downharpoonright_{A_{xyz}}(z,x,y).$$

As discussed earlier, the expressions $I(F), J(F)$ can now be written as quadratic forms in the coefficients of the $F\downharpoonright_P$, and the vanishing marginal condition \eqref{fmor} imposes some linear constraints on these coefficients.

Observe that the polytope $D_{xyz}$ and all of its permutations make no contribution to either the functional $J(F)$ or to the marginal condition \eqref{fmor}, and give a non-negative contribution to $I(F)$.  Thus without loss of generality we may assume that
$$ F\downharpoonright_{D_{xyz}} = 0.$$
However, the other nine polytopes $A_{xyz}, B_{xyz}, C_{xyz}, E_{xyz}, S_{xyz}, T_{xyz}, U_{xyz}, G_{xyz}, H_{xyz}$ have at least one permutation which gives a non-trivial contribution to either $J(F)$ or to \eqref{fmor}, and cannot be easily eliminated.

Now we compute $I(F)$.  By symmetry we have
$$ I(F) = 3! I(F \downharpoonright_{R_{xyz}} ) = 6 \sum_P I( F \downharpoonright_P ) $$
where $P$ ranges over the nine polytopes $A_{xyz}, B_{xyz}, C_{xyz}, E_{xyz}, S_{xyz}, T_{xyz}, U_{xyz}, G_{xyz}, H_{xyz}$.  A tedious but straightforward computation shows that
\begin{align*}
I(F\downharpoonright_{A_{xyz}}) & =\int_{x=0}^{1/2-\eps/2}\int_{y=0}^{x}\int_{z=x}^{1-\eps-x}F\downharpoonright_{A_{xyz}}^2\ dz\ dy\ dx \\
I(F\downharpoonright_{B_{xyz}}) &= \left(\int_{z=1/2-\eps/2}^{1/2+\eps/2}\int_{x=1-\eps-z}^{z} +\int_{z=1/2+\eps/2}^{1-\eps}\int_{x=1-\eps-z}^{1+\eps-z}\right) \int_{y=0}^{1-\eps-z}F\downharpoonright_{B_{xyz}}^2\ dy\ dx\ dz \\
I(F\downharpoonright_{C_{xyz}}) & = \left(\int_{y=0}^{1/2-3\eps/2} \int_{x=y}^{y+2\eps} + \int_{y=1/2-3\eps/2}^{1/2-\eps}\int_{x=y}^{1-\eps-y} \right) \int_{z=1-\eps-y}^{1+\eps-x} \\
&\quad + \int_{y=1/2-\eps}^{1/2-\eps/2}\int_{x=y}^{1-\eps-y}\int_{z=1-\eps-y}^{3/2-x-y} F\downharpoonright_{C_{xyz}}^2\ dz\ dx\ dy \\
I(F\downharpoonright_{E_{xyz}}) & = \int_{z=1/2+\eps/2}^{1-\eps}\int_{x=1+\eps-z}^{z}\int_{y=0}^{1-\eps-z} F\downharpoonright_{E_{xyz}}^2\ dy\ dx\ dz \\
I(F\downharpoonright_{S_{xyz}}) & =  \left(\int_{y=0}^{1/2-3\eps/2}\int_{z=1-\eps-y}^{1/2+\eps} +\int_{y=1/2-3\eps/2}^{1/2-\eps}\int_{z=y+2\eps}^{1/2+\eps} \right)\int_{x=1+\eps-z}^{1-\eps-y} F\downharpoonright_{S_{xyz}}^2\ dx\ dz\ dy \\
I(F\downharpoonright_{T_{xyz}}) &  = \left(\int_{z=1/2+\eps}^{1/2+2\eps}\int_{x=1+\eps-z}^{3/2-z} + \int_{z=1/2+2\eps}^{1+\eps}\int_{x=1/2-\eps}^{3/2-z}\right)\int_{y=0}^{3/2-x-z} F\downharpoonright_{T_{xyz}}^2\ dy\ dz\ dx \\
I(F\downharpoonright_{U_{xyz}}) & = \int_{x=0}^{1/2-\eps}\int_{y=0}^{x}\int_{z=1+\eps-x}^{1+\eps-y} F\downharpoonright_{U_{xyz}}^2\ dz\ dy\ dx\\
I(F\downharpoonright_{G_{xyz}}) & = \int_{x=0}^{1/2-\eps}\int_{y=0}^{x}\int_{z=1+\eps-y}^{3/2-x-y} F\downharpoonright_{G_{xyz}}^2\ dx\ dz\ dy
\end{align*}
and
\begin{align*}
I(F\downharpoonright_{H_{xyz}}) & =\left(\int_{x=1/2+\eps/2}^{1-\eps}\int_{y=1-\eps-x}^{3/2-2x} +\int_{x=1-\eps}^{3/4}\int_{y=0}^{3/2-2x} \right) \int_{z=x}^{3/2-x-y}\\
&\quad + \int_{x=1/2}^{1/2+\eps/2}\int_{y=1-\eps-x}^{1/2-\eps}\int_{z=1+\eps-x}^{3/2-x-y} F\downharpoonright_{H_{xyz}}^2\ dz\ dy\ dx.
\end{align*}

Now we consider the quantity $J(F)$.  Here we only have the symmetry of swapping $x$ and $y$, so that
$$ J(F) = 6 \int_{0 < y < x; x+y < 1-\eps} \left(\int_0^{3/2-x-y} F(x,y,z)\ dz\right)^2 dx dy.$$
The region of integration meets the polytopes $A_{xyz}$, $A_{yzx}$, $A_{zyx}$, $B_{xyz}$, $B_{zyx}$, $C_{xyz}$, $E_{xyz}$, $E_{zyx}$, $S_{xyz}$, $T_{xyz}$, $U_{xyz}$, and $G_{xyz}$.

Projecting these polytopes to the $(x,y)$-plane, we have the diagram:
\begin{center}\includegraphics[scale=1.5]{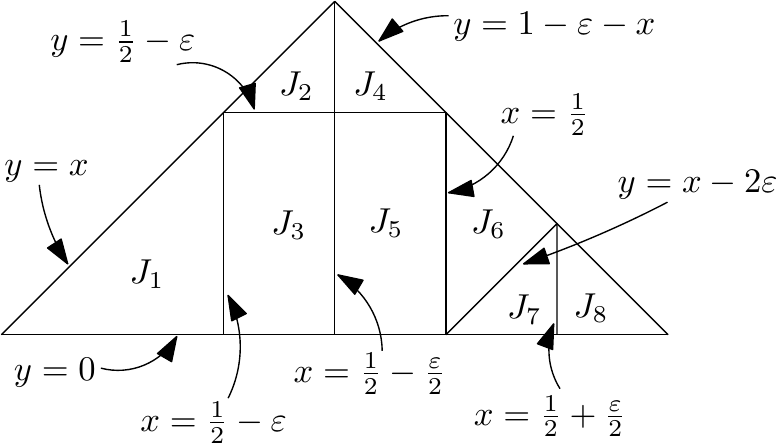}\end{center}

This diagram is drawn to scale in the case when $\varepsilon=1/4$, otherwise there is a separation between the $J_5$ and $J_7$ regions.  For these eight regions there are eight corresponding integrals $J_1,J_2,\ldots, J_8$, thus
$$ J(F) = 6( J_1 + \dots + J_8).$$

We have
\begin{align*}
J_1 & =\int_{x=0}^{1/2-\eps}\int_{y=0}^{x} \left(\int_{z=0}^{y}F\downharpoonright_{A_{yzx}}+\int_{z=y}^{x}F\downharpoonright_{A_{zyx}} +\int_{z=x}^{1-\eps-x}F\downharpoonright_{A_{xyz}} +\int_{z=1-\eps-x}^{1-\eps-y}F\downharpoonright_{B_{xyz}}  \right. \\
&\quad \left. +\int_{z=1-\eps-y}^{1+\eps-x}F\downharpoonright_{C_{xyz}} + \int_{z=1+\eps-x}^{1+\eps-y}F\downharpoonright_{U_{xyz}} + \int_{z=1+\eps-y}^{3/2-x-y}F\downharpoonright_{G_{xyz}}\ dz \right)^{2}\ dy\ dx.
\end{align*}
Next comes
\begin{align*}
J_2 & =\int_{x=1/2-\eps}^{1/2-\eps/2}\int_{y=1/2-\eps}^{x} \left(\int_{z=0}^{y}F\downharpoonright_{A_{yzx}}+\int_{z=y}^{x}F\downharpoonright_{A_{zyx}}+\int_{z=x}^{1-\eps-x}F\downharpoonright_{A_{xyz}} +\int_{z=1-\eps-x}^{1-\eps-y}F\downharpoonright_{B_{xyz}}\right. \\
&\quad \left. +\int_{z=1-\eps-y}^{3/2-x-y}F\downharpoonright_{C_{xyz}}\ dz \right)^{2}\ dy\ dx.
\end{align*}
Third is the piece
\begin{align*}
J_3 & =\int_{x=1/2-\eps}^{1/2-\eps/2}\int_{y=0}^{1/2-\eps} \left(\int_{z=0}^{y}F\downharpoonright_{A_{yzx}}+\int_{z=y}^{x}F\downharpoonright_{A_{zyx}} +\int_{z=x}^{1-\eps-x}F\downharpoonright_{A_{xyz}} +\int_{z=1-\eps-x}^{1-\eps-y}F\downharpoonright_{B_{xyz}} \right. \\
&\quad \left. + \int_{z=1-\eps-y}^{1+\eps-x}F\downharpoonright_{C_{xyz}} + \int_{z=1+\eps-x}^{3/2-x-y}F\downharpoonright_{T_{xyz}}\ dz \right)^{2}\ dy\ dx.
\end{align*}

We now have dealt with all integrals involving $A_{xyz}$, and all remaining integrals pass through $B_{zyx}$.  Continuing, we have
\begin{align*}
J_4 & =\int_{x=1/2-\eps/2}^{1/2}\int_{y=1/2-\eps}^{1-\eps-x} \left(\int_{z=0}^{y}F\downharpoonright_{A_{yzx}}+\int_{z=y}^{1-\eps-x}F\downharpoonright_{A_{zyx}} +\int_{z=1-\eps-x}^{x}F\downharpoonright_{B_{zyx}} +\int_{z=x}^{1-\eps-y}F\downharpoonright_{B_{xyz}}\right. \\
&\quad \left. +\int_{z=1-\eps-y}^{3/2-x-y}F\downharpoonright_{C_{xyz}}\ dz \right)^{2}\ dy\ dx.
\end{align*}
Another component is
\begin{align*}
J_5 & = \int_{x=1/2-\eps/2}^{1/2}\int_{y=0}^{1/2-\eps} \left(\int_{z=0}^{y}F\downharpoonright_{A_{yzx}}+\int_{z=y}^{1-\eps-x}F\downharpoonright_{A_{zyx}} \right. \\
&\quad \left. +\int_{z=1-\eps-x}^{x}F\downharpoonright_{B_{zyx}} +\int_{z=x}^{1-\eps-y}F\downharpoonright_{B_{xyz}} +\int_{z=1-\eps-y}^{1+\eps-x}F\downharpoonright_{C_{xyz}} + \int_{z=1+\eps-x}^{3/2-x-y}F\downharpoonright_{T_{xyz}}\ dz \right)^{2}\ dy\ dx.
\end{align*}

The most complicated piece is
\begin{align*}
J_6 & =\left(\int_{x=1/2}^{2\eps} \int_{y=0}^{1-\eps-x} + \int_{x=2\eps}^{1/2+\eps/2}\int_{y=x-2\eps}^{1-\eps-x} \right) \left(\int_{z=0}^{y}F\downharpoonright_{A_{yzx}}+\int_{z=y}^{1-\eps-x}F\downharpoonright_{A_{zyx}} +\int_{z=1-\eps-x}^{x}F\downharpoonright_{B_{zyx}}  \right. \\
&\quad \left. +\int_{z=x}^{1-\eps-y}F\downharpoonright_{B_{xyz}} +\int_{z=1-\eps-y}^{1+\eps-x}F\downharpoonright_{C_{xyz}} + \int_{z=1+\eps-x}^{1/2+\eps}F\downharpoonright_{S_{xyz}}+ \int_{z=1/2+\eps}^{3/2-x-y}F\downharpoonright_{T_{xyz}}\ dz \right)^{2}\ dy\ dx.
\end{align*}
Here we use $\left(\int_{x=1/2}^{2\eps} \int_{y=0}^{1-\eps-x} + \int_{x=2\eps}^{1/2+\eps/2}\int_{y=x-2\eps}^{1-\eps-x} \right) f(x,y)\ dy dx$ as an abbreviation for
$$ \int_{x=1/2}^{2\eps} \int_{y=0}^{1-\eps-x} f(x,y)\ dy dx + \int_{x=2\eps}^{1/2+\eps/2}\int_{y=x-2\eps}^{1-\eps-x} f(x,y)\ dy dx.$$

We have now exhausted $C_{xyz}$.  The seventh piece is
\begin{align*}
J_7 & = \int_{x=2\eps}^{1/2+\eps/2}\int_{y=0}^{x-2\eps} \left(\int_{z=0}^{y}F\downharpoonright_{A_{yzx}}+\int_{z=y}^{1-\eps-x}F\downharpoonright_{A_{zyx}} +\int_{z=1-\eps-x}^{x}F\downharpoonright_{B_{zyx}} \right. \\
&\quad \left. +\int_{z=x}^{1+\eps-x}F\downharpoonright_{B_{xyz}} + \int_{z=1+\eps-x}^{1-\eps-y}F\downharpoonright_{E_{xyz}} + \int_{1-\eps-y}^{1/2+\eps}F\downharpoonright_{S_{xyz}}+ \int_{1/2+\eps}^{3/2-x-y}F\downharpoonright_{T_{xyz}} \ dz \right)^{2}\ dy\ dx.
\end{align*}
Finally, we have
\begin{align*}
J_8 & = \int_{x=1/2+\eps/2}^{1-\eps}\int_{y=0}^{1-\eps-x} \left(\int_{z=0}^{y}F\downharpoonright_{A_{yzx}}+\int_{z=y}^{1-\eps-x}F\downharpoonright_{A_{zyx}} +\int_{z=1-\eps-x}^{1+\eps-x}F\downharpoonright_{B_{zyx}} \right. \\
&\quad \left. + \int_{z=1+\eps-x}^{x}F\downharpoonright_{E_{zyx}} + \int_{z=x}^{1-\eps-y}F\downharpoonright_{E_{xyz}} + \int_{1-\eps-y}^{1/2+\eps}F\downharpoonright_{S_{xyz}}+ \int_{1/2+\eps}^{3/2-x-y}F\downharpoonright_{T_{xyz}}\ dz \right)^{2}\ dy\ dx.
\end{align*}

In the case $\eps=1/4$, the marginal conditions \eqref{fmor} reduce to requiring
\begin{align}
\int_{z=0}^{3/2-x-y}F\downharpoonright_{G_{yzx}}\ dz &= 0\label{m1}\\
\int_{z=0}^{y}F\downharpoonright_{G_{yzx}} +\int_{z=y}^{3/2-x-y}F\downharpoonright_{G_{zyx}}\ dz &= 0\label{m2}\\
\int_{z=0}^{1+\varepsilon-x}F\downharpoonright_{U_{yzx}} + \int_{z=1+\varepsilon-x}^{y}F\downharpoonright_{G_{yzx}} +\int_{z=y}^{3/2-x-y}F\downharpoonright_{G_{zyx}}\ dz &= 0\label{m3}\\
\int_{z=0}^{1+\varepsilon-x}F\downharpoonright_{U_{yzx}} + \int_{z=1+\varepsilon-x}^{3/2-x-y}F\downharpoonright_{G_{yzx}}\ dz &= 0\label{m4}\\
\int_{z=0}^{3/2-x-y}F\downharpoonright_{T_{yzx}}\ dz &= 0\label{m5}\\
\int_{z=0}^{1-\varepsilon-x}F\downharpoonright_{E_{yzx}} + \int_{z=1-\varepsilon-x}^{1-\varepsilon-y}F\downharpoonright_{S_{yzx}} + \int_{z=1-\varepsilon-y}^{3/2-x-y}F\downharpoonright_{H_{yzx}}\ dz &= 0.\label{m7}
\end{align}
Each of these constraints is only required to hold for some portion of the parameter space $\{ (x,y): 1+\eps \leq x+y \leq 3/2 \}$, but as the left-hand sides are all polynomial functions in $x,y$ (using the signed definite integral $\int_b^a = -\int_a^b$), it is equivalent to require that all coefficients of these polynomial functions vanish.

Now we specify $F$.  After some numerical experimentation, we have found the simplest choice of $F$ that still achieves the desired goal comes by taking $F(x,y,z)$ to be a polynomial of degree $1$ on each of $E_{xyz}$, $S_{xyz}$, $H_{xyz}$, degree $2$ on $T_{xyz}$, vanishing on $D_{xyz}$, and degree $3$ on the remaining five relevant components of $R_{xyz}$.  After solving the quadratic program, rounding, and clearing denominators, we arrive at the choice
\begin{align*}
F\downharpoonright_{A_{xyz}} &:= -66+96 x-147 x^2+125 x^3+128 y-122 x y+104 x^2 y-275 y^2+394 y^3+99 z\\
&\quad -58 x z+63 x^2 z-98 y z+51 x y z+41 y^2 z-112 z^2+24 x z^2+72 y z^2+50 z^3 \\
F\downharpoonright_{B_{xyz}} &:= -41+52 x-73 x^2+25 x^3+108 y-66 x y+71 x^2 y-294 y^2+56 x y^2+363 y^3\\
&\quad +33 z+15 x z+22 x^2 z-40 y z-42 x y z+75 y^2 z-36 z^2-24 x z^2+26 y z^2+20 z^3 \\
F\downharpoonright_{C_{xyz}} &:= -22+45 x-35 x^2+63 y-99 x y+82 x^2 y-140 y^2+54 x y^2+179 y^3 \\
F\downharpoonright_{E_{xyz}} &:= -12+8 x+32 y \\
F\downharpoonright_{S_{xyz}} &:= -6+8 x+16 y \\
F\downharpoonright_{T_{xyz}} &:= 18-30 x+12 x^2+42 y-20 x y-66 y^2-45 z+34 x z+22 z^2 \\
F\downharpoonright_{U_{xyz}} &:= 94-1823 x+5760 x^2-5128 x^3+54 y-168 x^2 y+105 y^2+1422 x z-2340 x^2 z\\
&\quad -192 y^2 z-128 z^2-268 x z^2+64 z^3 \\
F\downharpoonright_{G_{xyz}} &:= 5274-19833 x+18570 x^2-5128 x^3-18024 y+44696 x y-20664 x^2 y+16158 y^2\\
&\quad -19056 x y^2-4592 y^3-10704 z+26860 x z-12588 x^2 z+24448 y z-30352 x y z\\
&\quad -10980 y^2 z+7240 z^2-9092 x z^2-8288 y z^2-1632 z^3 \\
F\downharpoonright_{H_{xyz}} &:= 8 z.
\end{align*}
One may compute that
$$ I(F) = \frac{62082439864241}{507343011840}$$
and
$$ J(F) = \frac{9933190664926733}{40587440947200}$$
with all the marginal conditions \eqref{m1}-\eqref{m7} obeyed, thus
$$ \frac{J(F)}{I(F)} = 2 + \frac{286648173}{4966595189139280}$$
and \eqref{tt1} follows.

\section{The parity problem}\label{parity-sec}

In this section we argue why the ``parity barrier'' of Selberg \cite{selberg} prohibits sieve-theoretic methods, such as the ones in this paper, from obtaining any bound on $H_1$ that is stronger than $H_1 \leq 6$, even on the assumption of strong distributional conjectures such as the generalized Elliott-Halberstam conjecture $\GEH[\vartheta]$, and even if one uses sieves other than the Selberg sieve.  Our discussion will be somewhat informal and heuristic in nature.

We begin by briefly recalling how the bound $H_1 \leq 6$ on GEH (i.e., Theorem \ref{main}(xii)) was proven.  This was deduced from the claim $\DHL[3,2]$, or more specifically from the claim that the set
\begin{equation}\label{A-def}
A := \{ n \in \N: \hbox{ at least two of } n, n+2, n+6 \hbox{ are prime} \}
\end{equation}
was infinite.

To do this, we (implicitly) established a lower bound
$$ \sum_n \nu(n) \onef_A(n) > 0$$
for some non-negative weight $\nu: \N \to \R^+$ supported on $[x,2x]$ for a sufficiently large $x$.  This bound was in turn established (after a lengthy sieve-theoretic analysis, and with a carefully chosen weight $\nu$) from upper bounds on various discrepancies.  More precisely, one required good upper bounds (on average) for the expressions
\begin{equation}\label{flip}
 \left|\sum_{x \leq n \leq 2x: n = a\ (q)} f(n+h) - \frac{1}{\phi(q)} \sum_{x \leq n \leq 2x: (n+h,q)=1} f(n+h)\right|
\end{equation}
for all $h \in \{0,2,6\}$ and various residue classes $a\ (q)$ with $q \leq x^{1-\eps}$ and arithmetic functions $f$, such as the constant function $f=1$, the von Mangoldt function $f = \Lambda$, or Dirichlet convolutions $f = \alpha \star \beta$ of the type considered in Claim \ref{geh-def}.  (In the presentation of this argument in previous sections, the shift by $h$ was eliminated using the change of variables $n' = n+h$, but for the current discussion it is important that we do not use this shift.)  One also required good asymptotic control on the main terms
\begin{equation}\label{flop}
\sum_{x \leq n \leq 2x: (n+h,q)=1} f(n+h).
\end{equation}

Once one eliminates the shift by $h$, an inspection of these arguments reveals that they would be equally valid if one inserted a further non-negative weight $\omega: \N \to \R^+$ in the summation over $n$.  More precisely, the above sieve-theoretic argument would also deduce the lower bound
$$ \sum_n \nu(n) \onef_A(n) \omega(n) > 0$$
if one had control on the weighted discrepancies
\begin{equation}\label{flip-weight}
 \left|\sum_{x \leq n \leq 2x: n= a\ (q)} f(n+h) \omega(n) - \frac{1}{\phi(q)} \sum_{x \leq n \leq 2x: (n+h,q)=1} f(n+h) \omega(n)\right|
\end{equation}
and on the weighted main terms
\begin{equation}\label{flop-weight}
\sum_{x \leq n \leq 2x: (n+h,q)=1} f(n+h) \omega(n)
\end{equation}
that were of the same form as in the unweighted case $\omega=1$.

Now suppose for instance that one was trying to prove the bound $H_1 \leq 4$.  A natural way to proceed here would be to replace the set $A$ in \eqref{A-def} with the smaller set
\begin{equation}\label{app}
 A' := \{ n \in \N: n, n+2 \hbox{ are both prime} \} \cup \{ n \in \N: n+2, n+6 \hbox{ are both prime} \}
\end{equation}
and hope to establish a bound of the form
$$ \sum_n \nu(n) \onef_{A'}(n) > 0$$
for a well-chosen function $\nu: \N \to \R^+$ supported on $[x,2x]$, by deriving this bound from suitable (averaged) upper bounds on the discrepancies \eqref{flip} and control on the main terms \eqref{flop}.  If the arguments were sieve-theoretic in nature, then (as in the $H_1 \leq 6$ case), one could then also deduce the lower bound
\begin{equation}\label{ap-lower}
 \sum_n \nu(n) \onef_{A'}(n) \omega(n) > 0
\end{equation}
for any non-negative weight $\omega: \N \to \R^+$, provided that one had the same control on the weighted discrepancies \eqref{flip-weight} and weighted main terms \eqref{flop-weight} that one did on \eqref{flip}, \eqref{flop}.

We apply this observation to the weight
\begin{align*}
\omega(n) &:= (1 - \lambda(n) \lambda(n+2)) (1 - \lambda(n+2) \lambda(n+6)) \\
&=  1 - \lambda(n)\lambda(n+2) - \lambda(n+2)\lambda(n+6) + \lambda(n) \lambda(n+6)
\end{align*}
where $\lambda(n) := (-1)^{\Omega(n)}$ is the Liouville function.  Observe that $\omega$ vanishes for any $n\in A'$, and hence
\begin{equation}\label{ap-none}
\sum_n \nu(n) \onef_{A'}(n) \omega(n) = 0
\end{equation}
for any $\nu$.  On the other hand, the ``M\"obius randomness law'' (see e.g. \cite{ik}) predicts a significant amount of cancellation for any non-trivial sum involving the M\"obius function $\mu$, or the closely related Liouville function $\lambda$.  For instance, the expression
$$ \sum_{x \leq n \leq 2x: n = a\ (q)} \lambda(n+h)$$
is expected to be very small (of size\footnote{Indeed, one might be even more ambitious and conjecture a square-root cancellation $\lessapprox \sqrt{x/q}$ for such sums (see \cite{mont} for some similar conjectures), although such stronger cancellations generally do not play an essential role in sieve-theoretic computations.} $O( \frac{x}{q} \log^{-A} x)$ for any fixed $A$) for any residue class $a\ (q)$ with $q \leq x^{1-\eps}$, and any $h \in \{0,2,6\}$; similarly for more complicated expressions such as
$$ \sum_{x \leq n \leq 2x: n = a\ (q)} \lambda(n+2) \lambda(n+6)$$
or
$$ \sum_{x \leq n \leq 2x: n = a\ (q)} \Lambda(n) \lambda(n+2) \lambda(n+6)$$
or more generally
$$ \sum_{x \leq n \leq 2x: n = a\ (q)} f(n) \lambda(n+2) \lambda(n+6)$$
where $f$ is a Dirichlet convolution $\alpha \star \beta$ of the form considered in Claim \ref{geh-def}.  Similarly for expressions such as
$$ \sum_{x \leq n \leq 2x: n = a\ (q)} f(n) \lambda(n) \lambda(n+2);$$
note from the complete multiplicativity of $\lambda$ that $(\alpha \star \beta) \lambda = (\alpha \lambda) \star (\beta \lambda)$, so if $f$ is of the form in Claim \ref{geh-def}, then $f\lambda$ is also.  In view of these observations (and similar observations arising from permutations of $\{0,2,6\}$), we conclude (heuristically, at least) that all the bounds that are believed to hold for \eqref{flip}, \eqref{flop} should also hold (up to minor changes in the implied constants) for \eqref{flip-weight}, \eqref{flop-weight}.  Thus, if the bound $H_1 \leq 4$ could be proven in a sieve-theoretic fashion, one should be able to conclude the bound \eqref{ap-lower}, which is in direct contradiction to \eqref{ap-none}.

\begin{remark}  Similar arguments work for any set of the form
$$ A_H := \{ n \in \N: \exists n \leq p_1 < p_2 \leq n+H; p_1,p_2 \hbox{ both prime}, p_2 - p_1 \leq 4 \}$$
and any fixed $H > 0$, to prohibit any non-trivial lower bound on $\sum_n \nu(n) \onef_{A_H}(n)$ from sieve-theoretic methods. Indeed, one uses the weight
$$ \omega(n) := \prod_{0 \leq i \leq i' \leq H; (n+i,3) = (n+i',3) = 1; i'-i \leq 4} (1 - \lambda(n+i) \lambda(n+i'));$$
we leave the details to the interested reader.  This seems to block any attempt to use any argument based only on the distribution of the prime numbers and related expressions in arithmetic progressions to prove $H_1 \leq 4$.
\end{remark}

The same arguments of course also prohibit a sieve-theoretic proof of the twin prime conjecture $H_1 = 2$.   In this case one can use the simpler weight $\omega(n) = 1 - \lambda(n) \lambda(n+2)$ to rule out such a proof, and the argument is essentially due to Selberg \cite{selberg}.

Of course, the parity barrier could be circumvented if one were able to introduce stronger sieve-theoretic axioms than the ``linear'' axioms currently available (which only control sums of the form \eqref{flip} or \eqref{flop}).  For instance, if one were able to obtain non-trivial bounds for ``bilinear'' expressions such as
$$ \sum_{x \leq n \leq 2x} f(n) \Lambda(n+2) = \sum_d \sum_m \alpha(d) \beta(m) \onef_{[x,2x]}(dm) \Lambda(dm+2)$$
for functions $f = \alpha \star \beta$ of the form in Claim \ref{geh-def}, then (by a modification of the proof of Proposition \ref{geh-eh}) one would very likely obtain non-trivial bounds on
$$ \sum_{x \leq n \leq 2x} \Lambda(n) \Lambda(n+2)$$
which would soon lead to a proof of the twin prime conjecture.  Unfortunately, we do not know of any plausible way to control such bilinear expressions.  (Note however that there are some other situations in which bilinear sieve axioms may be established, for instance in the argument of Friedlander and Iwaniec \cite{fi} establishing an infinitude of primes of the form $a^2+b^4$.)

\section{Additional remarks}\label{remarks-sec}

The proof of Theorem \ref{main-dhl}(xii) may be modified to establish the following variant:

\begin{proposition}  Assume the generalized Elliott-Halberstam conjecture $\GEH[\vartheta]$ for all $0 < \vartheta < 1$.  Let $0 < \eps < 1/2$ be fixed.  Then if $x$ is a sufficiently large multiple of $6$, there exists a natural number $n$ with $\eps x \leq n \leq (1-\eps) x$ such that at least two of $n, n-2, x-n$ are prime.  Similarly if $n-2$ is replaced by $n+2$.
\end{proposition}

Note that if at least two of $n,n-2,x-n$ are prime, then either $n,n+2$ are twin primes, or else at least one of $x,x-2$ is expressible as the sum of two primes, and Theorem \ref{disj} easily follows.

\begin{proof} (Sketch) We just discuss the case of $n-2$, as the $n+2$ case is similar. Observe from the Chinese remainder theorem (and the hypothesis that $x$ is divisible by $6$) that one can find a residue class $b\ (W)$ such that $b, b-2, x-b$ are all coprime to $W$ (in particular, one has $b=1\ (6)$).   By a routine modification of the proof of Lemma \ref{crit}, it suffices to  find a non-negative weight
  function $\nu \colon \N \to \R^+$ and fixed quantities $\alpha > 0$ and $\beta_1,\beta_2,\beta_3 \geq
  0$, such that one has the
  asymptotic upper bound
$$
 \sum_{\substack{\eps x \leq n \leq (1-\eps) x\\ n = b\ (W)}} \nu(n) \leq {\mathfrak S} (\alpha+o(1)) B^{-k} \frac{(1-2\eps) x}{W},
$$
the asymptotic lower bounds
\begin{align*}
  \sum_{\substack{\eps x \leq n \leq (1-\eps) x\\ n = b\ (W)}} \nu(n) \theta(n) &\geq {\mathfrak S} (\beta_1-o(1)) B^{1-k} \frac{(1-2\eps) x}{\phi(W)} \\
  \sum_{\substack{\eps x \leq n \leq (1-\eps) x\\ n = b\ (W)}} \nu(n) \theta(n+2) &\geq {\mathfrak S} (\beta_2-o(1)) B^{1-k} \frac{(1-2\eps) x}{\phi(W)} \\
  \sum_{\substack{\eps x \leq n \leq (1-\eps) x\\ n = b\ (W)}} \nu(n) \theta(x-n) &\geq {\mathfrak S} (\beta_3-o(1)) B^{1-k} \frac{(1-2\eps) x}{\phi(W)} 
\end{align*}
and the inequality
$$ \beta_1+\beta_2+\beta_3 > 2 \alpha,$$
where ${\mathfrak S}$ is the singular series
$$ {\mathfrak S} := \prod_{p|x(x-2); p > w} \frac{p}{p-1}.$$
We select $\nu$ to be of the form
$$ \nu(n) = \left( \sum_{j=1}^J c_j \lambda_{F_{j,1}}(n) \lambda_{F_{j,2}}(n+2) \lambda_{F_{j,3}}(x-n) \right)^2 $$
for various fixed coefficients $c_1,\dots,c_J \in \R$ and fixed smooth compactly supported functions $F_{j,i}: [0,+\infty) \to \R$ with $j=1,\dots,J$ and $i=1,\dots,3$.  It is then routine\footnote{One new technical difficulty here is that some of the various moduli $[d_j,d'_j]$ arising in these arguments are not required to be coprime at primes $p > w$ dividing $x$ or $x-2$; this requires some modification to Lemma \ref{mul-asym} that ultimately leads to the appearance of the singular series ${\mathfrak S}$.  However, these modifications are quite standard, and we do not give the details here.} to verify that analogues of Theorem \ref{prime-asym} and Theorem \ref{nonprime-asym} hold for the various components of $\nu$, with the role of $x$ in the right-hand side replaced by $(1-2\eps) x$, and the claim then follows by a suitable modification of Theorem \ref{epsilon-beyond}, taking advantage of the function $F$ constructed in Theorem \ref{piece}.
\end{proof}

It is likely that the bounds in Theorem \ref{main} can be improved further by refining the sieve-theoretic methods employed in this paper, with the exception of part (xii) for which the parity problem prevents further improvement, as discussed in Section \ref{parity-sec}.  We list some possible avenues to such improvements as follows:

\begin{enumerate}
\item In Theorem \ref{mke-lower}, the bound $M_{k,\eps} > 4$ was obtained for some $\eps>0$ and $k=50$.  It is possible that $k$ could be lowered slightly, for instance to $k=49$, by further numerical computations, but we were only barely able to establish the $k=50$ bound after two weeks of computation.  However, there may be a more efficient way to solve the required variational problem (e.g. by selecting a more efficient basis than the symmetric monomial basis) that would allow one to advance in this direction; this would improve the bound $H_1 \leq 246$ slightly.  Extrapolation of existing numerics also raises the possibility that $M_{53}$ exceeds $4$, in which case the bound of $270$ in Theorem \ref{main}(vii) could be lowered to $264$.
\item To reduce $k$ (and thus $H_1$) further, one could try to solve another variational problem, such as the one arising in Theorem \ref{maynard-trunc} or in Theorem \ref{epsilon-beyond}, rather than trying to lower bound $M_k$ or $M_{k,\eps}$.  It is also possible to use the more complicated versions of $\MPZ[\varpi,\delta]$ established in \cite{polymath8a} (in which the modulus $q$ is assumed to be densely divisible rather than smooth) to replace the truncated simplex appearing in Theorem \ref{maynard-trunc} with a more complicated region (such regions also appear implicitly in \cite[\S 4.5]{polymath8a}).  However, in the medium-dimensional setting $k \approx 50$, we were not able to accurately and rapidly evaluate the various integrals associated to these variational problems when applied to a suitable basis of functions.  One key difficulty here is that whereas polynomials appear to be an adequate choice of basis for the $M_k$, an analysis of the Euler-Lagrange equation reveals that one should use piecewise polynomial basis functions instead for more complicated variational problems such as the $M_{k,\eps}$ problem (as was done in the three-dimensional case in Section \ref{3d}), and these are difficult to work with in medium dimensions.  From our experience with the low $k$ problems, it looks like one should allow these piecewise polynomials to have relatively high degree on some polytopes, low degree on other polytopes, and vanish completely on yet further polytopes\footnote{In particular, the optimal choice $F$ for $M_{k,\eps}$ should vanish on the polytope $\{ (t_1,\dots,t_k) \in (1+\eps) \cdot {\mathcal R}_k: \sum_{i \neq i_0} t_i \ge 1-\eps \hbox{ for all } i_0=1,\dots,k\}$.}, but we do not have a systematic understanding of what the optimal placement of degrees should be.  
\item In Theorem \ref{epsilon-beyond}, the function $F$ was required to be supported in the simplex $\frac{k}{k-1} \cdot {\mathcal R}_k$.  However, one can consider functions $F$ supported in other regions $R$, subject to the constraint that all elements of the sumset $R+R$ lie in a region treatable by one of the cases of Theorem \ref{nonprime-asym}.  This could potentially lead to other optimization problems that lead to superior numerology, although again it appears difficult to perform efficient numerics for such problems in the medium $k$ regime $k \approx 50$.  One possibility would be to adopt a ``free boundary'' perspective, in which the support of $F$ is not fixed in advance, but is allowed to evolve by some iterative numerical scheme.
\item To improve the bounds on $H_m$ for $m=2,3,4,5$, one could seek a better lower bound on $M_k$ than the one provided by Theorem \ref{explicit}; one could also try to lower bound more complicated quantities such as $M_{k,\eps}$.
\item One could attempt to improve the range of $\varpi,\delta$ for which estimates of the form $\MPZ[\varpi,\delta]$ are known to hold, which would improve the results of Theorem \ref{main}(ii)-(vi).  For instance, we believe that the condition $600 \varpi + 180\delta < 7$ in Theorem \ref{mpz-poly} could be improved slightly to $1080 \varpi + 330 \delta < 13$ by refining the arguments in \cite{polymath8a}, but this requires a hypothesis of square root cancellation in a certain four-dimensional exponential sum over finite fields, which we have thus far been unable to establish rigorously.  Another direction to pursue would be to improve the $\delta$ parameter, or to otherwise relax the requirement of smoothness in the moduli, in order to reduce the need to pass to a truncation of the simplex ${\mathcal R}_k$, which is the primary reason why the $m=1$ results are currently unable to use the existing estimates of the form $\MPZ[\varpi,\delta]$.  Another speculative possibility is to seek $\MPZ[\varpi,\delta]$ type estimates which only control distribution for a positive proportion of smooth moduli, rather than for all moduli, and then to design a sieve $\nu$ adapted to just that proportion of moduli (cf. \cite{fouvry-invent}).  Finally, there may be a way to combine the arguments currently used to prove $\MPZ[\varpi,\delta]$ with the automorphic forms (or ``Kloostermania'') methods used to prove nontrivial equidistribution results with respect to a fixed modulus, although we do not have any ideas on how to actually achieve such a combination.
\item It is also possible that one could tighten the argument in Lemma \ref{crit}, for instance by establishing a non-trivial lower bound on the portion of the sum $\sum_n \nu(n)$ when $n+h_1,\dots,n+h_k$ are all composite, or a sufficiently strong upper bound on the pair correlations $\sum_n \theta(n+h_i) \theta(n+h_j)$ (see \cite[\S 6]{banks} for a recent implementation of this latter idea).  However, our preliminary attempts to exploit these adjustments suggested that the gain from the former idea would be exponentially small in $k$, whereas the gain from the latter would also be very slight (perhaps reducing $k$ by $O(1)$ in large $k$ regimes, e.g. $k \geq 5000$).
\item All of our sieves used are essentially of Selberg type, being the square of a divisor sum.  We have experimented with a number of non-Selberg type sieves (for instance trying to exploit the obvious positivity of $1 - \sum_{p \leq x: p|n} \frac{\log p}{\log x}$ when $n \leq x$), however none of these variants offered a numerical improvement over the Selberg sieve.  Indeed it appears that after optimizing the cutoff function $F$, the Selberg sieve is in some sense a ``local maximum'' in the space of non-negative sieve functions, and one would need a radically different sieve to obtain numerically superior results.
\item Our numerical bounds for the diameter $H(k)$ of the narrowest admissible $k$-tuple are known to be exact for $k \leq 342$, but there is scope for some slight improvement for larger values of $k$, which would lead to some improvements in the bounds on $H_m$ for $m=2,3,4,5$.  However, we believe that our bounds on $H_m$ are already fairly close (e.g. within $10\%$) of optimal, so there is only a limited amount of gain to be obtained solely from this component of the argument.
\end{enumerate}

\section{Narrow admissible tuples}\label{tuples-sec}

In this section we outline the methods used to obtain the numerical bounds on $H(k)$ given by Theorem~\ref{hk-bound}, which are reproduced below:
\smallskip

\begin{enumerate}
\item $H(3) = 6$,
\item $H(50) = 246$,
\item $H(51) = 252$,
\item $H(54) = 270$,
\item $H(\num{5511}) \leq \num{52116}$,
\item $H(\num{35410}) \leq \num{398130}$,
\item $H(\num{41588}) \leq \num{474266}$,
\item $H(\num{309661}) \leq \num{4137854}$,
\item $H(\num{1649821}) \leq \num{24797814}$,
\item $H(\num{75845707}) \leq \num{1431556072}$,
\item $H(\num{3473955908}) \leq \num{80550202480}$.
\end{enumerate}
\smallskip

\subsection{\texorpdfstring{$H(k)$}{H(k)} values for small \texorpdfstring{$k$}{k}}

The equalities in the first four bounds (1)-(4) were previously known.
The case $H(3)=6$ is obvious: the admissible 3-tuples $(0,2,6)$ and $(0,4,6)$ have diameter $6$ and no $3$-tuple of smaller diameter is admissible.
The cases $H(50)=246$, $H(51)=252$, and $H(54)=270$ follow from results of Clark and Jarvis~\cite{clark}.
They define $\varrho^*(x)$ to be the largest integer $k$ for which there exists an admissible $k$-tuple that lies in a half-open interval $(y,y+x]$ of length~$x$.
For each integer $k>1$, the largest $x$ for which $\varrho^*(x)=k$ is precisely $H(k+1)$.
Table~1 of \cite{clark} lists these largest $x$ values for $2\le k\le 170$, and we find that $H(50)=246$, $H(51)=252$, and $H(54)=270$.
Admissible tuples that realize these bounds are shown in Figures~\ref{k50tup}, ~\ref{k51tup} and~\ref{k54tup}.

\begin{figure}
\begin{align*}
&0,4,6,16,30,34,36,46,48,58,60,64,70,78,84,88,90,94,100,106,\\
&108,114,118,126,130,136,144,148,150,156,160,168,174,178,184,\\
&190,196,198,204,210,214,216,220,226,228,234,238,240,244,246.
\end{align*}
\caption{Admissible $50$-tuple realizing $H(50)=246$.}\label{k50tup}
\end{figure}

\begin{figure}
\begin{align*}
&0,6,10,12,22,36,40,42,52,54,64,66,70,76,84,90,94,96,100,106,\\
&112,114,120,124,132,136,142,150,154,156,162,166,174,180,184,\\
&190,196,202,204,210,216,220,222,226,232,234,240,244,246,250,252.
\end{align*}
\caption{Admissible $51$-tuple realizing $H(51)=252$.}\label{k51tup}
\end{figure}
\begin{figure}

\begin{align*}
&0,4,10,18,24,28,30,40,54,58,60,70,72,82,84,88,94,102,108,112,114,\\
&118,124,130,132,138,142,150,154,160,168,172,174,180,184,192,198,202,\\
&208,214,220,222,228,234,238,240,244,250,252,258,262,264,268,270.
\end{align*}
\caption{Admissible $54$-tuple realizing $H(54)=270$.}\label{k54tup}
\end{figure}

\subsection{\texorpdfstring{$H(k)$}{H(k)} bounds for mid-range \texorpdfstring{$k$}{k}}\label{secmidk}

As previously noted, exact values for $H(k)$ are known only for $k\le 342$.
The upper bounds on $H(k)$ for the five cases (5)-(9) were obtained by constructing admissible $k$-tuples using techniques developed during the first part of the Polymath8 project.
These are described in detail in Section~3 of \cite{polymath8a-unabridged}, but for the sake of completeness we summarize the most relevant methods here.

\subsubsection{Fast admissibility testing}
A key component of all our constructions is the ability to efficiently determine whether a given $k$-tuple $\mathcal{H}=(h_1,\ldots, h_k)$ is admissible.  We say that $\mathcal{H}$ is \emph{admissible modulo} $p$ if its elements do not form a complete set of residues modulo $p$.
Any $k$-tuple $\mathcal{H}$ is automatically admissible modulo all primes $p > k$, since a $k$-tuple cannot occupy more than $k$ residue classes; thus we only need to test admissibility modulo primes $p < k$.

A simple way to test admissibility modulo $p$ is to  enumerate the elements of $\mathcal{H}$ modulo $p$ and keep track of which residue classes have been encountered in a table with $p$ boolean-valued entries.
Assuming the elements of $\mathcal{H}$ have absolute value bounded by $O(k\log k)$ (true of all the tuples we consider), this approach yields a total bit-complexity of $O(k^2/\log k\ \textsf{M}(\log k))$, where $\textsf{M}(n)$ denotes the complexity of multiplying two $n$-bit integers, which, up to a constant factor, also bounds the complexity of division with remainder.
Applying the Sch\"onhage-Strassen bound $\textsf{M}(n)=O(n\log n\log\log n)$ from~\cite{schonhage}, this is $O(k^2\log\log k\log\log\log k)$, essentially quadratic in $k$.

This approach can be improved by observing that for most of the primes $p < k$ there are likely to be many unoccupied residue classes modulo $p$.  In order to verify admissibility at $p$ it is enough to find one of them, and we typically do not need to check them all in order to do so.  Using a heuristic model that assumes the elements of $\mathcal{H}$ are approximately equidistributed modulo $p$, one can determine a bound $m < p$ such that $k$ random elements of $\Z/p\Z$ are unlikely to occupy all of the residue classes in $[0,m]$.  By representing the $k$-tuple $\mathcal{H}$ as a boolean vector $\mathcal{B}=(b_0,\ldots,b_{h_k-h_1})$ in which $b_i=1$ if and only if $i=h_j-h_1$ for some $h_j\in \mathcal{H}$, we can efficiently test whether $\mathcal{H}$ occupies every residue class in $[0,m]$ by examining the entries
\[
b_0,\ldots,b_m,b_p,\ldots,b_{p+m},b_{2p},\ldots,b_{2p+m},\ldots
\]
of $\mathcal{B}$.
The key point is that when $p<k$ is large, say $p > (1+\epsilon)k/\log k$, we can choose~$m$ so that we only need to examine a small subset of the entries in $\mathcal{B}$.
Indeed, for primes $p > k/c$ (for any constant $c$), we can take $m=O(1)$ and only need to examine $O(\log k)$ elements of $\mathcal{B}$ (assuming its total size is $O(k\log k)$, which applies to all the tuples we consider here).

Of course it may happen that $\mathcal{H}$ occupies every residue class in $[0,m]$ modulo $p$.
In this case we revert to our original approach of enumerating the elements of $\mathcal{H}$ modulo $p$, but we expect this to happen for only a small proportion of the primes $p < k$.
Heuristically, this reduces the complexity of admissibility testing by a factor of $O(\log k)$, making it sub-quadratic.  In practice we find this approach to be much more efficient than the straight-forward method when $k$ is large.  See \cite[\S 3.1]{polymath8a} for further details.

\subsubsection{Sieving methods}

Our techniques for constructing admissible $k$-tuples all involve sieving an integer interval $[s,t]$ of residue classes modulo primes $p < k$ and then selecting an admissible $k$-tuple from the survivors.  There are various approaches one can take, depending on the choice of interval and the residue classes to sieve.
We list four of these below, starting with the classical sieve of Eratosthenes and proceeding to more modern variations.

\begin{itemize}
\item\textbf{Sieve of Eratosthenes}.
We sieve an interval $[2,x]$ to obtain admissible $k$-tuples
\[
p_{m+1},\ldots,p_{m+k}.
\]
with $m$ as small as possible.
If we sieve the residue class $0(p)$ for all primes $p\le k$ we have $m=\pi(k)$ and $p_{m+1}>k$. In this case no admissibility testing is required, since the residue class $0(p)$ is unoccupied for all $p \le k$.
Applying the Prime Number Theorem in the forms
\begin{align*}
p_k &=k \log k+ k \log \log k - k +O\Bigl(k \frac{ \log \log k}{\log k} \Bigr),\\
\pi(x)&=\frac{x}{\log x}+O\Bigl(\frac{x}{\log^2 x}\Bigr),
\end{align*}
this construction yields the upper bound
\begin{equation}\label{hk-eratosthenes}
H(k)\le k\log k + k\log\log k  - k + o(k).
\end{equation}
As an optimization, rather than sieving modulo every prime $p \le k$ we instead sieve modulo increasing primes $p$ and stop as soon as the first $k$ survivors form an admissible tuple.
This will typically happen for some $p_m < k$.
\smallbreak

\item\textbf{Hensley-Richards sieve}.
The bound in \eqref{hk-eratosthenes} was improved by Hensley and Richards~\cite{hensley,
  hensley-2, richards}, who observed that rather than sieving $[2,x]$ it is better to sieve the interval $[-x/2,x/2]$ to obtain admissible $k$-tuples of the form
\[
-p_{m+\lfloor k/2\rfloor-1},\ldots,p_{m+1},\ldots,-1,1,\ldots,p_{m+1},\ldots,p_{m+\lfloor(k+1)/2\rfloor-1},
\]
where we again wish to make $m$ as small as possible.
It follows from Lemma 5 of \cite{hensley-2} that one can take $m=o(k/\log k)$, leading to the improved upper bound
\begin{equation}\label{hk-hensely-richards}
H(k)\le k\log k + k\log\log k -(1+\log 2)k + o(k).
\end{equation}
\smallbreak


\item\textbf{Shifted Schinzel sieve}.
As noted by Schinzel in \cite{schinzel}, in the Hensley-Richards sieve it is slightly better to sieve $1(2)$ rather than $0(2)$; this leaves unsieved powers of~$2$ near the center of the interval $[-x/2,x/2]$ that would otherwise be removed (more generally, one can sieve $1(p)$ for many small primes $p$, but we did not).
Additionally, we find that shifting the interval $[-x/2,x/2]$ can yield significant improvements (one can also view this as changing the choices of residue classes).

This leads to the following approach: we sieve an interval $[s,s+x]$ of odd integers and multiples of odd primes $p\le p_m$, where $x$ is large enough to ensure at least $k$ survivors, and $m$ is large enough to ensure that the survivors form an admissible tuple, with $x$ and $m$ minimal subject to these constraints.  A tuple of exactly~$k$ survivors is then chosen to minimize the diameter.  By varying $s$ and comparing the results, we can choose a starting point $s\in [-x/2,x/2]$ that yields the smallest final diameter.
For large $k$ we typically find $s\approx k$ is optimal, as opposed to $s\approx -(k/2)\log k$ in the Hensley-Richards sieve.
\smallbreak

\item\textbf{Shifted greedy sieve}.
As a further optimization, we can allow greater freedom in the choice of residue class to sieve.
We begin as in the shifted Schinzel sieve, but for primes $p \le p_m$ that exceed $2\sqrt{k\log k}$, rather than sieving $0(p)$ we choose a minimally occupied residue class $a(p)$.
As above we sieve the interval $[s,s+x]$ for varying values of $s\in [-x/2,x/2]$ and select the best result, but unlike the shifted Schinzel sieve, for large $k$ we typically choose $s\approx -(k/\log k - k)/2$.

We remark that while one might suppose that it would be better to choose a minimally occupied residue class at all primes, not just the larger ones, we find that this is generally not the case.  Fixing a structured choice of residue classes for the small primes avoids the erratic behavior that can result from making greedy choices to soon (see \cite[Fig. 1]{gordon} for an illustration of this).

\end{itemize}
\smallbreak

Table~\ref{kmidtable} lists the bounds obtained by applying each of these techniques (in the online version of this paper, each table entry includes a link to the constructed tuple).
To the admissible tuples obtained using the shifted greedy sieve we additionally applied various local optimizations that are detailed in \cite[\S 3.6]{polymath8a}.
As can be seen in the table, the additional improvement due to these local optimizations is quite small compared to that gained by using better sieving algorithms, especially when $k$ is large.

Table~\ref{kmidtable} also lists the value $\lfloor k\log k+k\rfloor$ that we conjecture as an upper bound on $H(k)$ for all sufficiently large $k$.

\begin{table}
\centering
\setlength{\extrarowheight}{2pt}
\caption{Upper bounds on $H(k)$ for selected values of $k$.}\label{kmidtable}
\begin{tabular}{lrrrrr}
\toprule
$k$                 & \num{5511}         & \num{35410}          & \num{41588}           & \num{309661}           & \num{1649821}\\
\midrule
$k$ primes past $k$ & \tupref{5511}{56538} & \tupref{35410}{433992} & \tupref{41588}{516586} & \tupref{309661}{4505700} & \tupref{1649821}{26916060} \\
Eratosthenes        & \tupref{5511}{55160} & \tupref{35410}{424636} & \tupref{41588}{505734} & \tupref{309661}{4430212} & \tupref{1649821}{26540720} \\
Hensley-Richards           & \tupref{5511}{54480} & \tupref{35410}{415642} & \tupref{41588}{494866} & \tupref{309661}{4312612} & \tupref{1649821}{25841884} \\
Shifted Schinzel    & \tupref{5511}{53774} & \tupref{35410}{411060} & \tupref{41588}{489056} & \tupref{309661}{4261858} & \tupref{1649821}{25541910} \\
Shifted greedy      & \tupref{5511}{52296} & \tupref{35410}{399936} & \tupref{41588}{476028} & \tupref{309661}{4142780} & \tupref{1649821}{24798306} \\
Best known          & \tupref{5511}{52116} & \tupref{35410}{398130} & \tupref{41588}{474266} & \tupref{309661}{4137854} & \tupref{1649821}{24797814} \\
\midrule
$\lfloor k\log k + k \rfloor$       & \num{52985}         & \num{406320}          & \num{483899}          & \num{4224777}           & \num{25268951} \\
\bottomrule
\end{tabular}
\end{table}

\subsection{\texorpdfstring{$H(k)$}{H(k)} bounds for large \texorpdfstring{$k$}{k}}.
The upper bounds on $H(k)$ for the last two cases (10) and (11) were obtained using modified versions of the techniques described above that are better suited to handling very large values of $k$.
These entail three types of optimizations that are summarized in the subsections below.

\subsubsection{Improved time complexity}
As noted above, the complexity of admissibility testing is quasi-quadratic in $k$.  Each of the techniques listed in \S\ref{secmidk} involves optimizing over a parameter space whose size is at least quasi-linear in $k$, leading to an overall quasi-cubic time complexity for constructing a narrow admissible $k$-tuple; this makes it impractical to handle $k > 10^9$.
We can reduce this complexity in a number of ways.

First, we can combine parameter optimization and admissibility testing.
In both the sieve of Eratosthenes and Hensley-Richards sieves, taking $m=k$ guarantees an admissible $k$-tuple.
For $m<k$, if the corresponding $k$-tuple is inadmissible, it is typically because it is inadmissible modulo the smallest prime $p_{m+1}$ that appears in the tuple.
This suggests a heuristic approach in which we start with $m=k$, and then iteratively reduce~$m$, testing the admissibility of each $k$-tuple modulo $p_{m+1}$ as we go, until we can proceed no further.
We then verify that the last $k$-tuple that was admissible modulo $p_{m+1}$ is also admissible modulo all primes $p > p_{m+1}$ (we know it is admissible at all primes $p \le p_m$ because we have sieved a residue class for each of these primes).  We expect this to be the case, but if not we can increase $m$ as required.
Heuristically this yields a quasi-quadratic running time, and in practice it takes less time to find the minimal $m$ than it does to verify the admissibility of the resulting $k$-tuple.

Second, we can avoid a complete search of the parameter space.
In the case of the shifted Schinzel sieve, for example, we find empirically that taking $s=k$ typically yields an admissible $k$-tuple whose diameter is not much larger than that achieved by an optimal choice of $s$; we can then simply focus on optimizing~$m$ using the strategy described above.
Similar comments apply to the shifted greedy sieve.

\subsubsection{Improved space complexity}
We expect a narrow admissible $k$-tuple to have diameter $d=(1+o(1))k\log k$.
Whether we encode this tuple as a sequence of $k$ integers, or as a bitmap of $d+1$ bits, as in the fast admissibility testing algorithm, we will need approximately $k\log k$ bits.  For $k > 10^9$ this may be too large to conveniently fit in memory.  We can reduce the space to $O(k\log\log k)$ bits by encoding the $k$-tuple as a sequence of $k-1$ gaps; the average gap between consecutive entries has size $\log k$ and can be encoded in $O(\log\log k)$ bits.
In practical terms, for the sequences we constructed almost all gaps can be encoded using a single 8-bit byte for each gap.

One can further reduce space by partitioning the sieving interval into windows.
For the construction of our largest tuples, we used windows of size $O(\sqrt{d})$ and converted to a gap-sequence representation only after sieving at all primes up to an $O(\sqrt{d})$ bound.

\subsubsection{Parallelization}
With the exception of the greedy sieve, all the techniques described above are easily parallelized.
The greedy sieve is more difficult to parallelize because the choice of a minimally occupied residue class modulo $p$ depends on the set of survivors obtained after sieving modulo primes less than $p$.
To address this issue we modified the greedy approach to work with batches of consecutive primes of size $n$, where $n$ is a multiple of the number of parallel threads of execution.  After sieving fixed residue classes modulo all small primes $p < 2\sqrt{k\log k}$, we determine minimally occupied residue classes for the next $n$ primes in parallel, sieve these residue classes, and then proceed to the next batch of $n$ primes.

In addition to the techniques described above, we also considered a modified Schinzel sieve in which we check admissibility modulo each successive prime $p$ before sieving multiples of $p$, in order to verify that sieving modulo $p$ is actually necessary.  For values of $p$ close to but slightly less than $p_m$ it will often be the case that the set of survivors is already admissibile modulo $p$, even though it does contain multiples of $p$ (because some other residue class is unoccupied).
As with the greedy sieve, when using this approach we sieve residue classes in batches of size $n$ to facilitate parallelization.

\subsubsection{Results for large \texorpdfstring{$k$}{k}}

Table~\ref{klargetable} lists the bounds obtained for the two largest values of $k$.
For $k=\num{75845707}$ the best results were obtained with a shifted greedy sieve that was modified for parallel execution as described above, using the fixed shift parameter $s=-(k\log k-k)/2$.
A list of the sieved residue classes is available at
\begin{center}
 \url{math.mit.edu/~drew/greedy_75845707_1431556072.txt}.
\end{center}
This file contains values of $k$, $s$, $d$, and $m$, along with a list of prime indices $n_i> m$ and residue classes $r_i$ such that sieving the interval $[s,s+d]$ of odd integers, multiples of $p_n$ for $1<n \le m$, and at $r_i$ modulo $p_{n_i}$ yields an admissible $k$-tuple.

For $k=\num{3473955908}$ we did not attempt any form of greedy sieving due to practical limits on the time and computational resources available. The best results were obtained using a modified Schinzel sieve that avoids unnecessary sieving, as described above, using the fixed shift parameter $s=k0$.
A list of the sieved residue classes is available at
\begin{center}
 \url{math.mit.edu/~drew/schinzel_3473955908_80550202480.txt}.
\end{center}
This file contains values of $k$, $s$, $d$, and $m$, along with a list of prime indices $n_i> m$ such that sieving the interval $[s,s+d]$ of odd integers, multiples of $p_n$ for $1<n \le m$, and multiples of $p_{n_i}$ yields an admissible $k$-tuple.

Source code for our implementation is available at \url{math.mit.edu/~drew/ompadm_v0.5.tar}; this code can be used to verify the admissibility of both the tuples listed above.

\begin{table}
\centering
\setlength{\extrarowheight}{2pt}
\caption{Upper bounds on $H(k)$ for selected values of $k$.}\label{klargetable}
\begin{tabular}{lrr}
\toprule
$k$                 & \num{75845707}   & \num{3473955908}  \\
\midrule
$k$ primes past $k$ & \num{1541858666} & \num{84449123072} \\
Eratosthenes        & \num{1526698470} & \num{83833839848} \\
Hensley-Richards    & \num{1488227220} & \num{81912638914} \\
Shifted Schinzel    & \num{1467584468} & \num{80761835464} \\
Shifted Greedy      & \num{1431556072} & not available\\
Best known         & \num{1431556072} & \num{80550202480} \\
\midrule
$\lfloor k\log k + k \rfloor$       & \num{1452006268}         & \num{79791764059} \\
\bottomrule
\end{tabular}
\end{table}


\begin{backmatter}



\section*{Acknowledgements}

This paper is part of the \emph{Polymath project}, which was launched
by Timothy Gowers in February 2009 as an experiment to see if research
mathematics could be conducted by a massive online collaboration.
The current project (which was administered by Terence Tao) is the eighth
project in this series, and this is the second paper arising from that project, after \cite{polymath8a}.  Further information on the Polymath project can be
found on the web site \url{michaelnielsen.org/polymath1}.  Information
about this specific project may be found at
\begin{center}
\small{\url{michaelnielsen.org/polymath1/index.php?title=Bounded\_gaps\_between\_primes}}
\end{center}
and a full list of participants and their grant acknowledgments may be
found at
\begin{center}
\small{\url{michaelnielsen.org/polymath1/index.php?title=Polymath8\_grant\_acknowledgments}}
\end{center}
\par
We thank Thomas Engelsma for supplying us with his data on narrow admissible tuples, and Henryk Iwaniec for useful suggestions.  We also thank the anonymous referees for some suggestions in improving the content and exposition of the paper.

\end{backmatter}
\end{document}